\documentclass{compositio}
 \usepackage{mathtools} 
\usepackage{mathrsfs}
\usepackage{bm}

\usepackage{verbatim}
\usepackage{xcolor}
\usepackage{graphicx}
\usepackage{textcomp}
\usepackage[all]{xy}

 \usepackage{enumitem}
\usepackage{etoolbox}
\usepackage{needspace}
 
\usepackage{graphicx}
\usepackage[unicode, psdextra, colorlinks=true, linkcolor=black, citecolor=black]{hyperref}
\numberwithin{equation}{section}

\theoremstyle{plain}
\newtheorem{theorem}{Theorem}[section]
\newtheorem{proposition}{Proposition}[section]
\newtheorem{corollary}{Corollary}[section]
\newtheorem{lemma}{Lemma}[section]

\theoremstyle{definition}
\newtheorem{definition}{Definition}[section]

\theoremstyle{remark}
\newtheorem{rem}{Remark}[section]

\setlist[itemize]{label=\textbullet}
  \newcommand{\vol}{\mathop{\mathrm{vol}}}
\newcommand{\w}{ \widehat}
\newcommand{\ww}{\widetilde}

\newcommand{\MA}{Monge-Ampère}
\newcommand{\und}{\underline{\vol}}   
\newcommand{\ove}{\overline{\vol}}
\newcommand{\psh}{\text{PSH}}

  \allowdisplaybreaks[1]
\begin{document}
 \begin{sloppypar}

\title[Monge-Amp\`ere type  equations]{%
  {\fontsize{13.5pt}{16pt}\selectfont\rmfamily % 这里的 13.5pt 可以精确到小数点
    \scalebox{0.92}[1.0]{ MONGE-AMP\`ERE TYPE  EQUATIONS ON COMPACT HERMITIAN} \\ 
    MANIFOLDS WITH BOUNDED MASS PROPERTY
  }
}

\author{\textsc{Xuan Li}   }
  \email{ \textit{lixuanfd@gmail.com}}
%\address{}
  
%\author{}
%\email{}
%\address{}

%\footnotetext{Classification: 32U05, 32w20}

\classification[color=black]{32U05, 32U15, 32W20}

%\keywords{mor}

\begin{abstract}   
    In this paper, we study    possibly non-closed big  $(1,1)$-forms   on  a compact Hermitian  manifold  satisfying the bounded mass property.   We   propose several \mbox{criteria} for  the existence of  rooftop  envelopes.  As applications, we  establish the existence of  solutions to   complex Monge-Amp\`ere type equations with prescribed singularities, \mbox{allowing} for  non-pluripolar measures on the right-hand side. We also  obtain     stability results when  \mbox{singularity} types vary, by extending the Darvas-Di Nezza-Lu distance to   the \mbox{Hermitian} context.
\end{abstract}

\maketitle
 
%\tableofcontents

\section{Introduction }
\label{sec:introduction}
\hspace*{1.3em}  
Thanks to the resolution of the Calabi conjecture by Yau \cite{YAU78}, along with Bedford-Taylor's work on the Dirichlet problem \cite{bt76,BT82} and Ko{\l}odziej's $L^\infty$-estimates \cite{Kol98}, the complex Monge-Amp\`ere equation has since become a central theme in  complex analysis, complex geometry  and differential geometry.

  While   classical Bedford-Taylor pluripotential theory primarily focused on bounded plurisubharmonic  functions, its extension to compact Kähler manifolds was systematically  developed  in \cite{GZ07,GZ05}, allowing the complex Monge-Ampère operator to be defined on a much broader, possibly unbounded, class of potentials. These collective foundations then enabled Eyssidieux-Guedj-Zeriahi \cite{EGZ09} to extend the study of the  Monge-Amp\`ere  equations  to certain singular spaces.  %showing that the canonical model of a smooth projective variety of general type admits a Kähler-Einstein metric of negative Ricci curvature.
Over recent decades, 
a series of  further contributions, including \cite{BEGZ10, BBGZ_varma,BBEGZ_flow,BDL17,WNmono,Darvas_geom,DX21,darvas2020relative} and many others, have revealed the rich metric geometry of various potential spaces, which provide a natural framework for the variational approach to complex Monge-Ampère equations and  have significantly expanded the scope and applicability of pluripotential \mbox{methods} in complex geometry.

Let $(X, \omega_X)$ be a compact Hermitian manifold of dimension $n$. Fix another hermitian form $\omega$ on $X$, and let $\mu_X$ be a smooth volume form. The complex \MA\ equation   
     $$(\omega+dd^c\varphi)^n=c\mu_X$$
where $c>0$ is a constant and    $\varphi$ is a smooth $\omega$-psh function,    was first studied by  Cherri  \cite{Che87} and Hanani \cite{HANANI96}  under special curvature assumptions.  % Motivated by the  work of   Fu-Li-Yau   \cite{FLY12}, constructing special hermitian metrics on non-K\"ahler manifolds has \mbox{attracted} \mbox{considerable} interest. 
Following the    progress made by Guan-Li \cite{GL10}, %and Zhang \cite{zh10},     
Tosatti-Weinkove \cite{TW10} succeeded in fully extending Yau's theorem to    \mbox{general} Hermitian manifolds.   
It is  also natural to consider  Hermitian Monge-Amp\`ere equations with continuous weak solutions   for more general densities on the right-hand side. Such degenerate equations have been extensively investigated in \cite{blo_L2,DK12,KN15,Ngu16,KN19,KN_22domin,LPT,GP22}.  

If we merely assume that  the reference form $\omega$ is semipositive (and either big or of positive lower volume),  and  that  the right-hand side measure has  density in $L^p$  for some $p>1$,   Guedj-Lu  \cite{GL1}  developed a new pluripotential approach   to obtain  $L^{\infty}$ a priori  estimates. Higher regularity results  were also obtained in their work under suitable assumptions (see also \cite{Dang_herm} for  the nef and big case). In this context, one may further seek   solutions  with prescribed singularities, in which case  the equation is understood to hold via the non-pluripolar product as defined in \cite{BEGZ10,BGL25}. This direction was recently pursued by Alehyane-Lu-Salouf \cite{ALS25}, who obtained the existence of solutions with model singularities for the semipositive and big case.

 In the more general setting where the reference form is merely assumed to be big, solutions with minimal singularities, which are locally bounded on a Zariski open set, are established by Boucksom-Guedj-Lu \cite{BGL25} through an extension of the uniform $L^\infty$ a priori estimates.
 
 %for  \(L^p\) ($p>1$) right-hand sides.
 The major difficulty in this problem is that the Monge–Ampère volume can vary among potentials with the same singularity type; consequently, the comparison principle cannot be established as in the Kähler case. Let $\theta$ be a smooth real $(1,1)$-form and $\varphi$ a $\theta$-psh function.  In \cite{GL2,BGL25}, the authors defined the upper volume and the lower volume as 
 \[
\ove(\theta, \varphi) \;:=\; \sup_{\psi } \int_X (\theta+dd^c\psi)^n,
\qquad
\und (\theta,\varphi) \;:=\; \inf_{\psi } \int_X (\theta+dd^c\psi)^n,
\]
where the supremum and infimum are taken over all \(\theta\)-psh functions \(\psi\) having same singularity type as \(\varphi\) (i.e., \(\psi = \varphi + O(1)\)).  For $\theta=\omega_X$ and bounded $\varphi $,  we  simply write     $\ove(\omega_X)$ and  $ \und(\omega_X)$.   

We say that $X$ has the bounded
 mass property if $\ove(\omega_X)<+\infty$   and that $X$ has the positive
 volume property if $\und(\omega_X)>0$. 
   Both properties are  independent of the choice of the hermitian metric  and are  bimeromorphically invariant.    A natural and important question is whether   these volume conditions   hold. 
 Many examples have been provided in \cite{GL2,AGL23,BGL25}, \mbox{including} \mbox{Fujiki} manifolds, compact complex manifolds admitting a Guan-Li metric (i.e., satisfying $dd^c\omega_X=0$ and $d\omega_X\wedge d^c\omega_X=0$), and three-dimensional compact complex manifolds admitting a pluriclosed metric (i.e., $dd^c\omega_X=0$). However,  
    there is no known counterexample so far.    When these volume conditions are satisfied, the Demailly-P{\u{a}}un conjecture \cite[Conjecture 0.8]{DP04} can also be established in the hermitian setting  (see \cite{GL2,BGL25}): a nef form $\theta$ is big if and only if $\und(\{\theta\})>0$, where $\und(\{\theta\})\coloneqq \lim_{\varepsilon\to0} \und(\theta+\varepsilon\omega_X)$.
 
In what follows,  we will always assume that  the bounded mass property
as well as the positive volume property hold on $X$.  Let $\theta$ be a big form, i.e.,  there exists a $\theta$-psh function $\rho$ with analytic singularities   such that $\theta + dd^c   \rho  \geq \delta \omega_X$ for some  $\delta>0$.     Let  \(\phi \in \psh(X, \theta)\) be a model potential, meaning that   \(\phi = P_\theta[\phi]\) and $\und(\theta,\phi)>0$. 
   Associated with $\phi$ is the relative full mass class $\mathcal{E}(X, \theta, \phi)$, consisting of $\theta$-psh functions less singular than $\phi$ such that $\overline{\vol}(\theta, u) = \overline{\vol}(\theta, \phi)$.  This notion was originally introduced for closed forms in relative pluripotential theory in view of its central role in the variational approach to complex Monge-Ampère equations (see, e.g., \cite{DDNL18mono}). We now present our first main result.
  
\begin{theorem}[(Theorem \ref{MA_prescribed} and Theorem \ref{MA_pre_0})]\label{thm1.1}   
Assume   $\ove(\omega_X)<+\infty$ and $\und(\omega_X)>0$. Let $\theta$ be a big form and  $\phi$  a $\theta$-psh model potential. 
Let $\mu$ be a positive Radon measure which does not charge pluripolar sets. 
Then,   

\noindent{(i)} there exists a function \(\varphi \in \mathcal{E}(X, \theta,\phi)\) and a unique constant \(c>0\) such that  
\begin{align*}
(\theta+dd^c\varphi)^n = c\mu;
\end{align*}

\noindent{(ii)} for any \(\lambda > 0\), there exists a unique \(\varphi \in \mathcal{E}(X, \theta,\phi)\) such that
\begin{align*}
(\theta+dd^c\varphi)^n = e^{\lambda \varphi} \mu.
\end{align*}
\end{theorem} 
 \newpage
 In the case when \(X\) is K\"ahler and \(\theta\) is closed, 
Theorem~\ref{thm1.1} was established in \cite{GZ05} for K\"ahler classes (with \(\phi=0\)), 
extended in \cite{BEGZ10} to big classes (with \(\phi=V_{\theta}\)), 
and further generalized in \cite{DDNL18mono,DDNL21LOG} to the case of an arbitrary 
\(\theta\)-psh model potential.  

The uniqueness of solutions to the first equation is a subtle issue in the hermitian setting. A positive result is currently available   in the special case where \(\theta\) is a hermitian form, \(\phi=0\), and  \(\mu\) is a measure well dominated by capacity, absolutely continuous with respect to the volume form, whose \(L^1\)-density is strictly positive   
 (see \cite{KN19,KN_22domin}).   

To prove the theorem above, we begin with the case where $\phi$ has  minimal singularities and the measure  $\mu$ satisfies $\mu\leq A(\omega_X+dd^c\psi)^n$ for a constant $A>0$ and some  bounded $\omega_X$-psh function $\psi$. Following an idea of Guedj and Zeriahi \cite{GZ05} (which can be traced back to \cite{Ceg98} in the local setting), we approximate  
$\mu$  via local convolution.  One of the main ingredients  in the proof is to derive a uniform $(1+\varepsilon)$-energy estimate. In particular, one can characterize the range of the complex \MA\ operator on the finite energy space $\mathcal{E}^{p}(X,\theta,\phi)$ ($p \ge 1$), which is defined as the subset of all $\varphi \in \mathcal{E}(X,\theta,\phi)$ such that $(\varphi-\phi) \in L^p(X,\theta_\varphi^n)$ (see, e.g., \cite{darvas2020relative}).

\begin{theorem}  [(Theorem \ref{infinity_energy})] \label{thm_infnity} Assume   $\ove(\omega_X)<+\infty$ and $\und(\omega_X)>0$. Let $\theta$ be a big form and $\phi$  a $\theta$-psh model potential.    Let \(\mu\) be a positive Radon measure satisfying  $$
\mu(E) \le A [\operatorname{Cap}_{\omega_X}(E)]^a $$
for some constants \(a, A > 0\) and for every Borel set \(E\subset X  \). Let  $p\geq1$. If $a >\frac{np}{n+p}$, 
 then   there exists    a function   $\varphi\in \mathcal{E}^{p}(X,\theta,\phi)$ and a unique constant $c>0$ such that      $\theta_{\varphi}^n= c\mu.$      
 
 In particular, if $a\geq n$,   the   equation    admits a solution  $\varphi\in \bigcap_{p\geq1}\mathcal{E}^{p}(X,\theta,\phi)$. 
\end{theorem}

When \(\theta\) is Kähler, Guedj-Zeriahi \cite{GZ05} proves, via the comparison principle, that the solution \(\varphi \in \mathcal{E}^{p}(X,\theta)\) exists provided \(a > \frac{1}{1+p}\).   Moreover,  when \(\theta\) is a hermitian form and     \(a > 1\), it was shown in \cite{KN_22domin} (see   \cite{Kol98,KOL03,EGZ09} for the K\"ahler case) that the Monge-Amp\`ere equation admits a continuous solution.

Another difficulty lies in  deriving relative $L^{\infty}$-type estimates from   energy estimates. 
Let $\{\varphi_j\}_j$ be a sequence of relative full‑mass potentials that is uniformly bounded from above.
From the viewpoint of pluripotential theory,  showing that the rooftop envelope \(P_{\theta}(\inf_{j} \varphi_j)\) belongs to 
\(\mathcal{E}(X,\theta,\phi)\) provides a precise way to characterize the relative \(L^{\infty}\)-estimates for  unbounded potentials. Consequently, we obtain the following   result.

\begin{theorem}[(Theorem \ref{1-ener_thm})] \label{thm_1.2} 
    Assume   $\ove(\omega_X)<+\infty$   $\und(\omega_X)>0$. Let $\theta$ be a big form and $\phi$  a $\theta$-psh model potential.   Let  $\varphi_j,\varphi\in \mathcal{E}(X,\theta,\phi)$ be such that  $\varphi_j$ converges to $\varphi  $ in $L^1(X)$. 
 Assume that  $$ \lim_{j\to+\infty}  \int_X |\varphi_j-\varphi|\,\theta_{\varphi_j}^n =0.$$ 
Then, by passing to a subsequence,    $P_{\theta}(\inf_{l\geq j} \varphi_l )\in  \mathcal{E}(X,\theta,\phi)$. 
 In particular, $\varphi_{j}$ converges to $\varphi$ in capacity.  
\end{theorem}

When $X$ is Kähler and $\theta$ is closed, it is well known that the $d_1$-convergence of   finite-energy potentials implies     convergence in capacity. (see, e.g.,  \cite{BBGZ_varma,Dar_finite,BDL17,GTru_quasimo}). In the hermitian setting, the second statement of Theorem \ref{thm_1.2} was obtained in  \cite[Proposition 2.5]{KN22_weak}, assuming that $\varphi_j$ is uniformly bounded.

  We also require the following generalization of the Cegrell-Kołodziej-Xing stability theorem (see \cite{CK06,Xing,DH_con}), which can also be viewed as a variant of the domination principle.
\begin{proposition}[(Proposition \ref{general_asl})] \label{pro_1.1}
Assume   $\ove(\omega_X)<+\infty$ and $\und(\omega_X)>0$. Let $\theta$ be a big form and $\phi$  a $\theta$-psh model potential.
Let $\lambda>0$.
     Let \( u_j,\varphi \in \mathcal{E}(X, \theta, \phi) \) be such that \( \boldsymbol{1}_{D_j} \theta_{u_j}^n \leq \mu \), where  $D_j\coloneqq \{u_j\leq \varphi\}$ and  \( \mu \) is a   non-pluripolar Radon measure. Assume that  $u_j$ converges    to a \(\theta\)-psh function $ u$ in \( L^1(X) \). 
     
     Then,  by passing to a subsequence,   both $u$ and  $ P_{\theta}(\inf_{l\geq j} u_l) $ belong to $\mathcal{E}(X, \theta, \phi)$.
\end{proposition}

 The case $D_j=X$ corresponds to the stability theorem in the hermitian setting, which was addressed in \cite{ALS24}. For Kähler manifolds with $\theta$ being closed, quantitative estimates of related setting were also investigated  in \cite[Theorem 4.4]{DV_qua1} and \cite[Theorem 1.5]{DV_qua2}. %Notably, compared to \cite{ALS24}, the bounded mass property is required in the proof of this generalized domination principle. 
 
To establish Theorem \ref{thm1.1} for a general model potential $\phi$, we adopt a supersolution technique employed in \cite{ALS25} by first working with the   equations twisted by an exponential. This      avoids the mass comparison arguments traditionally required in the K\"{a}hler setting (see \cite{GLZ17, LN19, DDNL21LOG}). The strategy consists of building a supersolution from the solution in the full mass class and then obtaining the desired solution via the continuity method. The main technical step lies in the openness part, which reduces to the following subsolution theorem.

\begin{theorem}[(Theorem \ref{thm_sum})] \label{thm_1.3} Assume   $\ove(\omega_X)<+\infty$ and $\und(\omega_X)>0$. Let $\theta$ be a big form and $\phi$  a $\theta$-psh model potential.  
Assume $u,v\in\mathcal{E}(X,\theta,\phi)$. Fix $\lambda>0$ and  set  $\mu \coloneqq e^{-\lambda u}\theta_{u}^n+e^{-\lambda v}\theta_{v}^n$. Then there exists a unique $\varphi\in\mathcal{E}(X,\theta,\phi)$ such that $$(\theta+dd^c\varphi)^n=e^{\lambda\varphi}\mu.$$
\end{theorem}

  The special case where $u$ and $v$ have model singularities was studied in \cite[Section~2.5]{DDNLFULL} and  \cite[Theorem 5.2]{ALS25}. To establish this result in our  setting, we first address the case where $\theta_v^n \leq A(\omega_X + dd^c \psi)^n$ for a constant $A > 0$ and a bounded $\omega_X$-psh function $\psi$ by repeatedly employing the criteria proposed in  Section~\ref{general_rel}.  This is sufficient to derive Theorem~\ref{thm1.1} and thereby allows us to prove Theorem~\ref{thm_1.3} in full generality.

Note also that the \(L^{\infty}\) a priori estimates for measures with \(L^p\)-densities, as established in \cite[Theorem~2.3]{GL3}, remain valid in the relative setting. Combined with Theorem~\ref{thm1.1}, a further argument leads to the following corollary:
\begin{corollary}[(Corollary \ref{cor_6.1})]\label{cor_1.1} Assume   $\ove(\omega_X)<+\infty$ and $\und(\omega_X)>0$. Let $\theta$ be a big form and $\phi$  a $\theta$-psh model potential.
Let \(0\leq f \in L^{p}(X,\omega_X^n)\) with \(p > 1\)  and \(\int_X f\,\omega_X^n > 0\). Then

\noindent\textup{(i)} there exists a function \(\varphi \in \psh(X,\theta)\)    with  the same singularity type as $\phi$ and a unique constant \(c > 0\) such that   
 $$ 
(\theta+dd^c\varphi)^n = c f \omega_X^n; $$

\noindent\textup{(ii)} for any \(\lambda > 0\), there exists a unique \(\varphi \in \psh(X,\theta)\)   with the same singularity type as $\phi$  such that
$$ (\theta+dd^c\varphi)^n = e^{\lambda\varphi} f \omega_X^n.
 $$  
\end{corollary}
Given that the proof of Theorem~\ref{thm1.1} relies on the bounded mass property, a natural question that arises is whether this condition can be omitted in   Corollary~\ref{cor_1.1}, as considered  in \cite{GL3,BGL25,ALS25}.

 Beyond this,  to investigate the stability of solutions to \MA\ equations  when the prescribed singularity type varies, it is natural to introduce the following
distance-like   function:
\begin{align*}
d_{\theta}(u, v) := 2 \ove( \theta,\max(u,v)) -   \ove( \theta, u ) - \ove( \theta, v ),
\end{align*}
where   $u$  and $v$ are  $\theta$-psh functions.    
We denote by $\mathcal{S}(X,\theta)$    the space of singularity types of $\theta$-psh functions. Let $\delta>0$ be a constant. Set
$$\mathcal{S}_{\delta}(X,\theta):=\{u\in \mathcal{S}(X,\theta):  \, \und(\theta,u)>\delta\}.$$

  The   function $d_{\theta}$ is non-degenerate on the space of model potentials.  When $X$ is   Kähler and  $\theta$ is closed,  its restriction to $\mathcal{S}_{\delta}(X,\theta)$  is comparable to the one  introduced in \cite{DDNL21}.  Despite the lack of a (quasi-)triangle inequality in the hermitian setting, we can still establish:  

\begin{theorem}[(Theorem \ref{cauchy})]\label{thm_1.4} Assume   $\ove(\omega_X) < +\infty$ and $\und(\omega_X) > 0$. Let $\theta$ be a big form. Then for every $\delta > 0$, the space $\mathcal{S}_{\delta}(X, \theta)$ is complete. That is, if $\{u_j\}_j$ is a  $d_{\theta}$-Cauchy sequence in $\mathcal{S}_{\delta}(X, \theta)$  $($i.e., for every $\varepsilon > 0$ there exists $N\in\mathbb{N}$ such that $d_{\theta}(u_j, u_k) < \varepsilon$ for all $j, k > N$$)$, then there exists $u \in \mathcal{S}_{\delta}(X, \theta)$ such that $d_{\theta}(u_j, u) \to 0$.
\end{theorem}

\begin{theorem}[(Theorem \ref{model_capacity})] \label{thm_1.5} Assume   $\ove(\omega_X)<+\infty$ and $\und(\omega_X)>0$. Let $\theta$ be a big form. 
    Fix $\delta>0$ and  let $\phi_j, \phi \in \mathcal{S}_{\delta}(X,\theta)$ be model potentials.  Suppose that $d_{\theta}(\phi_j,\phi)\to0$ as $j\to+\infty$. Then, by passing to a sequence, there exist   a decreasing sequence $v_{j}\geq \phi_{j}$ and an increasing sequence $w_{j}\leq \phi_{j}$ such that $d_{\theta}(v_{j},\phi)\to0$ and $d_{\theta}(w_{j},\phi)\to0$. In particular,   $\phi_j$ converges to $ \phi$ in  capacity.
\end{theorem}
  This then allows us to generalize Theorem \ref{thm_1.2} and Proposition \ref{pro_1.1} to the case of varying prescribed singularities. As an application, we establish the following stability result for solutions to  Monge-Amp\`ere equations with an exponential twist.    For more related results, we refer to \cite{KN19, DDNL21, LPT, DV_qua2} and the references therein.

\begin{corollary} [(Corollary  \ref{cor_6.2})]\label{cor_1.2} Assume   $\ove(\omega_X)<+\infty$ and  $\und(\omega_X)>0$. 
Fix \(\delta,\lambda > 0\) and $p>1$. Let \(\phi_j, \phi \in \mathcal{S}_{\delta}(X,\theta)\) be model potentials such that \(d_{\theta}(\phi_j,\phi) \to 0\).   
Suppose $0 \leq f_j, f \in L^p(X, \omega_X^n)$ have uniformly bounded $L^p$-norms, satisfy $\int_X f_j \omega_X^n, \int_X f \omega_X^n > 0$, and $f_j\to f$ in $L^1$.  Let $\varphi_j, \varphi \in \psh(X, \theta)$ be such that $\varphi_j \simeq \phi_j$, $\varphi \simeq \phi$, and
\[
\theta_{\varphi_j}^n = e^{\lambda\varphi_j} f_j \omega_X^n,
\qquad
\theta_{\varphi}^n = e^{\lambda\varphi} f \omega_X^n.
\]
Then $\varphi_j$ converges   to $\varphi$ in capacity.
\end{corollary}
\noindent\textbf{Organization.}   In Section \ref{sec:preliminary}, we recall the definition of the non‑pluripolar product and its related  properties. In Section \ref{sec_full}, we extend the relative full‑mass class in relative pluripotential theory to the hermitian setting. In Section \ref{general_rel}, we establish several criteria that guarantee the existence of the rooftop envelope, including Theorem \ref{thm_1.2}, Proposition \ref{pro_1.1},    Theorem \ref{thm_1.4}, and Theorem \ref{thm_1.5}. In Section \ref{sec_5}, we study the complex Monge–Ampère equation in  the full mass class. Finally, in Section \ref{sec_6}, we proceed to prove Theorem \ref{thm1.1}, Theorem \ref{thm_infnity}  and Theorem \ref{thm_1.3}.

\noindent\textbf{Acknowledgments.} The author would like to thank  Prof. Jixiang Fu for his consistent guidance and encouragement. He also thanks  Mohammeds Alouf for a helpful  conversation.

\section{Preliminary}
\label{sec:preliminary}
Let $X$ be an $n$-dimensional compact complex manifold     endowed with a hermitian $(1,1)$-form $\omega_X$ and let  $\theta$ be a  (possibly non-closed) smooth real  $(1, 1)$-form on $X$.   \vspace{-0.9em}
\subsection{\textbf{Non-pluripolar products and quasi-psh envelopes}}
%\vspace{0.6em} \\  
%\textbf{2.1. Quasi-plurisubharmonic  functions.    } 
   A  function $u : X \to \mathbb{R} \cup \{-\infty\}$ is  called  quasi-psh if it can be locally written as a sum of a smooth function and a psh function. A quasi-psh function $u$ is called $\theta$-psh, denoted $u \in \mathrm{PSH}(X,\theta)$ if $\theta_u := \theta + dd^c u \geq 0$ in the weak sense of currents.

    The  Bott-Chern   space   $ {\text{BC}}^{p,q}(X)$ of $dd^c$-class of   smooth $(p,q)$-forms is defined as the cokernel of   
$dd^c : {\Omega}^{p-1,q-1}(X) \to {\Omega}^{p,q}(X)$, i.e., 
\[
  {\text{BC}}^{p,q}(X) := \frac{ {\Omega}^{p,q}(X) }
{dd^c{\Omega}^{p-1,q-1}(X)}.
\]

\begin{definition}\label{def_2.1} A class $\{\theta\}\in {\text{BC}}^{1,1}(X,\mathbb{R})$ is called \emph{big} if there exists a $\theta$-psh function with analytic singularities $\rho$ such that $\theta + dd^c   \rho  \geq \delta \omega_X$ for some $\delta>0$. For simplicity, we say that \(\theta\) itself is big. 
 \end{definition}      %   A class $\{\theta\}\in {\text{BC}}^{1,1}(X,\mathbb{R})$ is called \emph{nef} if  for every $\varepsilon>0$ there exists a smooth   function  $\rho_{\varepsilon}$ such that $\theta + dd^c   \rho_{\varepsilon}  \geq -\varepsilon \omega_X$. For simplicity, we say that \(\theta\) itself is nef.

 Let $\theta_1,\dots,\theta_n$ be   smooth real $(1,1)$-forms and  $u_j \in \text{PSH}(X,\theta_j)$ for each $j=1,\dots,n$. Following the construction of Bedford-Taylor   \cite{bt76,BT82}, we consider  the following     sequence of positive measures
\begin{equation*}
\boldsymbol{1}_{\bigcap_j\{u_j>-t\}} (\theta_{1}+dd^c\max(u_1, -t))\wedge \cdots\wedge (\theta_{n}+dd^c\max(u_n, -t)), 
\end{equation*}
       which is non-decreasing in $t\in\mathbb{R}$. If the total mass of this sequence is uniformly bounded, it converges to the so-called non-pluripolar product \cite{BEGZ10,BGL25}:  $$ (\theta_{1}+dd^c u_1 )\wedge \cdots\wedge (\theta_{n}+dd^c u_n ). $$
  The resulting positive measure does not charge any pluripolar sets.  In the particular case 
when $\theta_1   = \cdots = \theta_n = \theta$ and $ u_1 = \cdots = u_n = u$,
we denote the product by $\theta_u^n$ and refer to it as the  Monge-Ampère   measure of $u$. %\subsection{\textbf{Quasi-psh envelopes}} 
%\vspace{0.6em} \\  
%\textbf{2.2. Plurisubharmonic  envelopes.    }
Let \(u, v \in \mathrm{PSH}(X,\theta)\). We say that \(u\) is  more singular than \(v\), and write \(u \preceq v\), if there exists a constant \(C \in \mathbb{R}\) such that \(u \le v + C\). The equivalence relation \(u \simeq v\) means \(u \preceq v\) and \(v \preceq u\).  The least singular element   among all \(\theta\)-psh functions is given by 
\[
V_{\theta} \coloneqq \sup \{h \in \mathrm{PSH}(X,\theta) : h \le 0\}
\]
  Accordingly, a \(\theta\)-psh function \(u\) is said to have  minimal singularities  if \(u \simeq V_{\theta}\).
For a measurable function \(f: X \to [-\infty, \infty]\), we consider the \(\theta\)-psh envelope 
$$P_{\theta}(f):=   \text{usc}\Bigl(\sup\{h\in \text{PSH}(X,\theta)\, : \, h\leq  f \}\Bigr).$$  
 When   $f$ can be written as the difference of quasi-psh functions, the mean‑value inequalities imply that  $P_{\theta}(f)< +\infty$ and  that the upper‑semicontinuous regularization in its  definition is superfluous (see,   e.g., \cite[Theorem 9.17]{GZ}). In the particular case $f=\min\{u,v\}$,  we denote by  $  P_{\theta}(u,v)$    the rooftop envelope with respect to $u$ and $v$.

 We first recall the plurifine locality of $\theta$-psh functions, originally due to \cite{BT87}.  
\begin{proposition}
      If $u,v \in \mathrm{PSH}(X,\theta) $ and $\mathbf{1}_{\{u >v\}}\theta_{u}^n$ is a Radon measure,  then 
\[
  \mathbf{1}_{\{u > v\}} (\theta + dd^c u)^n = \mathbf{1}_{\{u > v\}}(\theta + dd^c \max(u,v))^n.
\]
\end{proposition}

 The following is an adaption of \cite[Theorem 2.7]{darvas2020relative} to the hermitian setting. 
\begin{proposition}[(cf. {\cite[Theorem 2.9]{BGL25}})] \label{contact_0}
Assume    $f $ is    \mbox{quasi-continuous}  and   $P_{\theta}(f) \in\psh(X,\theta)$. Then
\[
\int_{\{P_{\theta}(f)<f\}} (\theta + dd^c P_{\theta}(f))^n = 0.
\]
\end{proposition}
    
  We next rephrase the maximum principle as follows.  Assuming $\theta$ is big, the proof relies on an argument restricted to the open set $\{\rho > -\infty\}$, with $\rho$ as given in Definition \ref{def_2.1}. 
   \begin{proposition}[(cf. {\cite[Proposition 2.11]{BGL25}})] \label{max_p}
Assume that $\theta$ is big. If $u,v \in \mathrm{PSH}(X,\theta) $ and $\theta_{\max(u,v)}^n$ is a Radon measure,  then 
\[
  \mathbf{1}_{\{u \geq v\}} (\theta + dd^c u)^n + \mathbf{1}_{\{v > u\}} (\theta + dd^c v)^n \leq (\theta + dd^c \max(u,v))^n.
\]
In particular, if $u \leq v$ and $1_{\{u=v\}} (\theta+dd^c v)^n$ is a Radon measure,  then
\[
\mathbf{1}_{\{u = v\}} (\theta + dd^c u)^n \leq \mathbf{1}_{\{u = v\}} (\theta + dd^c v)^n.
\] 
\end{proposition}

 \begin{rem}
When the terms in Proposition \ref{max_p} are not guaranteed to be Radon measures, the inequalities are understood to hold locally after multiplying by the characteristic function
\(
\mathbf{1}_{\{\min(u,v) > V_{\theta} - t\} \cap \{\rho > -\infty\}}\),  \( t \in \mathbb{R}\).
\end{rem}

As an immediate consequence, we obtain the   minimum principle as below.
\begin{corollary} [(cf. {\cite[Corollary 2.12]{BGL25}})] \label{min_p}
  Assume that $\theta$ is big. Let $u,v \in \mathrm{PSH}(X,\theta)$ with $P_{\theta}(u,v) \not \equiv -\infty$. If the right‑hand side of the inequality below is a Radon measure, then
   \[
(\theta + dd^c P_\theta(u,v))^n \leq 
\mathbf{1}_{\{P_\theta(  u,v )=v \leq u\}} (\theta + dd^c v)^n + 
\mathbf{1}_{\{P_\theta( u,v )=u < v\}} (\theta + dd^c u)^n.
\]
\end{corollary} 
\subsection{\textbf{Upper and lower Monge-Ampère volumes}}
%\vspace{0.4em}  
%\noindent \textbf{2.3. \MA\ volumes.    }
Following \cite{BGL25}, we define the   upper volume  and   lower  volume  of a \(\theta\)-psh function \(u\)  by
\[ \overline{\mathrm{vol}}(\theta,u) \coloneqq \sup_{v} \int_X (\theta + dd^c v)^n
 , \qquad
  \underline{\mathrm{vol}}(\theta,u) \coloneqq \inf_{v } \int_X (\theta + dd^c v)^n,
\]
where the supremum and infimum are taken over all $\theta$-psh functions $v$ with  the same singularity type as $u$ and the integral  $\int_X (\theta + dd^c v)^n$ is understood as the limit
\[
\int_X (\theta + dd^c v)^n \coloneqq \lim_{t \to +\infty} \int_{\{v > -t\}} (\theta + dd^c v)^n .
\]
By definition,   \(0 \le \underline{\mathrm{vol}}(\theta,u) \le \overline{\mathrm{vol}}(\theta,u) \le +\infty\). We also set
\[
\underline{\mathrm{vol}}(\theta) \coloneqq \underline{\mathrm{vol}}(\theta,V_\theta), \qquad
\overline{\mathrm{vol}}(\theta) \coloneqq \overline{\mathrm{vol}}(\theta,V_\theta).
\]

With these concepts in place, we now  restate the monotonicity property of the upper   and  lower volumes in the following manner.
\begin{proposition}[(cf. {\cite[Proposition~3.7]{BGL25}})]\label{bgl_mono}
Let \(\theta,\eta\) be smooth \((1,1)\)-forms with \(\theta \le \eta\). Assume that \(\theta\) is big. If \(u \in \mathrm{PSH}(X,\theta)\) and \(v \in \mathrm{PSH}(X,\eta)\) satisfy \(u \preceq v\), then \(\overline{\mathrm{vol}}(\theta,u) \le \overline{\mathrm{vol}}(\eta,v)\). Furthermore, if \(\underline{\mathrm{vol}}(\theta,u) < +\infty\), then  \(\underline{\mathrm{vol}}(\theta,u) \le \underline{\mathrm{vol}}(\eta,v)\).
\end{proposition}

\begin{proof}
Assume first that \(\overline{\mathrm{vol}}(\eta,v) < \infty\). For any \(u' \in \mathrm{PSH}(X,\theta)\) with \(u' \simeq u\), the lower semicontinuity of the Monge-Ampère product (see   Lemma \ref{lsc} below) gives
\[
\int_X (\eta + dd^c u')^n
\le \liminf_{t \to +\infty} \int_X \bigl(\eta + dd^c \max(u',v-t)\bigr)^n
\le \overline{\mathrm{vol}}(\eta,v).
\]
Taking the supremum over all such \(u'\) yields the first inequality.

Now suppose \(\underline{\mathrm{vol}}(\theta,u) < +\infty\), we may  assume that \(\int_X \theta_u^n < +\infty\). Choose \(v' \in \mathrm{PSH}(X,\eta)\) with $v'\simeq v$ and  \(\int_X \eta_{v'}^n < +\infty\). Applying the minimum principle (Corollary~\ref{min_p}) we obtain
\[
\underline{\mathrm{vol}}(\theta,u)
\le \int_X \bigl(\theta + dd^c P_{\theta}(u,v'-t)\bigr)^n
\le \int_{\{u \le v'-t\}} (\theta+dd^cu)^n +\int_{\{u \ge v'-t\}} (\eta+dd^cv' )^n 
\]

Since \(\theta_u^n\)  does not charge pluripolar sets, letting \(t \to +\infty\)  yields that
\( 
\underline{\mathrm{vol}}(\theta,u) \le \int_X \eta_{v'}^n .
\)
Taking the infimum over all such \(v'\) then finishes the proof.
\end{proof}
The following result generalizes \cite[Proposition~3.2]{GL2} (see also \cite[Proposition~3.6]{BGL25}) to big classes:

\begin{proposition}\label{bgl_sum}
  Let \(\theta,\eta\) be smooth \((1,1)\)-forms. Assume that \(\theta\) is big. Let \(u \in \mathrm{PSH}(X,\theta)\) and \(v \in \mathrm{PSH}(X,\eta)\). Then
\[
\overline{\mathrm{vol}}(\theta,u) \le \overline{\mathrm{vol}}(\theta+\eta,u+v),
\quad \quad
\underline{\mathrm{vol}}(\theta,u) \le \underline{\mathrm{vol}}(\theta+\eta,u+v).
\]
\end{proposition}

\begin{proof}
The first inequality follows directly from the definition. For the second, assume  that \(\underline{\mathrm{vol}}(\theta+\eta,u+v) < +\infty\). Choose \(\varphi \in \mathrm{PSH}(X,\theta+\eta)\) with \(\varphi \simeq u+v\) and \(\int_X (\theta+\eta+dd^c\varphi)^n < +\infty\).  Then \(u' \coloneqq P_{\theta}(\varphi-v) \not\equiv -\infty\) and \(u' \simeq u\).
By Proposition~\ref{contact_0} and the minimum principle (Corollary~\ref{min_p}), we obtain
\[
\underline{\mathrm{vol}}(\theta,u)
\le \int_X (\theta+dd^c u')^n
\le \int_{\{u'+v = \varphi\}} (\theta+\eta+dd^c\varphi)^n
\le \int_X (\theta+\eta+dd^c\varphi)^n .
\]

Taking the infimum over all such \(\varphi\) completes the proof.
\end{proof}

\begin{rem}For any big form $\theta$, the above results imply that the bounded mass property
holds if and only if $\ove(\theta) < +\infty$. In this case, the positive volume
property is equivalent to $\und(\theta) > 0$.

Moreover, the proof also implies that \(\theta\) is non‑collapsing in the sense of \cite[Definition~1.4]{GL2}, i.e., \(\int_X \theta_u^n > 0\) for all $\theta$-psh functions $u$ with minimal singularities. This follows from the fact that  the envelope
$P_{\delta\omega_X}(u-\rho)$ is a bounded $\delta\omega_X$-psh function.
 \end{rem}

We also need the following extension of the lower semicontinuity property for non-pluripolar products.
\begin{lemma}[(cf. {\cite[Proposition 2.5]{BGL25} and \cite[Lemma 2.7]{ALS25}})] \label{lsc}
     Let $u_j,u$ be $\theta$-psh functions such that $u_j$ converges  in capacity to $u$. Let $\chi_j,\chi\geq0$ be quasi-continuous functions such that $\chi_j$ converges  in capacity to  $\chi$. Then 
     \begin{align*}
         \int_X \chi\, \theta_u^n \leq \liminf_{j\to+\infty}\int_X \chi_j\, \theta_{u_j}^n
     \end{align*}
Assume, in addition, that  $\sup_{j\in\mathbb{N}}\int_X\theta_{u_j}^n<+\infty$ and that  \(
\int_{\{u_j \leq -t\}} \theta_{u_j}^n
\) converges uniformly  in \( j \) to $0$ as $t\to+\infty$.  
 Then $\theta_{u_j}^n$ converges to $\theta_u^n$ weakly.     
\end{lemma}
\begin{proof} Fix $C>0$. 
    Set $\chi_C\coloneqq \min(\chi,C)$ and $\chi_{j,C}\coloneqq \min(\chi_j,C)$. By \cite[Proposition 2.5]{BGL25}, we have 
    \begin{align*}
         \int_X \chi_C\, \theta_u^n \leq \liminf_{j\to+\infty}\int_X \chi_{j,C}\, \theta_{u_j}^n \leq \liminf_{j\to+\infty}\int_X \chi_{j }\, \theta_{u_j}^n
     \end{align*}
     Letting $C\to+\infty$ on the left-hand side gives  the desired inequality. For the second statement, denote $$\chi_{j,t}^\varepsilon\coloneqq \frac{\max( u_j+t,0)}{\max( u_j+t,0)+\varepsilon}, \quad  \chi_{ t}^\varepsilon\coloneqq \frac{\max( u +t,0)}{\max( u +t,0)+\varepsilon}, \quad \varepsilon>0 .$$
  Using the plurifine locality and \cite[Theorem 4.26]{GZ}, we have
\[
\begin{aligned}
\int_X \theta_u^n
&\ge \int_X \chi_{t}^{\varepsilon} \, \theta_u^n
   = \int_X \chi_{t}^{\varepsilon} \, \theta_{\max(u,-t)}^n \\
&= \lim_{j\to+\infty} \int_X \chi_{j,t}^{\varepsilon} \, \theta_{\max(u_j,-t)}^n
   = \lim_{j\to+\infty} \int_X \chi_{j,t}^{\varepsilon} \, \theta_{u_j}^n \\
&\ge \limsup_{j\to+\infty} \int_{\{u_j \ge -t+1\}} \frac{1}{1+\varepsilon} \, \theta_{u_j}^n  \geq \frac{1}{1+\varepsilon} \limsup_{j\to+\infty} \int_X \theta_{u_j}^n + o(t).
\end{aligned}
\]
The penultimate inequality follows from the fact that 
\(\chi_{j,t}^{\varepsilon} \ge 1/(1+\varepsilon)\) on the set \(\{u_j \ge -t+1\}\), 
while the last inequality is a direct consequence of the hypothesis.  Letting \(\varepsilon \to 0\) and then \(t \to +\infty\) on the right‑hand side, we obtain 
\(
\int_X \theta_u^n \ge \limsup_{j\to+\infty} \int_X \theta_{u_j}^n.
\)
Together with the first part, we obtain the weak convergence \(\theta_{u_j}^n \to \theta_u^n\).
\end{proof}

%\begin{definition}
 %   We say that $X$ has the bounded
 %mass property if $\ove(\omega_X)<+\infty$   for some   hermitian form  $\omega_X$. \\
 %    We say that $X$ has the positive
   % volume property if $\und(\omega_X)>0$   for some   hermitian form  $\omega_X$.
%\end{definition}

%A question raised by Guedj–Lu \cite{GL2} asks whether one always has $\ove(\omega_X)<+\infty$ and $\und(\omega_X)>0$. Both conditions are bimeromorphic invariants (cf.\ \cite[Proposition~3.7]{GL2}) and are known to hold, for instance, when \(X\) belongs to the Fujiki class \(\mathcal{C}\) or satisfies the Guan--Li condition \[dd^c\omega_X = 0, \qquad d\omega_X \wedge d^c\omega_X = 0 .\] They also hold for three-dimensional compact complex manifolds admitting a pluriclosed metric (i.e., \(dd^c\omega_X = 0\)). Further examples of manifolds satisfying these properties can be found in \cite{AGL23}. So far, no counterexample has been found.

\section{The Relative Full Mass Classes  }\label{sec_full}
In the rest of this paper, we work under the assumptions that 
\(\overline{\mathrm{vol}}(\omega_X) < +\infty\) and \(\underline{\mathrm{vol}}(\omega_X) > 0\), 
and that \(\theta\) is a big form.

Given \(\theta\)-psh functions \(u\) and \(v\),   following  \cite{RWN14,DDNL18mono},  the envelope of $v$ relative to the singularity type of \(u\)   is defined by  
$$P_{\theta}[u](v) \coloneqq \mathrm{usc} \Bigl( \lim_{C \to +\infty} P_{\theta}(u + C, v) \Bigr) = \mathrm{usc} \left( \sup \big\{ h \in \mathrm{PSH}(X,\theta) : h \preceq u, \ h \le v \big\} \right).$$
When $v = V_{\theta}$, we simply  write $P_{\theta}[u]\coloneqq P_{\theta}[u](V_{\theta})$. While $P_{\theta}[u]$ may not have the same singularity type as $u$, the upper semicontinuity of Lelong numbers ensures that their Lelong numbers coincide over any modification.

\begin{definition}
 We say that $\phi\in\psh(X,\theta)$ is a model potential, if $\phi=P_{\theta}[\phi]$ and   $\und(\theta,\phi)>0$. A \(\theta\)-psh function \(u\) is said to have  model  singularities  if \(u \simeq \phi\) for some model potential $\phi$. \end{definition}
We now extend a fundamental lemma from \cite[Lemma~4.3]{DDNL21} to non‑closed forms
\begin{lemma}\label{lemma ddnl}
 Let \( u, \varphi \in \psh(X, \theta) \) with \( u \preceq \varphi \). Suppose
\[
\ove(\theta, u) > \ove(\theta, \varphi) - \und(\theta, \varphi)
\quad \text{and} \quad
\und(\theta, \varphi) > 0.
\]
Then, for any
\[
a \in \Bigl( 0,\, 1 - \bigl( \frac{ \ove(\theta, \varphi) - \ove(\theta, u) }{ \und(\theta, \varphi) } \bigr)^{\!1/n} \Bigr),
\]
there exists \( h \in \psh(X, \theta) \) such that \( a\varphi + (1-a)h \leq u \).
\end{lemma}
\begin{proof}
  Set \(b \coloneqq 1/(1-a) > 1\). We may assume \(\varphi \le 0\) and suppose, for a contradiction, that \(P_{\theta}\bigl(bu - (b-1)\varphi\bigr) \equiv -\infty\). Denote $$u_j:=\max(u,\varphi-j)\quad \text{and} \quad  \varphi_j\coloneqq P_{\theta}(b u_j-(b-1)\varphi).$$ Then we have \(\sup_X \varphi_j \searrow -\infty\).  Fix $j>k>0$. Let $ C_j:=\{\varphi_j=bu_j-(b-1)\varphi\}$ denote the contact set. Since $\{u>-k\} \subseteq\{u>\varphi-j\}$,  we can apply Proposition~\ref{contact_0}, the minimum principle (Corollary~\ref{min_p}) and plurifine locality to obtain 
\begin{align*}
  \int_{\{\varphi_j \le -bk\}} \theta_{\varphi_j}^n  
 \le b^n \int_{\{\varphi_j \le -bk\} \cap C_j} \theta_{u_j}^n
 \le b^n \int_{\{u \le -k\}} \theta_{u_j}^n  
&= b^n\Bigl( \int_X \theta_{u_j}^n - \int_{\{u > -k\}} \theta_u^n \Bigr).
\end{align*}
For $j$ sufficiently large we have $\{\varphi_j\leq-bk\}=X$, and therefore $\underline{\mathrm{vol}}(\theta,\varphi)
\leq \int_{\{\varphi_j \le -bk\}} \theta_{\varphi_j}^n$.  
    Letting \(j \to +\infty\) and then \(k \to +\infty\) yields
      \begin{align}\label{rang_dd}\und(\theta,\phi)\leq b^n \Bigl(\ove (\theta,\varphi)-\int_X\theta_{u}^n\Bigr)\end{align}
For any \(u' \in \mathrm{PSH}(X,\theta)\) with \(u' \simeq u\) we still have \(P_{\theta}\bigl(bu'-(b-1)\varphi\bigr) \equiv -\infty\). Hence \eqref{rang_dd} remains valid after replacing \(u\) by \(u'\). Taking the infimum over all such \(u'\) gives a contradiction.
\end{proof}

     As an immediate consequence of the preceding lemma, for any $\theta$-psh function  $u$ satisfying $\overline{\mathrm{vol}}(\theta, u) > \overline{\mathrm{vol}}(\theta) - \underline{\mathrm{vol}}(\theta)$, there exists a constant  \(\delta > 0\) and a $\theta$-psh function \(v  \) with \(v \preceq u\) such that \(\theta + dd^c v \ge \delta \omega_X\)  (see  \cite[Proposition 3.6]{DX21}). This property actually holds for general $u$.  The original argument is due to \cite[Proposition~2.11]{ALS25}, and we adapt it here.
 
\begin{proposition}\label{admit_herm}
Let \(u,\varphi \in \mathrm{PSH}(X,\theta)\) with \(u \preceq \varphi\).  If \(\underline{\mathrm{vol}}(\theta,u) > 0\), then there exists a constant \(a \in (0,1)\) and a function \(h \in \mathrm{PSH}(X,\theta)\) such that \(a\varphi + (1-a)h \le u\).
\end{proposition}

\begin{proof}
 We may assume that \(\varphi \le 0\) and suppose, by contradiction, that for every \(b = 1/(1-a) > 1\) we have \(P_{\theta}\bigl(bu - (b-1)\varphi\bigr) \equiv -\infty\). Set \[
u_j \coloneqq \max(u,\varphi-j), \,\,
v_j \coloneqq u + (b-1)u_j \,\, \text{and} \,\,\,\,
\varphi_j \coloneqq P_{\theta}\bigl(v_j - (b-1)\varphi\bigr).  
\] Then \(\sup_X \varphi_j \searrow -\infty\). Let \(C_j \coloneqq \{\varphi_j = v_j - (b-1)\varphi\}\).
Fix \(j>k>0\). For \(j\) sufficiently large, Proposition~\ref{contact_0} and the minimum principle then imply that 
\[
\underline{\mathrm{vol}}(\theta,u)
\le \int_{\{\varphi_j \le -bk\}} \theta_{\varphi_j}^n
\le \int_{\{\varphi_j \le -bk\} \cap C_j} (b\theta + dd^c v_j)^n
\le \int_{\{u \le -k\}} (b\theta + dd^c v_j)^n .
\]

Using the bounded mass property we obtain
\[
\int_{\{u \le -k\}} (b\theta + dd^c v_j)^n
\le \int_{\{u \le -k\}} \theta_u^n + O(b-1).
\]

Letting \(j, k \to +\infty\) and then \(b \to 1\) yields a contradiction.
\end{proof}

\begin{rem}
When the bounded mass property and the positive volume property  are not assumed, there may be no \(\phi \in \mathrm{PSH}(X, \theta)\) satisfying \(\underline{\mathrm{vol}}(\theta, \phi) > 0\). For semi‑positive forms, the authors in \cite{ALS25} therefore define the model potential by requiring \(P_{\varepsilon\theta}(\phi) \not\equiv -\infty\) for some \(\varepsilon > 0\), rather than \(\underline{\mathrm{vol}}(\theta, \phi) > 0\).
\end{rem}
   
\begin{definition}\label{def_fullmas}
 Let \(\varphi \in \operatorname{PSH}(X, \theta)\). The  relative full mass class  with respect to \(\varphi\) is defined by  
\[
\mathcal{E}(X, \theta, \varphi) \;:=\; 
\Bigl\{ 
u \in \operatorname{PSH}(X, \theta) \;:\; 
u \preceq \varphi \ \text{and}\ 
\overline{\operatorname{vol}}(\theta, u) = \overline{\operatorname{vol}}(\theta, \varphi) 
\Bigr\}.
\]
When \(\varphi = V_{\theta}\), we simply write  
\( 
\mathcal{E}(X, \theta) \coloneqq \mathcal{E}(X, \theta, V_{\theta}).
\)
\end{definition}

%\begin{proposition} The set 
 %  $\mathcal{E}(X, \theta, \varphi)$ is an extremal face of $\psh(X,\theta)$, i.e.,  for any $\theta$-psh functions $u,v$, $tu+(1-t)v\in\mathcal{E}(X, \theta, \varphi)$ for some $0<1$ implies $u,v\in\mathcal{E}(X, \theta, \varphi)$. 
   
  % In particular, $\mathcal{E}(X, \theta, \varphi)$ is a convex subset of $\psh(X,\theta)$.
  %  \end{proposition}
  % \begin{proof}
  % By the monotonicity, it suffice to consider the case where $u\simeq\varphi$ and $\ove(\theta,\varphi)>0$ and show that $v\in  \mathcal{E}(X, \theta, \varphi)$. 
 %  \end{proof}

\begin{corollary}\label{full_same}
    Let $\varphi\in\psh(X,\theta)$ and $u\in \mathcal{E}(X,\theta,\varphi)$. Then  $\und(\theta,u)=\und(\theta,\varphi)$. In addition,  if   $\und(\theta,\varphi)>0$, then $P_{\theta}[u]=P_{\theta}[\varphi]$.
\end{corollary}
\begin{proof} If $\und(\theta,\varphi)=0$, then $\und(\theta,u)=0$ follows from the monotonicity (Proposition \ref{bgl_mono}). Now assume    $\und(\theta,\varphi)>0$.
    According to Lemma \ref{lemma ddnl}, for any $a\in(0,1)$, there exists a $\theta$-psh function $h$ such that $a\varphi+(1-a)h\leq u$.   Proposition~\ref{bgl_mono} and Proposition~\ref{bgl_sum} then imply that  $$a^n \und(\theta,\varphi)\leq \und(\theta,u)\leq\und(\theta,\varphi)\,\,\, \text{and}\,\,\, aP_{\theta}[\varphi]+(1-a)P_{\theta}[h]\leq P_{\theta}[u]\leq P_{\theta}[\varphi]. $$
     Letting $a\to1$ completes this proof.
\end{proof}

 Next, we establish the following generalization of
 the domination principle (see   \cite[Proposition 2.8]{GL3}, \cite[Lemma 4.2]{BGL25}).

 \begin{proposition}   \label{domina_1}
Fix a constant $0 \leq c < 1$. Let $v, \varphi \in \psh(X,\theta)$ with $v\preceq \varphi$ and  let $ u\in \mathcal{E}(X,\theta,\varphi)$.  If   $\und(\theta,\varphi)>0$ and 
\[
\mathbf{1}_{\{u < v\}} \theta_u ^n \leq c \mathbf{1}_{\{u < v\}} \theta_v^n,
\]
then $u \geq v$.
\end{proposition} 

\begin{proof}
Replacing $v$ by $\max(u,v)$ and  using the plurifne-locality, we may assume   $  u\leq v$. For  $b > 1$, set 
$u_b := P_\theta(bu - (b - 1)v)$ and $D := \{u_b = bu - (b - 1)v\}$. Lemma~\ref{lemma ddnl} shows that  $u_b\in\psh(X,\theta)$    for all $b>1$. Moreover, we have $u_b\in\mathcal{E}(X,\theta,\varphi)$.  Indeed, for any fixed $b$, we have $u_b\preceq \varphi$, and for every $t\geq b$,  \begin{align}\label{u_b_full} u_b\geq (1-t^{-1}b)\varphi+t^{-1}bu_t.\end{align}
     Proposition \ref{bgl_mono} and Proposition~\ref{bgl_sum}   thus give    $\ove(\theta,u_b)=\ove(\theta,\varphi)$ and $\und(\theta,u_b)=\und(\theta,\varphi)>0$.
Applying Proposition \ref{contact_0} and the minimum principle, we obtain, for   sufficiently large $b$, 
\begin{align*}\label{refer_do}
\mathbf{1}_D b^{-n}  \theta_{u_b}^n + \mathbf{1}_D c  \theta_v^n  
 \leq \mathbf{1}_D b^{-n}  \theta_{u_b}^n + \mathbf{1}_D (1 - b^{-1})^n  \theta_v^n  
 \leq  \mathbf{1}_D \theta_u^n .
\end{align*}
By hypothesis we obtain  $\mathbf{1}_{D \cap \{u < v\}} \theta_{u_b}^n = 0$.   Thus we infer that the measure $\theta_{u_b}^n$       is supported on $D \cap \{u = v\}$.  Using again  Proposition \ref{contact_0} and the minimum principle, 
\[
  0<\und(\theta,\varphi)=\und(\theta,u_b) \leq \int_{D \cap \{u = v\}}\theta_{u_b}^n \leq \int_{ \{u_b = u\}}\theta_{u}^n
\]
 Now fix an arbitrary $\delta>0$. Assume by contradiction  that  $E := \{u < v - \delta\}$ is non-empty. Because $u,v$ are quasi-psh, the set $E$ has positive Lebesgue measure. Therefore, $\sup_X u_b \searrow -\infty$. Letting $b\to +\infty$ then gives $\int_{ \{u_b = u\}}\theta_{u}^n \to0$, contradicting the above inequality.
\end{proof}

\begin{corollary}\label{cor_domi} Let $u,v$ be as above.
    \begin{enumerate}
    \item If \( \theta_u^n \leq c\,\theta_v^n\), then \(c\geq1\).  
    \item If \(e^{-\lambda u}  \theta_u^n \le e^{-\lambda v}  \theta_v^n\) for some \(\lambda > 0\), then \(u\geq v\).
\end{enumerate}
\end{corollary}

\begin{lemma}\label{stable_lem}
    Let $u$ be a $\theta$-psh function such that $\und(\theta,P_{\theta}[u])>0$.  Then $\ove(\theta,u)=\ove(\theta,P_{\theta}[u])$.  
\end{lemma}
\begin{proof}
By Proposition \ref{bgl_mono}, we have $\ove(\theta,u)\leq\ove(\theta,P_{\theta}[u]).$ For the reverse inequality, take any $u'\in\psh(X,\theta)$ with $u'\simeq P_{\theta}[u]$.  For each $C>0$,  we set $$\varphi_{C}:=P_{\theta}(u+C,u'),\quad \varphi:=\text{usc}\Big( \lim_{C \to +\infty} \varphi_C \Big) =P_{\theta}[u](u') .$$
We have  $\varphi\simeq u$ and $\varphi\leq u'$. By Proposition \ref{contact_0} and the minimum principle, we obtain 
\begin{align}\label{apply_chi}  \boldsymbol{1}_{\{\varphi_C<u' \}}\theta_{\varphi_C}^n \leq  \boldsymbol{1}_{\{\varphi_C=u+C<u' \}}\theta_{u}^n.  \end{align} 
Applying the lower semicontinuity of the non-pluripolar product (Lemma~\ref{lsc}) 
with the   quasi-continuous functions   $$\chi_C\coloneqq \frac{\min(u'-\varphi_C,0)}{\min(u'-\varphi_C,0)+\varepsilon},\quad \chi \coloneqq \frac{\min(u'-\varphi ,0)}{\min(u'-\varphi ,0)+\varepsilon}, \quad  \varepsilon>0  $$  and using (\ref{apply_chi}), we obtain   $ \chi\theta_{\varphi }^n=0$.   Letting  $\varepsilon\to0$ then gives  $  \boldsymbol{1}_{\{\varphi <u' \}}\theta_{\varphi }^n=0$.  The domination principle (Proposition \ref{domina_1}) now implies $\varphi=u'$. Using the lower
 semicontinuity of   non-pluripolar masses again, we have $$ \int_X \theta_{u'}^n\leq   \liminf_{C \to +\infty} \int_X\theta_{\varphi_C}^n \leq \ove(\theta, \varphi_C)=\ove(\theta,u).    $$
Finally, taking the supremum over all such $u'$ gives $\ove(\theta,P_{\theta}[u])\leq\ove(\theta,u)$, 
which completes the proof.
\end{proof}

In general, the proof of the lemma above shows the following
\begin{proposition}\label{contact_2}
    Assume   $u,v$ and $P(u,v)$ are $\theta$-psh functions. Then 
    $ \theta_{P_{\theta}[u](v)}^n\leq \boldsymbol{1}_{\{P_{\theta}[u](v)=v\}} \theta_{v}^n$. In particular, $\theta_{P_{\theta}[u]}^n\leq \boldsymbol{1}_{\{P_{\theta}[u]=0\}} \theta_{V_{\theta}}^n.$ 
\end{proposition}
\begin{rem}
   This property is contained in \cite[Theorem 3.6]{darvas2020relative}.  
   When $X$ is K\"ahler and $\theta$ is closed, the $C^{1,\bar{1}}$-regularity of $V_{\theta}$ on the  locus $\{\rho>-\infty\}$ implies  the equality $\theta_{V_{\theta}}^n=\boldsymbol{1}_{\{V_{\theta}=0\}}\theta^n$ (see, e.g.,  ~\cite{BD12,DT23}).  In the hermitian context, the  $C^{1,\alpha}$-regularity of $V_{\theta}$   is obtained  in \cite{Dang_herm},  provided that \(\theta\) is nef and big.
\end{rem}

To establish the stability of the envelope operator,  consider the following class of potentials:
\[
\mathcal{F}(u) := \left\{ v \in \mathrm{PSH}(X, \theta) \ \middle| \ u \leq v \leq 0 \text{ and } \ove(\theta,u)=\ove(\theta,v) \right\}.
\]
\begin{proposition}\label{stable_prop}
    Assume  that $u\in\psh(X,\theta)$ and  $\und(\theta,P_{\theta}[u])>0$. Then $P_{\theta}[u]=\sup_{v\in\mathcal{F}(u)}v$.
    In particular, $P_{\theta}[u]=P_{\theta}[P_{\theta}[u]]$ and $P_{\theta}[u]$ is   a model potential. 
\end{proposition}
\begin{proof}
      Lemma \ref{stable_lem} implies that  $P_{\theta}[u]\in \mathcal{F}(u)$. It suffices to show $P_{\theta}[u]$ is the maximal element of $\mathcal{F}(u)$. Take any $v\in \mathcal{F}(u)$. By Proposition \ref{contact_2}, we have 
    $$\boldsymbol{1}_{\{P_{\theta}[u]<v\}}\theta_{ P_{\theta}[u]}^n \leq \boldsymbol{1}_{\{ P_{\theta}[u]<0\}}\theta_{ P_{\theta}[u]}^n=0.$$
    The domination principle (Proposition \ref{domina_1}) applied with $\varphi = P_{\theta}[v]$, then yields $P_{\theta}[v]\geq v$, completing the proof. 
\end{proof}

\begin{corollary}\label{lower_equal}
   Let $u\in\psh (X,\theta)$. Then  $\und(\theta,u)=\und(\theta,P_{\theta}[u])$.
\end{corollary}
\begin{proof}
  If $\und(\theta,P_{\theta}[u])=0$, the statement is clear. Otherwise, Proposition \ref{stable_prop} together with Lemma \ref{stable_lem} implies that 
\(P_{\theta}[u]\) is a model potential and \(u \in \mathcal{E}(X,\theta,P_{\theta}[u])\). 
Applying Corollary~\ref{full_same} then gives \(\und(\theta,u) = \und(\theta,P_{\theta}[u])\).
\end{proof}
\begin{rem}
 In general, equality for the upper volume is not expected to hold (cf.\ \cite[Remark 3.4]{BGL25} for a similar obstacle).
\end{rem}

 \begin{proposition}\label{dec_inc}
  For a  sequence $\varepsilon_j \searrow 0$   of nonnegative numbers,  set $\theta_j \coloneqq \theta + \varepsilon_j \omega_X$.
    Assume that $\varphi \in \psh(X,\theta)$ and that $\varphi_j \in \psh(X,\theta_j)$.
     \begin{enumerate}
      
         \item If $\theta_j=\theta$, $\und(\theta,\varphi)>0$ and $\varphi_j \nearrow \varphi$ a.e., then $P_{\theta}[\varphi_j] \nearrow P_{\theta}[\varphi]$ a.e.,   $\ove(\theta,\varphi_j)\nearrow \ove(\theta,\varphi)$ and $\und(\theta,\varphi_j)\nearrow \und(\theta,\varphi)$.
         \item If $\varphi_j=P_{\theta_j}[\varphi_j]$ and $\varphi_j\searrow\varphi$, then $\und(\theta_j,\varphi_j)\searrow \und(\theta,\varphi)$. In addition, if $\und(\theta,\varphi)>0$, then $\ove(\theta_j,\varphi_j)\searrow \ove(\theta,\varphi)$.
         \item  If $\und(\theta,\varphi)>0$, $\varphi_j\searrow\varphi$ and  $\ove(\theta_j,\varphi_j)\searrow\ove(\theta,\varphi)$,  then $P_{\theta_j}[\varphi_j] \searrow P_{\theta}[\varphi]$.
     \end{enumerate}
 \end{proposition}
 \begin{proof}
     \begin{enumerate}
         \item  Define $\varphi':= \text{usc}\Big( \lim_{j \to +\infty} P_{\theta}[\varphi_j]  \Big)$. Clearly     $\varphi\leq \varphi'\leq P_{\theta}[\varphi]$. 
         Using Proposition~\ref{contact_2},  
        we have 
         $$\boldsymbol{1}_{\{P_{\theta}[\varphi_j]<P_{\theta}[\varphi]\}}\theta_{P_{\theta}[\varphi_j]}^n\leq\boldsymbol{1}_{\{P_{\theta}[\varphi_j]<0\}}\theta_{P_{\theta}[\varphi_j]}^n=0.$$
      The lower semicontinuity of  non-pluripolar product  then   yields that $\boldsymbol{1}_{\{\varphi'<P_{\theta}[\varphi]\}}\theta_{\varphi'}^n=0$.  Therefore, by the domination principle (Proposition~\ref{domina_1}), we obtain \(\varphi' = P_{\theta}[\varphi]\).  
 It then follows that $P_{\theta}[\varphi_j] \nearrow P_{\theta}[\varphi]$ almost everywhere.
 
 \hspace{2em}  Next, we  take any $\psi\in\psh(X,\theta)$ such that  $\psi\simeq P_{\theta}[\varphi]$ and set $$\psi':=\text{usc} \bigl( \lim_{j \to +\infty} P_{\theta}[\varphi_j](\psi) \bigr) .$$ 
Then we have $\psi'\simeq P_{\theta}[\varphi]\simeq\psi$ and   $\psi'\leq\psi$. By a similar argument as above,   we infer that $\psi'=\psi$. Now, the   lower
semicontinuity of  non-pluripolar masses  gives 
\[\int_X\theta_{\psi}^n \leq   \liminf_{j \to +\infty} \int_X\theta_{P_{\theta}[\varphi_j](\psi)}^n \leq \liminf_{j \to +\infty}\ove(\theta, P_{\theta}[\varphi_j]).    \]
  Taking the supremum over  all such $\psi$ and using  Proposition \ref{bgl_mono}, we obtain \begin{align*}
   \ove(\theta, P_{\theta}[\varphi])=\lim_{j \to +\infty}\ove(\theta, P_{\theta}[\varphi_j])\end{align*}
 Then we  apply  Lemma \ref{lemma ddnl}  to $P_{\theta}[\varphi_j]$ and $ P_{\theta}[\varphi]$. This provides $a_j \nearrow 1$ and $h_j \in \psh(X,\theta)$ such that    $a_jP_{\theta}[\varphi]+(1-a_j)h_j\leq P_{\theta}[\varphi_j].$ 
 From this inequality, together with Proposition \ref{bgl_mono}, Proposition \ref{bgl_sum} 
and Corollary~\ref{lower_equal}, we obtain    $$ \und(\theta, \varphi)=\und(\theta, P_{\theta}[\varphi])=\lim_{j \to +\infty}\und(\theta, P_{\theta}[\varphi_j])=\lim_{j \to +\infty}\und(\theta,  \varphi_j )$$      Now, we have $ \und(\theta,  \varphi_j )>0$ for  sufficiently large $j$. 
Lemma \ref{stable_lem} then implies that 
$$  \ove(\theta, P_{\theta}[\varphi_j])=\ove(\theta, \varphi_j), \,\quad \,\,  \ove(\theta, P_{\theta}[\varphi])=\ove(\theta, \varphi ).$$ This completes the proof of first statement.
 \item    
 \textbf{Step 1.} We first   consider the case where   $\varphi_j=  P_{\theta_j}[\varphi]$.  Then we  have \begin{align*}  \und(\theta,  \varphi ) \leq\und(\theta_j,  \varphi  )\leq \und(\theta,  \varphi )+O(\varepsilon_j).\end{align*}  
where in the last inequality we have used \cite[Proposition 3.3(ii)]{BGL25}. 
     Therefore   $$\und(\theta_j, P_{\theta_j}[\varphi ]) = \und(\theta_j, \varphi)\searrow \und(\theta,\varphi)$$ 
     
     \hspace{2em} If, in addition, $\und(\theta, \varphi) > 0$, then Proposition \ref{admit_herm} provides us a hermitian current  $\theta + dd^c v \geq \delta \omega_X $ for some $\delta>0$, with $v \in \psh(X, \theta)$ satisfying $v \preceq \varphi$.   Observe also that  $\und(\theta_j,  \varphi  )>0$. Hence, by Lemma \ref{stable_lem}, we see that  $\ove(\theta_j, P_{\theta_j}[\varphi ])=\ove(\theta_j,  \varphi  )$. Applying   Proposition \ref{bgl_mono} and Proposition~\ref{bgl_sum} again, we obtain
$$  \ove(\theta,  \varphi  )\leq\ove(\theta_j,  \varphi  )\leq \ove( (1+\frac{\varepsilon_j}{\delta})\theta),  \varphi + \frac{\varepsilon_j}{\delta} v )\leq (1+\frac{\varepsilon_j}{\delta})^n \ove(  \theta,  \varphi  ),  $$
By letting $j\to+\infty$, we conclude that   $  \ove(\theta_j,  \varphi  ) \searrow \ove(\theta,  \varphi  )$.

 \textbf{Step 2.}  By Dini's convergence theorem, the previous step reduces the problem to proving that the result holds for the monotone sequence
\( 
P_{\theta_{j_0}}[\varphi_j] \searrow P_{\theta_{j_0}}[\varphi]
\)
for each fixed $j_0$. It therefore suffices to treat the case $\varepsilon_j = 0$ and $\und(\theta, \varphi) > 0$.

\hspace{2em}    Since $\ove(\theta,\varphi_j)$ is a Cauchy sequence, we see by Lemma \ref{lemma ddnl} 
that there exist   $a_k \nearrow 1$ and $h_{k,j} \in \psh(X,\theta)$  such that $$  a_k\varphi_k+(1-a_k)P_{\theta}[h_{k,j}]\leq \varphi_j $$ 
  for every $j>k$. Here we use the hypothesis that $\varphi_j=P_{\theta}[\varphi_j]$. After passing to a subsequence, we may assume that $P_{\theta}[h_{k,j}]$ converges in $L^1(X)$ to some
$h_k \in \psh(X,\theta)$ as $j \to +\infty$. We thus have $$  a_k\varphi_k+(1-a_k)h_k\leq \varphi.  $$
  Applying  Proposition \ref{bgl_mono} and Proposition \ref{bgl_sum}  to this inequality then yields  
\[
\ove(\theta, \varphi_j) \searrow \ove(\theta, \varphi),
\qquad
\und(\theta, \varphi_j) \searrow \und(\theta, \varphi).
\]
 \item  Define $\psi:=\lim_{j\to+\infty}P_{\theta_j}[\varphi_j]$. It follows that $\psi\geq P_{\theta}[\varphi]$. By part (ii) we have   $$\ove(\theta,P_{\theta}[\varphi])=\lim_{j\to+\infty} \ove(\theta_j,P_{\theta_j}[\varphi_j])=  \ove(\theta,\psi) $$
 Proposition \ref{stable_prop} then implies $\psi = P_{\theta}[\varphi]$, which completes the proof.
 \end{enumerate}
 \end{proof}

 \section{Existence of the Rooftop Envelope}\label{general_rel}

\subsection{\textbf{Several useful criteria}}\label{critea}

The following lemma generalizes \cite[Lemma 5.1]{DDNL21} to non‑closed   forms.
\begin{lemma}\label{diam_lem} Let   $p\in \mathbb{N}$. 
Assume that $u_1,\ldots, u_p, \varphi \in \mathrm{PSH}(X, \theta)$  and that $
\max(u_1,\ldots,u_p ) \preceq w$. Set  $\Delta(\theta,w)=\ove(\theta,w)-\und(\theta,w)$.   If  
\[
    \sum_{i=1}^p\ove(\theta,u_i) > (p-1)\ove(\theta,w) +\Delta(\theta,w), 
\]
then $P_{\theta}(u_1,\ldots,u_p) \in \mathrm{PSH}(X, \theta)$. In particular, let $\phi$ be a $\theta$-psh model potential. If $u_i\in\mathcal{E}(X,\theta,\phi)$, then $P_{\theta}(u_1,\ldots,u_p)\in \mathcal{E}(X,\theta,\phi)$.
\end{lemma}
\begin{proof}
    Assume by contradiction that $P_{\theta}(u_1,\ldots,u_p)\equiv-\infty$. Let $u_{ik}:=\max(u_i,w-k)$ and $\varphi_k:=P_{\theta}(u_{1k},\ldots,u_{pk})$.   Fix $k>s>0$. For $k$ sufficiently large, we have $\sup_X\varphi_k\leq-s$.  By the minimum  principle,  we obtain
    \begin{align*}
      \und(\theta,w)  \leq \int_{\{\varphi_k\leq-s\}}\theta_{\varphi_k}^n \leq \sum_{i=1}^p \int_{\{u_{ik}\leq-s\}}  \theta_{u_{ik}}^n&=\sum_{i=1}^p \Bigl(\int_{X}  \theta_{u_{ik}}^n-  \int_{\{u_i>-s\}} \theta_{u_i}^n\Bigr)\\& \leq p\,\ove(\theta,w)- \sum_{i=1}^p\int_{\{u_i>-s\}} \theta_{u_i}^n
    \end{align*}
    Letting $k\to+\infty$ and then $s\to+\infty$, we obtain  $\sum_{i=1}^p\int_{X} \theta_{u_i}^n\leq (p-1)\ove(\theta,w) +\Delta(\theta,w)$. Now take any $u'_i\in\psh(X,\theta)$ with $u'_i\sim u_i$.    By assumption, we have $P_{\theta}(u_1',\ldots,u_p')\equiv-\infty$. Hence the above estimate remains valid after replacing each $u_i$ by $u_i'$. Taking the supremum over all such $u'_i$  leads to a contradiction with the  hypothesis. 
    In the particular case, we take $w=\phi$, then \begin{align*}\sum_{i=1}^p\ove(\theta,u_i) - (p-1)\ove(\theta,w) -\Delta(\theta,w)=\und(\theta,\phi)>0,\end{align*} which implies   $P_{\theta}(u_1,\ldots,u_p) \in \mathrm{PSH}(X, \theta)$. Moreover, Lemma \ref{lemma ddnl} also gives us that there exist $a_k\nearrow1$ and $h_{ik}\in\mathcal{E}(X,\theta,\phi)$ (this follows from inequality (\ref{u_b_full})) satisfying $a_k\phi+(1-a_k)h_{ik}\leq u_i$. Observing that $P_{\theta}(h_{1k},\ldots,h_{pk}) \in \mathrm{PSH}(X, \theta)$.  Therefore  we have \begin{align*}
        a_k\phi+(1-a_k)P_{\theta}(h_{1k},\ldots,h_{pk})\leq P_{\theta}(u_1,\ldots,u_p).
    \end{align*}
   Let $a_k \nearrow 1$. Then by Proposition \ref{bgl_mono} and Proposition \ref{bgl_sum}, 
we obtain $\ove(\theta,P_{\theta}(u_1,\ldots,u_p)) = \ove(\theta,\phi)$, completing the proof.
\end{proof}

We also have the following volume diamond inequality, extending \cite[Theorem~5.4]{DDNL21}. 
\begin{proposition}\label{diam_ineq}
   Let $p\in\mathbb{N}$. Assume  $u_i \in \mathrm{PSH}(X, \theta)$ are such that $u_i = P_{\theta}[u_i]$ for each $i\in \{1,\ldots,p\}$.    Set $ w:=
\max(u_1,\ldots,u_p ) $ and $\Delta(\theta,w)=\ove(\theta,w)-\und(\theta,w)$.  If $P_{\theta}(u_1,\ldots,u_p)\in\psh(X,\theta)$, then 
\begin{align*}
    \sum_{i=1}^p\ove(\theta,u_i) \leq (p-1)\ove(\theta,w) +\Delta(\theta,w) + \ove(\theta,P_{\theta}(u_1,\ldots,u_p)), 
\end{align*}

\end{proposition}
\begin{proof}
 For smooth functions $f_i \in C^{\infty}(X)$ with $i=1,\ldots,p$, set $u_i' \coloneqq P_{\theta}[u_i](f_i)$, which satisfies $u_i' \simeq u_i$. Note that any $\psi_i \in \mathrm{PSH}(X,\theta)$ with $\psi_i \simeq u_i$ admits a decreasing approximation by such potentials $u_i'$. By the lower semicontinuity of non-pluripolar masses, we have 
$$ \overline{\mathrm{vol}}(\theta,u_i) = \sup_{f_i \in C^{\infty}(X)} \int_X (\theta + dd^c u'_i)^n. $$

Denote $M := \max\limits_{1 \le i \le p} \sup\limits_{X} |f_i|$.    Set  $u_{ik}:=\max(u_i',w-k)$ and $\varphi_k:=P_{\theta}(u_{1k},\ldots,u_{pk})$. By  plurifine locality and Proposition \ref{contact_2},   we have  
\begin{align*}
\theta_{u_{ik}}^n 
&= \mathbf{1}_{\{u_i' > w - k\}} \theta_{u_i'}^n + \mathbf{1}_{\{u_i' \leq w - k\}} \theta_{u_{ik}}^n \\
&= \mathbf{1}_{\{f_i = u_i' > w - k\}} \theta_{u_i'}^n + \mathbf{1}_{\{u_i' \leq w - k\}} \theta_{u_{ik}}^n \\
&= \theta_{u_i'}^n + \mathbf{1}_{\{u_i' \leq w - k\}} \theta_{u_{ik}}^n \quad \text{for } k > M.
\end{align*}
 Applying  the minimum  principle,   we obtain
\begin{align*}
\theta_{\varphi_k}^n 
 &\leq \sum_{i=1}^p\mathbf{1}_{\{\varphi_k=u_{ik}\}} \theta_{u_{ik}}^n   
\leq \sum_{i=1}^p \mathbf{1}_{\{\varphi_k = u_{ik} = f_i\}} \theta_{u_i'}^n  
  + \sum_{i=1}^p \mathbf{1}_{\{u_i' \leq w - k\}} \theta_{u_{ik}}^n
\end{align*}
 Since $\varphi_k\leq \min_i f_i$, we infer that  $\theta_{\varphi_k}^n$ is supported on the disjoint union of sets $\{\varphi_k=\min_i f_i\}$ and $\cup_i \{u_i'\leq w-k\}$. Therefore 
 \begin{align}\label{ineq_dia1}
     \und(\theta,w)\leq \int_X\theta_{\varphi_k}^n& \leq\int_{\{\varphi_k= \min_i f_i\}} \theta_{\varphi_k}^n +\sum_{i=1}^p \int_{\{u_i' \leq w - k\}} \theta_{u_{ik}}^n \nonumber\\&      \leq \int_{\{\varphi_k=\min_i f_i\}} \theta_{\varphi_k}^n+p\,\ove(\theta,w)- \sum_{i=1}^p\int_{\{u_i'>w-k\}} \theta_{u_i'}^n  
 \end{align}
  Set  $\varphi \coloneqq P_{\theta}(u_{1},\ldots,u_{p})$,  $\psi_{k}:=\max(\varphi_k, \min_i f_i-1)$, and  $\psi\coloneqq\max(\varphi, \min_i f_i-1)$. 
  Fix   $\varepsilon>0$, and define
\begin{align*}
      \chi_{k }^\varepsilon\coloneqq \frac{\max( \varphi_k-\min_i f_i+1,0)}{\max(\varphi_k-\min_i f_i+1,0)+\varepsilon}, \quad  \chi_{}^\varepsilon\coloneqq \frac{\max(\varphi -\min_i f_i+1,0)}{\max(\varphi-\min_i f_i+1,0)+\varepsilon}. 
\end{align*}
Note that each \(\psi_k\) is a bounded quasi-psh function and \(\psi_k \searrow \psi\) as \(k \to +\infty\).
  By    plurifine locality and  \cite[Theorem 4.26]{GZ}, we have 
 \begin{align*}   \limsup_{k\to+\infty} \int_{\{\varphi_k=\min_i f_i\}}\frac{1}{1+\varepsilon} \theta_{\varphi_k}^n &\leq   \limsup_{k\to+\infty}\int_{X}  \chi_k^{\varepsilon} \, \theta_{\varphi_k}^n =\limsup_{k\to+\infty}\int_{X}  \chi_k^{\varepsilon}  \,\theta_{\psi_k}^n \\& = \int_{X}  \chi^{\varepsilon} \, \theta_{\psi}^n=  \int_{X}  \chi^{\varepsilon}  \,\theta_{\varphi}^n \leq \int_{X} \theta_{\varphi}^n   .\end{align*}
Letting   $\varepsilon\to0$  yields \(\limsup\limits_{k\to+\infty} \int_{\{\varphi_k=\min_i f_i\}} \theta_{\varphi_k}^n \le \ove(\theta,\varphi).\)
Consequently, by letting $k \to +\infty$ in \eqref{ineq_dia1}, we conclude   that   $$\sum_{i=1}^p\int_{X} \theta_{u_i'}^n\leq (p-1)\ove(\theta,w) +\Delta(\theta,w)+\ove(\theta,\varphi)$$ The desired inequality then follows by taking the supremum over all   $f_i \in C^{\infty}(X)$.
\end{proof}

%Fix $s>0$ and $C>M$. Let  $\varphi :=P_{\theta}(u_{1},\ldots,u_{p})$, $\psi\coloneqq P_{\theta}(\min_i f_i) $  $\varphi_{k}^C:=\max(\varphi_k, \psi-C)$ and $\varphi^C:=\max(\varphi ,\psi-C)$. By  the plurifine locality and \cite[Theorem 4.26]{GZ}, we have \begin{align*}   \limsup_{k\to\infty} \int_{\{\varphi_k=\min_i f_i\}} \theta_{\varphi_k}^n \leq   \limsup_{k\to\infty}\int_{\{\psi=\min_i f_i\}} e^{ s(\varphi_k-\min_i f_i)} \theta_{\varphi_k^C}^n&\leq \int_{ X} e^{s(\varphi-\min_i f_i)} \theta_{\varphi^C}^n.\end{align*}

%Letting   $s\to+\infty$. Since $e^{s(\varphi-\min_i f_i)}\to \mathbf{1}_{\{\varphi=\min_i f_i\}}$ and $\varphi^C=\varphi$ on the contact set $\{\varphi=\min_i f_i\}$, we conclude   that \begin{align}\label{ineq_diam2}  \limsup_{k\to\infty} \int_{\{\varphi_k=\min_i f_i\}} \theta_{\varphi_k}^n\leq\int_{\{\varphi=\min_i f_i\}}  \theta_{\varphi}^n \leq\ove(\theta,P_{\theta}(u_{1},\ldots,u_{p})). \end{align} Finally, we let $k \to +\infty$ in \eqref{ineq_dia1} and combine with \eqref{ineq_diam2}. The desired inequality then follows by taking the infimum over all $f_i \in C^{\infty}(X)$.

\begin{proposition}\label{1-ener_pro}
    Let $\phi$ be a $\theta$-psh model potential.     Let $\psi, \varphi\in\mathcal{E}(X,\theta,\phi)$ be such that  $\psi \leq \varphi$. Let $u_j\in \mathcal{E}(X,\theta,\phi)$. Define $D_j \coloneqq \{u_j \le \psi \}$. 
Assume that 
    \[
    \sup_{j\in\mathbb{N}} \int_{X} |u_j - \varphi| \, \boldsymbol{1}_{D_j} \, \theta_{u_j}^n < +\infty.
    \]
    Then $\sup_X u_j$ is bounded from below.  Moreover, if $u_j \to u$ in $L^1(X)$, then $u \in \mathcal{E}(X,\theta,\phi)$.
 \end{proposition}
\begin{proof}
       From Lemma \ref{diam_lem}, we know
       that $P_{\theta}(u_j,\psi) \in \mathcal{E}(X,\theta,\phi)$. Note that by the minimum principle, we have
            \begin{align}\label{holds_equiv}  \sup_{j\in\mathbb{N}} \int_{X} |P_{\theta}(u_j,\psi)-\psi|\, \theta_{P_{\theta}(u_j,\psi)}^n &\leq     \sup_{j\in\mathbb{N}}\int_{\{u_j\leq \psi\}  } (\psi-u_j)\,  \theta_{u_j}^n\nonumber \\ &
            \leq \sup_{j\in\mathbb{N}}\int_{  D_j } |u_j-\varphi| \theta_{u_j}^n <+\infty . \end{align}
Hence by replacing $\varphi$ with $\psi$ and $u_j$ with $P_{\theta}(u_j,\psi)$, we may assume that   $\varphi=\psi$ and that  $u_j\leq\varphi$ everywhere. Thus $D_j=X$. Arguing by contradiction, we also assume that $\sup_X u_j \to -\infty$. Fix $s > 0$ and observe that 
\begin{align*}
   \und(\theta,\phi)\leq \int_X\theta_{u_j}^n  \leq &  \frac{1}{s}\int_{\{u_j\leq \varphi-s\} }|u_j-\varphi|\, \theta_{u_j}^n +  \int_{\{u_j>\varphi-s\}} \, \theta_{u_j}^n  \nonumber\\  \leq &\frac{1}{s}\sup_{j\in\mathbb{N}}\int_{X}|u_j-\varphi|\, \theta_{u_j}^n + \int_{\{u_j>\varphi-s\}}   \theta_{\varphi}^n +\int_{X}   \theta_{\max(u_j,\varphi-s)}^n -\int_{X}   \theta_{\varphi}^n 
  \nonumber \\   \leq &  \frac{1}{s}\sup_{j\in\mathbb{N}}\int_{X}|u_j-\varphi|\, \theta_{u_j}^n +\int_{\{u_j>\varphi-s\}}\theta_{\varphi}^n +     \ove(\theta,\phi)-\int_X  \theta_{\varphi}^n 
\end{align*}
where in the second line  we have used the maximum principle. Since $\sup_X u_j \to -\infty$, we have $\int_{\{u_j>\varphi-s\}}   \theta_{\varphi}^n \to 0$  as $j\to+\infty$. By letting $j\to +\infty$ and then $s\to +\infty$, we thus obtain
     \begin{align}\label{contrad_1}  \und(\theta,\phi)\leq  \ove(\theta,\phi)-\int_X  \theta_{\varphi}^n. \end{align}
  By inequality \eqref{holds_equiv}, the estimate \eqref{contrad_1} remains valid upon  replacing $\varphi$ by any $\varphi'$ satisfying $\varphi'\simeq \varphi$ and $\varphi'\leq\varphi$.  Taking the infimum over all such $\varphi'$ then yields $\und(\theta,\phi) = 0$, a contradiction.

Moreover, assume $u_j\to u$ in $L^1(X)$. Then  $u\preceq\phi$. Fix  $b>1$ and assume by contradiction that $P_{\theta}(bu-(b-1)\varphi)\equiv-\infty $.  Since $u_j,\varphi\in\mathcal{E}(X,\theta,\phi)$,  Lemma \ref{lemma ddnl} together with inequality (\ref{u_b_full}) gives  $$h_j\coloneqq P_{\theta}(bu_j-(b-1)\varphi)\in \mathcal{E}(X,\theta,\phi).$$

By Hartogs' lemma,  the hypothesis $P_{\theta}(bu-(b-1)\varphi)\equiv-\infty $ actually implies $h_j \to-\infty$.    Let $C_j\coloneqq\{P_{\theta}(bu_j-(b-1)\varphi)= bu_j-(b-1)\varphi\}$ denote the contact set. Using Proposition \ref{contact_0} and the minimum  principle again, we obtain
 \begin{align*} 
        \und(\theta,\phi)\leq \int_X\theta_{h_j}^n \leq & b^n \int_{\{h_j\leq \varphi-bs\}\cap C_j } \, \theta_{u_j}^n + \int_{\{h_j>\varphi-bs\} } \, \theta_{h_j}^n \nonumber\\   \leq & b^n \int_{\{u_j\leq \varphi-s\} }  \, \theta_{u_j}^n + \int_{\{h_j>\varphi-bs\}}   \theta_{\varphi}^n +\int_{X}   \theta_{\max(h_j,\varphi-bs)}^n -\int_{X}   \theta_{\varphi}^n 
  \nonumber \\    \leq & \frac{b^{n}}{s}\sup_{j\in\mathbb{N}}\int_{X}|u_j-\varphi|\, \theta_{u_j}^n + \int_{\{h_j>\varphi-s\}}   \theta_{\varphi}^n +\ove(\theta,\phi)-\int_X  \theta_{\varphi}^n        
\end{align*}
where in the second inequality we have used  Proposition \ref{contact_0} and the minimum principle,
and in the third inequality we have used the maximum principle. Since $h_j \to -\infty$, letting first $j \to +\infty$ and then $s \to +\infty$ yields \eqref{contrad_1} again.  
For any $\varphi' \simeq \varphi$ satisfying $\varphi' \le \varphi$,  the same argument remains valid after replacing  $\varphi$ by  $\varphi'$ and $u_j$ by $P_{\theta}(u_j,\varphi')$. Taking the infimum over all such $\varphi'$ yields $\und(\theta,\phi) = 0$,
which  again leads to a contradiction. 
We have thus shown that $P_{\theta}(bu-(b-1)\varphi)$ is 
 $\theta$-psh  for every $b>1$.
  Since  $$u\geq   (1-1/b)\varphi+
 (1/b)P_{\theta}(bu-(b-1)\varphi),$$    Proposition \ref{bgl_mono} and Proposition \ref{bgl_sum} then  yield $\ove(\theta,u) = \ove(\theta,\phi)$,  
hence $u \in \mathcal{E}(X,\theta,\phi)$. This completes the proof. 
\end{proof}
\begin{rem} 
If the assumption $\psi \leq \varphi$ is not made, the first conclusion of Proposition~\ref{1-ener_pro} remains valid. Since this  conclusion will not be used, we omit its proof here. The necessary modifications are clear from the proof of Proposition~\ref{general_asl} below. 
\end{rem}
As an application, we establish a criterion for convergence in capacity  which generalize a result
 of Kołodziej–Nguyen (see \cite[Proposition 2.5]{KN22_weak})  to unbounded potentials.  
\begin{theorem}\label{1-ener_thm}
     Let $\phi$ be a $\theta$-psh model potential and  $u_j, u \in \mathcal{E}(X,\theta,\phi)$. 
    Assume that $u_j $ converges to $ u$ in $L^1(X)$ and 
    \[
    \lim_{j\to+\infty} \int_X |u_j - u| \, \theta_{u_j}^n = 0.
    \]
    Then, by passing to a subsequence, $P_{\theta}(\inf_{l\ge j} u_l) \in \mathcal{E}(X,\theta,\phi)$.
    In particular, $u_j$ converges to $u$ in capacity.
\end{theorem}
\begin{proof}
    By passing to a subsequence, we assume  that 
      $$\int_X |u_j-u|\,\theta_{u_j}^n \leq {2^{-j}}$$
Denote $w_{j}^k=P_{\theta}(u_j,\ldots,u_{j+k})$. Proposition \ref{diam_ineq} ensures  that $w_{j}^k\in \mathcal{E}(X,\theta,\phi)$. By  the minimum  principle, we have 
  $$   \int_X |w_{j}^k-u|\, \theta_{w_{j}^k}^n  \leq     \sum_{l=j}^{j+k}\int_X |u_l-u|\,  \theta_{u_l}^n 
            \leq  2^{1-j}. $$
  Now apply Proposition \ref{1-ener_pro} to $w_j^k$ and $u$.
The decreasing limit $w_j \coloneqq \lim_k w_j^k$ is not identically $-\infty$,
and  belongs to $ \mathcal{E}(X,\theta,\phi)$.  By the lower semicontinuity of non-pluripolar product (Lemma~\ref{lsc}), we obtain  $$   \int_X |w_{j} -u|\, \theta_{w_j}^n \leq \liminf_k  \int_X |w_{j}^k-u|\, \theta_{w_{j}^k}^n   
            \leq  2^{1-j}. $$
  
 Let $w \coloneqq \lim_j w_j$ denote the increasing limit. Then $w \in \mathcal{E}(X,\theta,\phi)$.
Using Lemma \ref{lsc}  once more,   we deduce that $$ \int_X |w  -u|\, \theta_{w}^n=0.$$
  The domination principle  (Proposition  \ref{domina_1}) then yields that $w=u $. Finally, we define $v_j \coloneqq \operatorname{usc}\!\bigl( \sup_{l \ge j} u_l \bigr)$.
It follows that $u_j \le v_j \searrow u$ and $u_j \ge w_j \nearrow u$.
 Therefore, $u_j$ converges    in capacity  to $u$.   
\end{proof}

At the end of this subsection, we recall a generalization of the Cegrell-Kołodziej-Xing stability theorem \cite{CK06,Xing,DH_con} to the hermitian setting, established in \cite{ALS24,ALS25} (see also \cite[Theorem~1.5]{DV_qua2} for  quantitative aspects in  Kähler manifolds with closed 
$\theta$).  The stability theorem extend naturally to big classes; notably, the bounded mass condition is invoked here only in Lemma~\ref{diam_lem}.

\begin{theorem}[(cf. {\cite{ALS24,ALS25}})]\label{als_lem}
    Let \( u_j \in \mathcal{E}(X, \theta, \phi) \) be such that \(  \theta_{u_j}^n \leq \mu \), where \( \mu \) is a   non-pluripolar Radon measure. Assume that $u_j$ converges in \( L^1(X) \) to a \(\theta\)-psh function \(u \). Then, up to extracting a subsequence,    both $u$ and  $ P_{\theta}(\inf_{l\geq j} u_{l}) $ belong to $\mathcal{E}(X, \theta, \phi)$. Moreover,    $  u_j$ converges   in capacity to   $u$, and  $  \theta_{u_j}^n$ converges    weakly to  $\theta_u^n $.
\end{theorem}

We need the following generalization to the  stability theorem. One may  compare it with the domination principle (Proposition \ref{domina_1}).
\begin{proposition}\label{general_asl}
     Let \( u_j,\varphi \in \mathcal{E}(X, \theta, \phi) \) be such that \( \boldsymbol{1}_{D_j} \theta_{u_j}^n \leq \mu \), where  $D_j\coloneqq \{u_j\leq \varphi\}$ and  \( \mu \) is a   non-pluripolar Radon measure. Assume that $u_j$ converges in \( L^1(X) \) to a \(\theta\)-psh function \(u \). Then,  up to extracting a subsequence,    both $u$ and  $ P_{\theta}(\inf_{l\geq j} u_{l}) $ belong to $\mathcal{E}(X, \theta, \phi)$.
\end{proposition}
\begin{proof}
     By \cite[Lemma 11.5]{GZ} (see also \cite[Lemma 4.4]{DDNL21LOG}), we have \(\int_X |e^{u_j} - e^u| \, d\mu \to 0\) as \(j \to +\infty\).  Up to extracting a subsequence, we may assume   $$\int_X |e^{u_j}-e^{u}|\, \boldsymbol{1}_{D_j}\theta_{u_j}^n \leq {2^{-j}}.$$
     
Denote  $w_{j}^k=P_{\theta}(u_j,\ldots,u_{j+k})$. Proposition \ref{diam_ineq} then yields that $w_{j}^k\in \mathcal{E}(X,\theta,\phi)$. Fix $b>1$,       
Lemma \ref{lemma ddnl} together with (\ref{u_b_full}) implies that 
 $\w w_{j}^k= P_{\theta}(bw_j^k-(b-1)\varphi)\in \mathcal{E}(X,\theta,\phi) .$     
 
 Once the sequence $\{\sup_X \w w_j^k\}_k$ is shown to be bounded from below   for each $b $ and $j$, we obtain that    $\w w_j \coloneqq \lim_k \w w_j^k$ is not identically $-\infty$. This, in turn,   implies that $ P_{\theta}(\inf_{l\geq j} u_{l})=\lim_k w_{j}^k$ and  $u$ belong to $\mathcal{E}(X,\theta,\phi)$.  

To prove this, let $D_j^k\coloneqq \{w_j^k\leq \varphi\}$ and $\w D_j^k\coloneqq \{\w w_j^k\leq \varphi\}$.  Denote the contact set $$C_j^k\coloneqq \{\w w_j^k=bw_j^k-(b-1)\varphi\}.$$ Note that  $\w D_j^k\cap C_j^k \subseteq D_j^k $.  By Proposition \ref{contact_0} and the
minimum  principle, we have 
   $$   \boldsymbol{1}_{ \w D_j^k}\theta_{\w w_j^k}^n\leq b^n\boldsymbol{1}_{D_j^k}\theta_{w_j^k}^n \leq b^n\mu,$$
and hence
  \begin{align*}  
   \int_X |e^{\frac{1}{b}\w  w_j^k+(1-\frac{1}{b})\varphi}-e^{u}|\, \boldsymbol{1}_{ \w D_j^k}\theta_{\w w_j^k}^n  &\leq  b^n\int_X |e^{w_j^k}-e^{u}|\, \boldsymbol{1}_{D_j^k}\theta_{w_j^k}^n \\& \leq b^n\sum_{l\geq j}^{j+k}\int_X |e^{u_l}-e^{u}|\, \mu  \leq b^n{2^{1-j}}. 
  \end{align*}
Since \( \w{w}_j \) is increasing in \( j \), we   assume, by contradiction,  
that \( \w{w}_j^k \searrow -\infty \) as \( k \to +\infty \) for all sufficiently large \( j \). It follows that $$\lim_{j\to+\infty}\lim_{k\to+\infty} \int_Xe^u\,\boldsymbol{1}_{ \w D_j^k}\theta_{\w w_j^k}^n =0$$ 
Fix \(s \ge 0\). Applying the maximum principle  yields
\begin{align}\label{k_j_s}
   \int_X \theta_{\w w_j^k}^n  &\leq e^s\int_{\{u>-s\}\cap \w D_j^k} e^{u} \,  \theta_{\w w_j^k}^n +\int_{\{u\leq -s\}\cap\w D_j^k}\theta_{\w w_j^k}^n  + \int_{\{\w w_j^k> \varphi\}}\theta_{\w w_j^k}^n \nonumber \\ &\leq e^s\int_{X} e^{u}   \boldsymbol{1}_{ \w D_j^k}\theta_{\w w_j^k}^n +b^n\int_{\{u\leq -s\}} \mu  + \int_{\{\w w_j^k> \varphi\}}\theta_{\varphi}^n+\int_{X}\theta_{\max(\varphi,\w w_j^k)}^n -\int_{X}\theta_{ \varphi }^n
\end{align}
Since $\w w_j^k\in \mathcal{E}(X,\theta,\phi)$, we have 
    $$\int_X \theta_{\w w_j^k}^n\geq \und (\theta,\phi), \quad \int_{X}\theta_{\max(\varphi,\w w_j^k)}^n\leq \ove(\theta,\phi) .$$
  By letting first \( k \to +\infty \), then \( j \to +\infty \), and finally \( s \to +\infty \) in \eqref{k_j_s}, we obtain  $$\und (\theta,\phi)\leq \ove (\theta,\phi)-\int_{X}\theta_{ \varphi }^n.$$
 Observe that for any \(\theta\)-psh function \( \varphi' \simeq \varphi \) with \( \varphi' \le \varphi \),
the same estimate holds. Taking the infimum over all such \( \varphi' \) yields 
\( \und(\theta,\phi) = 0 \), a contradiction.
This completes the proof. 
\end{proof}

\subsection{\textbf{A semi-metric on the space of model singularity types}}

Denote by $\mathcal{S}(X,\theta)$    the space of singularity types of $\theta$-psh functions. 
For   $\theta$-psh functions $u, v$, we introduce the following  distance-like function on the space $\mathcal{S}(X,\theta)$  
\[
d_{\theta}(u, v) := 2 \ove( \theta,\max(u,v)) -   \ove( \theta, u ) - \ove( \theta, v ).
\]
 Given a constant $\delta>0$, we set 
$$\mathcal{S}_{\delta}(X,\theta):=\{u\in \mathcal{S}(X,\theta):  \, \und(X,\theta)>\delta\}. $$ 
By Proposition \ref{stable_prop},  $d_{\theta}$ defines a semi-metric on the space of model singularity types, by which we mean a non-degenerate symmetric function that may not satisfy the triangle inequality.  
Moreover, when \( X \) is K\"ahler and \( \theta \) is closed,     the restriction of \(d_{\theta}\) to \(\mathcal{S}_{\delta}(X,\theta)\)  is comparable to the pseudo-metric $d_{\mathcal{S }(\theta)}$ introduced in \cite{DDNL21}  (see, e.g., ~\cite[Proposition~2.4]{DV_qua2}).
\begin{lemma}\label{d_s_up}
Fix $\delta>0$ and let   $u_j, u \in \mathcal{S}_{\delta}(X,\theta)$ be model potentials.  Assume that $d_{\theta}(u_j,u)\to0$ as $j\to+\infty$. Then, by passing to a subsequence, there exists a decreasing sequence \( v_j \ge u_j \)
such that \( d_{\theta}(v_j, u) \to 0 \). In particular, \( u_j \) converges to \( u \) in \( L^1(X) \).
\end{lemma}
\begin{proof}
Without loss of generality, we may take $\delta=1$. Set $M:= \ove(\theta,V_{\theta})$. By definition we may replace \( u_j \) by \( \max(u_j, u) \) and thus assume \( u_j \ge u \). Then up to extracting a subsequence, we may further suppose $d_{\theta}(u_j,u)\leq 2^{-(j+1)n}$.  Define $$v_{j}^k:= \max (u_j,\cdots,u_{j+k}), \quad v_{j}=\text{usc}\big( \sup_{l\geq j} u_l\big) .$$    By Lemma \ref{lemma ddnl}, for each \( l \) with \( j \le l \le j+k \),   there exists a $\theta$-psh function $h_l\leq u$ such that $(1-2^{-l})u_{l}+2^{-l}h_l\leq u$. Since \( h_l \le u \le v_j^{l-j-1} \), we  thus obtain 
\[
(1 - 2^{-l}) v_j^{l-j} + 2^{-l} h_l \le v_j^{l-j-1}.
\]
  It follows that $\ove(\theta,v_{j}^{l-j})-\ove(\theta,v_{j}^{l-j-1}) \leq   (1-(1-2^{-l})^n )\ove(\theta,v_{j}^{l-j})\leq nM2^{-l}$. Thus 
  $$d_{\theta}(v_j^k,u)=\sum_{l={j+1}}^{j+k} d_{\theta}(v_j^{l-j},v_j^{l-j-1})+d_{\theta}(u_j,u)\leq nM2^{-j}+2^{-(j+1)n}.$$
  By Proposition \ref{dec_inc}(i), we obtain \( d_{\theta}(u, v_j) \to 0 \).
The decreasing sequence \( \{v_j\} \) is the one we need..

 Now we prove the second statement. As $ u_j \le \max(u_j, u) \le v_j$,  
 the argument above actually shows that 
$
d_{\theta}\bigl(u_j, v_j\bigr) \to 0.$   
 Since each \( u_j \) is a model potential, Lemma~\ref{lemma ddnl} provides sequences 
\( b_j \nearrow 1 \) and \(\theta\)-psh functions \( g_j = P_{\theta}[g_j] \) satisfying
\[
b_j v_j + (1 - b_j) g_j \le u_j.
\]
Letting $j\to+\infty$, this implies that $u_j\to u$ in   $L^1(X)$.
\end{proof}

Assume that $d_{\theta}(u_j,u)\to 0$ in $ \mathcal{S}_{\delta}(X,\theta)$. Let $v_j$ be the sequence constructed in   Lemma \ref{d_s_up}.  For  $j<k$, we observe that 
\[
\begin{aligned}
d_{\theta}(u_j, u_k) 
&= d_{\theta}\bigl(u_j, \max(u_j, u_k)\bigr) + d_{\theta}\bigl(u_k, \max(u_j, u_k)\bigr) \\
&\le d_{\theta}(u_j, v_j) + d_{\theta}(u_k, v_j) \\
&= d_{\theta}\bigl(u_j, \max(u_j, u)\bigr) + d_{\theta}\bigl(\max(u_j, u), v_j\bigr) \\
&\quad + d_{\theta}\bigl(u_k, \max(u_k, u)\bigr) + d_{\theta}\bigl(\max(u_k, u), v_j\bigr) \xrightarrow{j,k \to \infty} 0.
\end{aligned}
\]
Hence,   $\{u_j\}_j$ is a $d_{\theta}$-Cauchy sequence.
\begin{theorem}\label{cauchy}
   For every $\delta > 0$, the space $\mathcal{S}_{\delta}(X, \theta)$ is complete. That is, if $\{u_j\}_j \in\mathcal{S}_{\delta}(X, \theta)$ satisfies that   for every $\epsilon > 0$ there exists $N\in\mathbb{N}$ such that $d_{\theta}(u_j, u_k) < \epsilon$ for all $j, k > N$, then there exists $u \in \mathcal{S}_{\delta}(X, \theta)$ such that $d_{\theta}(u_j, u) \to 0$.
\end{theorem}
\begin{proof}
      Assume  $\delta=1$ and set  $M\coloneqq\ove(\theta,V_{\theta})$. 
         Up to extracting a subsequence,  we assume $d_{\theta}(u_j,u_{j+1})\leq 2^{-(j+1)n}$.  Define \( v_j^k \) and \( v_j \) as in the proof of Proposition \ref{d_s_up}.  For every  $j \leq l\leq j+k$,  Lemma \ref{lemma ddnl}   provides a $\theta$-psh function $h_l\leq u$ such that $(1-2^{-l})\max(u_{l},u_{l+1})+2^{-l}h_l\leq u_{l}$.  Since   $h_l\leq u_l $, we obtain  
  $$ (1-2^{-l})v_{j}^{l+1-j}+2^{-l}h_l\leq v_l^{l-j}.$$
  It follows that $\ove(\theta,v_{j}^{l+1-j})-\ove(\theta,v_{j}^{l-j}) \leq   (1-(1-2^{-l})^n )\ove(\theta,v_{j}^{l-j})\leq nM2^{-l}$. Consequently,
  $$d_{\theta}(v_j^k,u_j)=\sum_{l={j+1}}^{j+k} d_{\theta}(v_j^{l-j},v_j^{l-j-1})\leq nM2^{1-j}.$$
  By Proposition \ref{dec_inc}(i),  we deduce that $d_{\theta}({u_j,v_j})\to0$. Let $u$ denote the decreasing limit of $P_{\theta}[v_j]$. Proposition \ref{dec_inc}(ii) then gives \( d_{\theta}(u, v_j) \to 0 \), and hence $$d_{\theta}(u,u_j)\leq d_{\theta}(u_j,v_j)+d(u,v_j)\to 0.$$

   By Proposition \ref{d_s_up}, $u$ is an $L^1$-limit of $u_j$. To see that the convergence is not up to taking a subsequence, it suffices to show that if two subsequences $u_{a_j}\to u$ and $u_{b_j}\to u'$  satisfy $d_{\theta}(u_{a_j},u_{b_j}) \to 0$, then $u=u'$. 
  Since each $u_j$ is a model potential, using Lemma \ref{lemma ddnl},  there exist $c_j\nearrow1$ and   $\theta$-psh functions $h_j=P_{\theta}[h_j] $ such that $$c_j u_{b_j}+(1-c_j)h_j\leq c_j\max(u_{a_j},u_{b_j})+(1-c_j)h_j\leq u_{a_j}.$$ Letting $j\to+\infty$, we obtain $u\geq u'$. Exchanging the roles of  $a_j$ and $b_j$ gives the reverse inequality, so  $u=u'$. This completes the proof.
\end{proof}
The following lemma strengthens   Proposition \ref{diam_ineq} 
by eliminating the error term $\Delta(\theta,w)$ in the  approximating sense.
\begin{proposition}\label{d_s_low}
   Fix $\delta>0$ and let   $u_j, u \in \mathcal{S}_{\delta}(X,\theta)$.  Suppose that $d_{\theta}(u_j,u)\to0$ as $j\to+\infty$. Then $d_{\theta}(P_{\theta}(u_j,u),u)\to 0$.
\end{proposition}
\begin{proof}
    Assume $\delta=1$ and set $M\coloneqq \ove(\theta,V_{\theta})$. Since \( \und(\theta, \max(u_j, u))   - d_{\theta}(u_j, u)> 0 \) for $j$ sufficiently large, 
  Lemma \ref{diam_lem} then implies that \( P_{\theta}(u_j, u) \in\psh(X,\theta) \).  Up to extracting a subsequence, we assume $d_{\theta}(u_j,u)\leq 2^{-(j+1)n}$.
    For each \( 1 \le b < 2^{j+1} \),   Lemma \ref{lemma ddnl} yields
\[
g_{b,j} := P_{\theta}\bigl(b u_j - (b-1)\max(u_j, u)\bigr) \not\equiv -\infty, \quad
h_{b,j} := P_{\theta}\bigl(b u - (b-1)\max(u_j, u)\bigr) \not\equiv -\infty.
\]  Fix $0\le s\leq j$. Observing that 
  $$ g_{2^{j-s},j}\geq (1-2^{-s})\max(u_j,u)+2^{-s}g_{2^{j},j},\quad h_{2^{j-s},j}\geq (1-2^{-s})\max(u_j,u)+2^{-s}h_{2^{j},j}\quad $$
  It follows that   $2\ove(\theta,\max(u_j,u))-\ove(\theta,g_{2^{j-s},j})-\ove(\theta,g_{2^{j-s},j})\leq nM2^{1-s}$. Choose \( s \) large enough so that $nM2^{1-s}<1\leq \und(\theta,\max(u_j,u))$. Applying  Lemma \ref{diam_lem} with  $w=\max(u_j,u)$,  then we obtain  that   $P_{\theta}(g_{2^{j-s},j},h_{2^{j-s},j})\in \psh(X,\theta)$.    We thus have
  $$   P_{\theta}(u_j,u)\geq (1-2^{-(j-s)})\max(u_j,u)+2^{-(j-s)}P_{\theta}(g_{2^{j-s},j},h_{2^{j-s},j}) $$
  By Proposition \ref{bgl_mono} and Proposition~\ref{bgl_sum}, we conclude that $$d_{\theta}(P_{\theta}(u_j,u),u)\leq d_{\theta}(P_{\theta}(u_j,u),\max(u_j,u))\leq nM2^{-(j-s)}\to 0,\quad \text{as}\,\, j\to+\infty.$$     This  completes the proof. 
\end{proof}

\begin{theorem}\label{model_capacity}
    Fix $\delta>0$ and let   $u_j, u \in \mathcal{S}_{\delta}(X,\theta)$ be model potentials.  Suppose that $d_{\theta}(u_j,u)\to0$ as $j\to+\infty$. Then, after passing to a subsequence (still denoted by $u_{j}$), a decreasing sequence $v_{j}\geq u_{j}$ and an  increasing sequence $w_{j}\leq u_{j}$ such that $d_{\theta}(v_{j},u)\to0$ and $d_{\theta}(w_{j},u)\to0$. In particular,   $u_j$ converges to $ u$ in  capacity.
\end{theorem}
\begin{proof}
    The decreasing sequence is given by Lemma \ref{d_s_up}.  Assume  $\delta=1$ and set $M\coloneqq \ove(\theta,V_{\theta})$.  By Proposition $\ref{d_s_low}$, after replacing  $u_j$ with $P_{\theta}(u_j,u)$, we may assume that $u_j\leq u$. Up to extracting a subsequence, we can further assume $d_{\theta}(u_j,u)\leq 4^{-(j+1)n}$.  Then  $$\und(\theta,u)+\sum_{l=j}^{j+k}\ove(\theta,u_j)-(k+1)\ove(\theta,u) \geq 1-\sum_{l\geq j}4^{-(j+1)n} >0.$$
Apply Lemma \ref{ineq_dia1} with $w=u$ thus  implies that  $w_{j}^k:=P_{\theta}(u_j,\ldots,u_{j+k})\in\psh(X,\theta)$. Now fix \( l \) with \( j \le l \le j+k \).
For each \( 1 \le b < 4^{l+1} \), we have 
\( h_{b,l} := P_{\theta}\bigl(b u_l - (b-1) u\bigr) \in \psh(X, \theta) \) by Lemma \ref{lemma ddnl}. Observing that 
  $$  h_{2^{l},l}\geq (1-2^{-l})u+2^{-l}h_{4^{l},l}\quad $$
Consequently,  $(k+1)\ove(\theta,u)-\sum_{l=j}^{j+k}\ove(\theta,h_{2^l,l})\leq nM2^{1-j}\le 2^{-1}$ for $j$ sufficiently large. Applying Lemma \ref{diam_ineq}  with $w=u$ then yields that  $P_{\theta}(h_{2^j,j},\ldots,h_{2^{j+k},j+k})\in\psh(X,\theta)$. Hence \( P_{\theta}(h_{2^j,j},\dots,h_{2^{j},j+k}) \) is also \(\theta\)-psh.  Note that $$  w_{j}^k\geq (1-2^{-j})u+2^{-j}{P_{\theta}(h_{2^j,j},\ldots,h_{2^{j},j+k})}.$$  
 By Proposition \ref{bgl_mono} and Proposition~\ref{bgl_sum},  we have $$   \und(\theta,w_j^k)\geq\und(\theta,u)-nM2^{-j}\ge\frac{1}{2},\quad d_{\theta}(w_j^k,u)\leq nM2^{-j} $$
 for $j$ sufficiently large.
  Since we can verify from the definition that $w_j^k=P_{\theta} [w_{j}^k]$, we know that  $w_{j}^k$ is a model potential and that $\sup_X w_j^k=0$. Let $w_j$ denote the  decreasing limit of $ w_j^k$. Then $w_j$ is a $\theta$-psh function. By Proposition~\ref{dec_inc}(ii), we conclude that   $$   \und(\theta,w_j)\ge\frac{1}{2},\quad    d_{\theta}(w_j,u)\leq nM2^{-j}.$$
 The increasing sequence $\{w_j\}$ then gives the desired functions.  
\end{proof}

  Suppose \( u_j, u \in \mathcal{S}_{\delta}(X,\theta) \) with \( d_{\theta}(u_j, u) \to 0 \). After passing to a subsequence, set \( \phi_j = P_{\theta}[u_j] \), \( \phi = P_{\theta}[u] \), and define
\[
w_j^k = P_{\theta}(u_j, \ldots, u_{j+k}), \quad
\widehat{w}_j^k = P_{\theta}(\phi_j, \ldots, \phi_{j+k}), \quad
\widehat{w}_j = \lim_{k} \widehat{w}_j^k.
\]
The proof above actually shows that for some \( N \in \mathbb{N} \), whenever \( k \) is arbitrary and \( j > N \), we have \( w_j^k \not\equiv -\infty \) and \(  w_j^k  \in \mathcal{E}(X, \theta, \widehat{w}_j^k) \).

Moreover, since $d_{\theta}(\widehat{w}_j^k, \widehat{w}_j) \to 0$, Lemma \ref{diam_lem} implies that $P_{\theta}(w_j^k, \widehat{w}_j)$ is not identically $-\infty$ and belongs to $\mathcal{E}(X, \theta, \widehat{w}_j)$. Therefore, by an analogous argument, we may extend Proposition~\ref{general_asl} to the case of varying $\phi$.

\begin{corollary}\label{stability_lem}
 Fix $\delta>0$ and  let  $\phi_j,\phi\in\mathcal{S}_{\delta}(X,\theta)$ be model potentials such that $d_{\theta}(\phi_j,\phi)\to0$. Let $u_j  \in \mathcal{E}(X,\theta,\phi_j)$ and   $\varphi  \in \mathcal{E} (X,\theta,\phi)$.   Assume that $u_j$ converges to $u$ in $L^1(X)$ 
 and  that \( \boldsymbol{1}_{D_j} \theta_{u_j}^n \leq \mu \), where  $D_j\coloneqq \{u_j\leq \varphi\}$ and  \( \mu \) is a   non-pluripolar Radon measure. 
 
 Then, after passing to a subsequence,     $P_{\theta}(\inf_{l\ge j} u_l )\in  \mathcal{E}(X,\theta,P_{\theta}(\inf_{l\ge j}\phi_l))$ and $u\in \mathcal{E}(X,\theta,\phi)$. 
 Moreover, if $D_j=X$, then  $u_j\to u$ in capacity and  $  \theta_{u_j}^n\to\theta_u^n $ weakly.
\end{corollary}

\section{Monge-Amp\`ere Equations in the Full Mass Class}\label{sec_5}

\hspace*{1.5em}  Let $\mathcal{M}$ denote the (weakly) compact convex set of positive measures generated by all measures of the form $(\omega_X + dd^c \psi)^n$, where $\psi$ is a bounded  $\omega_X$-psh function satisfying $0 \leq \psi \leq 1$ (cf. \cite[Proposition 3.2]{BEGZ10}). By the Chern–Levine–Nirenberg inequality, for any $\omega_X$-psh function $\varphi$ normalized by $\sup_X \varphi = 0$ and any $\omega_X$-psh function $\psi$ with $0 \leq \psi \leq 1$, there exists a uniform constant $C$ such that
    $$\int_X -\varphi(\omega_X+dd^c\psi)^n\leq C.$$
Using the convexity of $\mathcal{M}$ and the lower semicontinuity of weak convergence of measures with respect to lower semicontinuous functions, we obtain for every $\nu \in \mathcal{M}$ \begin{align}\label{cln} \int_X -\varphi \,d\nu\leq C.\end{align} 

Now let $\mu$ be a non‑pluripolar positive Radon measure. Since $\mathcal{M}$ is compact and convex, a generalization of the Radon–Nikodym theorem (see \cite{Radon}) yields a decomposition
   \[
\mu = g\nu + \nu',
\]
where $\nu \in \mathcal{M}$, $0 \leq g \in L^1(X,\nu)$, and $\nu' \perp \mathcal{M}$.
 As  all such $\nu \in \mathcal{M}$ characterize the pluripolar sets (see  \cite[Proposition 8.30]{Dinew2019} and \cite{Vu_loc}),  we know that $\nu'$ is supported on a pluripolar set. This implies that $\nu'=0$ and that $\mu$ is absolutely continuous with respect to $\nu$.

\subsection{\textbf{Energy estimates}}
\begin{proposition}\label{p-energy}
Let \(\mu\) be a probability measure on \(X\) such that \(\mu \leq A\nu\) for some 
\(\nu \in \mathcal{M}\) and some constant \(A > 0\).  
Let \(\phi\) be a \(\theta\)-psh model potential.
For sufficiently small \(\alpha, \varepsilon > 0\), define
\[
p(\alpha, \varepsilon) \coloneqq 1 + \frac{(1-\varepsilon)(1-\alpha)}{n+\varepsilon},
\qquad
m(\varepsilon) \coloneqq 1 + \frac{1-\varepsilon}{n+\varepsilon}.
\]
Then for every \(p' < p \), there exists a uniform constant  
\(C(\alpha , \varepsilon,p',A, \phi ) > 0\) such that
\[
\int_X |u - \phi|^{p'} \, d\mu 
\;\leq\; 
C\Bigl[ \bigl(\int_X |u - \phi| \, \theta_u^n\bigr)^{m} + 1 \Bigr]
\]
holds for any \(\theta\)-psh function \(u\) that has the same singularity type as \(\phi\) 
and satisfies \(\sup_X u = 0\).
\end{proposition}

The proof utilizes techniques from the $L^{\infty}$ a priori estimates for complex 
Monge-Amp\`ere equations established in \cite[Theorem 2.1]{GL1}.

Before proceeding to the proof, we first recall a fundamental lemma, whose proof follows the same lines as 
\cite[Lemma 1.7]{GL1} in the relative setting. It is a direct consequence of 
Proposition~\ref{contact_0} together with the minimum principle.
 \begin{lemma}[(cf. {\cite[Lemma 1.7]{GL1}})]\label{concave_lem}
 Let \(u,\phi\) be as in Proposition~\ref{p-energy}.   Fix a concave increasing function \(\chi : \mathbb{R}^{-} \to \mathbb{R}^{-}\) such
that \(\chi'(0) \geq 1\). Let   \(v := \chi \circ (u - \phi) + \phi\). Then
\[
\bigl(\theta + dd^{c} P_{\theta}(v)\bigr)^{n} \;\leq\; 
\mathbf{1}_{\{P_{\theta}(v)=v\}} \; \bigl(\chi' \circ (u - \phi)\bigr)^{n} \; 
\bigl(\theta + dd^{c} u\bigr)^{n}.
\]

 \end{lemma} 
 
 \begin{proof}[Proof of Proposition~\ref{p-energy}]
Define \(  T_{\max} := \sup \bigl\{ t \geq 0 : \theta_u^n\{u<\phi-t\}>0\bigl\} \). Because \(u\) has the same singularity type as \(\phi\), we have \(T_{\max}<+\infty\).  
The domination principle implies that \(T_{\max}\) is the smallest constant 
satisfying \(u \ge \phi - T_{\max}\). We may assume \(T_{\max} \ge 2\); otherwise 
the desired estimate follows directly.

We now construct a concave increasing weight function \(\chi : \mathbb{R}^- \to \mathbb{R}^-\) 
with \(\chi(0)=0\) and \(\chi'(0)=1\). 
For this purpose,   we fix   a sufficiently small constant \(s>0\) and 
define an integrable function \(g_s(t)\) by  
\[
g_s(t) :=
\begin{cases}
\dfrac{1}{(1+t)^{1+\alpha}\theta_u^n(u< \phi- t)}, &\text{ }\,\, t \in [0, T_{\max}-s]; \\[6pt]
\dfrac{1}{(1+t)^{1+\alpha}}, &\text{ }\,\, t > T_{\max}-s.
\end{cases} 
\]
Set $f_s(x) \coloneqq \int_0^x g_s(t) \,dt + 1$. Denote   
\(N(\varepsilon) := \dfrac{n+\varepsilon}{1-\varepsilon}\) and define
      $$-\chi_s(-x)=\int_0^x f_s(t)^{\frac{1}{N(\varepsilon)}}dt$$
 The following estimate then holds:  \begin{align}\label{1+1/a}
 \int_X \bigl(\chi'_s \circ (u - \phi)\bigr)^{N(\varepsilon)}    \theta_u^n &= \int_X  f_s (|u - \phi|) \, \theta_u^n \nonumber \\ &= \int_0^{T_{\max}}g_s(t)\,\theta_u^n(u<\phi-t)\,dt   + \int_X\theta_{u}^n\nonumber\\& \leq \left(1 + \int_0^{+\infty} \frac{1}{(1+t)^{1+\alpha}}\,dt\right)\max\!\left(1,\int_X\theta_u^n\right) \nonumber\\
&=  \left(1 + \frac{1}{\alpha}\right)\max\!\left(1,\int_X\theta_u^n\right)\end{align}      
Set \( B := \int_X |u - \phi| \, d\theta_u^n \). By the Chebyshev  inequality, we obtain  
\[
\theta_u^n\bigl(\{u < \phi - t\}\bigr) \le \frac{B}{t}, \qquad \text{}t > 0.
\]  
For \( t \in [1,\, T_{\max} - s] \), we obtain
\[
g_s(t) \ge \frac{t}{B(2t)^{1+\alpha}} \ge \frac{C_1}{B t^{\alpha}} .
\]
Integrating both sides from \(1\) to \(t\), where \(t \in [1, T_{\max} - s]\), we have
\[
f_s(t) \ge C_1\frac{t^{1-\alpha}-1}{(1-\alpha)B}+1 \ge C_2\frac{t^{1-\alpha}}{\max(1,B)},
\]
because \(\alpha>0\) is  sufficiently small. Define
\( 
\beta(\alpha,\varepsilon) \coloneqq 1+\frac{1-\alpha}{N(\varepsilon)}
\).   Thus  we have
\begin{align}\label{beta}
-\chi_s(-t) &\ge C_3 \max(1, B)^{-\frac{1}{N(\varepsilon)}} \left( t^{\beta(\alpha,\varepsilon)} - 1 \right) - \chi_s(-1) \nonumber \\
&\geq C_4 \max(1, B)^{-\frac{1}{N(\varepsilon)}} \, t^{\beta(\alpha,\varepsilon)},
\end{align}
for all \( t \in [1, T_{\max} - s] \). Here we use the fact that   \( -\chi_s(-1) \geq 1 \). 

Now set \(v_s := \chi_s \circ (u - \phi) + \phi\) and \(\varphi_s \coloneqq P_{\theta}(v_s)\).
Since \(\chi_s\) is concave and   \(\chi_s(0)=0\), we have
\(|\chi_s(-t)| \le |t|\,\chi_s'(-t)\) for all \(t\le 0\). Note  that \(\varphi_s\) has the same singularity type as \(\phi\). Then by applying Lemma \ref{concave_lem} and   H\"older's inequality, we obtain the following energy estimate:
\begin{align*} \begin{aligned}
  \und(\theta,\phi)\left(-\sup_X(\varphi_s-\phi)\right)^{\varepsilon}& \leq \int_X (\phi - \varphi_s)^\varepsilon \, \theta_{\varphi_s}^n \\
& \quad \leq \int_X \bigl(-\chi_s \circ (u - \phi)\bigr)^\varepsilon \; 
      \bigl(\chi'_s \circ (u - \phi)\bigr)^n \,  \theta_u^n \\
& \quad \leq \int_X (\phi - u)^\varepsilon \; 
      \bigl(\chi'_s \circ (u - \phi)\bigr)^{n+\varepsilon} \,  \theta_u^n \\
& \quad \leq \left(\int_X (\phi - u) \theta_u^n \right)^{  \varepsilon } 
      \cdot  \left( \int_X \bigl(\chi'_s \circ (u - \phi)\bigr)^{\frac{n+\varepsilon}{1- \varepsilon}}    \theta_u^n \right)^{1-\varepsilon}   \end{aligned}
\end{align*}
Hence \begin{align}\label{CB} |\sup_X(\varphi_s-\phi)| \leq \ww C(\alpha,\varepsilon ,\phi)   B ,  \end{align}
where $\ww C$ is a uniform constant depending only on $\alpha^{-1}.\varepsilon^{-1}$ and  $\frac{\max\left(1,\ove(\theta,\phi)\right)}{\und(\theta,\phi)}$.
By the hypothesis that $\mu \leq A \nu $ for some  $\nu\in\mathcal{M}$,  we can define $$A_1(\mu)\coloneqq \sup\left\{  \int_X -\psi\,d\mu:  \psi\,\, \text{is $\theta$-psh  and} \,\sup_X\psi=0 \right\}<+\infty. $$

 Since $\varphi_s-\phi \leq u-\phi \leq 0$, we can use \eqref{beta} and \eqref{CB} to bound the $\mu$-measure of the level set $\{u<\phi-t\}$ for $t\in[1,T_{\max}-s]$ as follows:
 \begin{align*}
    \mu(u<\phi-t)&\leq  \frac{  \|\chi_s\circ(u-\phi)\|_{L^1(X,\mu)}}{|\chi_s(-t)|} \\& \leq C_4\max(1,B)^{ \frac{1}{N(\varepsilon)}} t^{-\beta( \alpha,\varepsilon)} \| \varphi_s-\phi\|_{L^1(X,\mu)} 
    \\&   \leq  C_4\max(1,B)^{ \frac{1}{N(\varepsilon)}} t^{-\beta( \alpha,\varepsilon)} \left(\|\varphi_s-\phi-\sup_X(\varphi_s-\phi)\|_{L^1(X,\mu)} +|\sup_X(\varphi_s-\phi)|\right)
     \\&   \leq     C_4\max(1,B)^{ \frac{1}{N(\varepsilon)}} t^{-\beta( \alpha,\varepsilon)}      (A_1(\mu)+\ww C(\alpha,\varepsilon,\phi)B) 
  \end{align*}

As we can take  \(s\)  sufficiently small, this estimate remains valid on  \([1,T_{\max}]\). Hence,  we obtain
\begin{align*}%\label{P-estim}
    \int_X|u-\phi|^p\, d\mu &= \int_0^{T_{\max}} t^{p-1}\, \mu\!\left(\{u<\phi-t\}\right)dt \nonumber\\
    &\leq \mu(X) + \int_1^{T_{\max}} t^{p-1}\, \mu\!\left(\{u<\phi-t\}\right)dt \nonumber\\
    &\leq C(\alpha,\varepsilon,A,\phi)\left(1+B^{1+\frac{1}{N(\varepsilon)}}\right)
           \int_1^{+\infty} t^{p-1-\beta(\alpha,\varepsilon)}\, dt. \end{align*}  
  We therefore conclude that the desired estimate holds for any $1\leq p < \beta(\alpha,\varepsilon)$.   
\end{proof} 

Adapting the proof of \cite[Theorem 2.3]{GL3} to the relative setting yields the following    $L^{\infty}$ a priori estimates. We  remark that the bounded mass property is not required here. 
The reason is that the total mass of \(\theta_u^n\) is already bounded and \(u\) has the same singularity type as \(\phi\); 
consequently, the domination principle (Proposition~\ref{domina_1}) remains applicable. 

\begin{theorem}[(cf.  {\cite[Theorem 2.3]{GL3}})]\label{rel_infity}
Let \(\mu\) be a probability measure on $X$  such that \(\operatorname{PSH}(X,\theta) \subset L^{m}(\mu)\) for some \(m > n\).  
Assume  that \(\phi\) is  a \(\theta\)-psh model potential and that  \(u \in \psh(X,\theta)\) satisfies \(\sup_X u = 0\), \(u \simeq \phi\), and 
\(\theta_u^n \le c\mu\) for a constant \(c > 0\). Then 
\[
u \ge \phi - T,
\]
where \(T > 0\) is a uniform constant depending   on the upper bound of 
\(\dfrac{c}{\und(\theta,\phi)}\) and on  
\[
A_m(\mu) \coloneqq 
\sup\left\{\, 
\Bigl(\int_X (-\varphi)^m \, d\mu\Bigr)^{\frac{1}{m}} \;:\; 
\varphi \text{ is $\theta$-psh with } \sup_X \varphi = 0 
\,\right\}.
\]
\end{theorem}
\subsection{\textbf{Existence of solutions}}
 \hspace{1.5em}  In this part, we proceed to prove Theorem \ref{thm1.1} for the case $\phi=V_{\theta}$. We follow  a strategy of Guedj and Zeriahi \cite{GZ05} (originally from \cite{Ceg98} in the local setting), which approximates the measure $\mu$ by means of local convolutions and a partition of unity.
\begin{theorem}\label{minimal_nonp_0}
   Let  $\mu$ be a positive non-pluripolar  Radon measure. Then there exist       $\varphi\in\mathcal{E}(X,\theta)$ normalized by  $\sup_X \varphi=0$ and a unique constant $c>0$ such that     $\theta_{\varphi}^n= c\mu.$    
\end{theorem}
\begin{proof}
     The constant 
$c$ is unique follows from Corollary~\ref{cor_domi}. Moreover, $ c$ is bounded below by
\( 
  \und(\theta)/  \int_X d\mu.
\) We first assume that  \(\mu \leq A\nu\) for some 
\(\nu \in \mathcal{M}\) and some constant \(A > 0\). Let \(\{U_{\alpha}\}_{\alpha=1}^N\) be a finite covering of \(X\) by coordinate balls, and let \(\{\rho_{\alpha}\}_{\alpha=1}^N\) be a partition of unity subordinate to \(\{U_{\alpha}\}_{\alpha=1}^N\).
We define the regularized measures \(\mu_j\) by 
\begin{align*} \mu_j = \sum_{\alpha=1}^{N} (\rho_{\alpha} \mu|_{U_{\alpha}}) * \chi_{j}, \end{align*}
where \(\chi_j\) are spherically symmetric mollifiers converging to the Dirac measure. For each $j$, applying \cite[Theorem D]{BGL25}, there exists a sup‑normalized $\theta$-psh function $\varphi_j \simeq V_{\theta}$ and a unique constant $c_j>0$ such that     $\theta_{\varphi_j}^n= c_j\mu_j.$   

Since $\mu_j \to \mu$ weakly and $\int_X \theta_{\varphi_j}^n \le \ove(\theta)$, the sequence $\{c_j\}$ is bounded from above. Passing to a subsequence, we may assume $c_j \to c$.
On each coordinate ball \(U_{\alpha}\), choose a smooth psh  function \(g_{\alpha}\) defined on a neighborhood of \(\overline{U}_{\alpha}\) such that \(\theta \le dd^c g_{\alpha}\). Then $g_\alpha$ and $\varphi_j+g_\alpha$ are psh on $U_\alpha$. For sufficiently large $j$, we have
   \begin{align}\label{convolution_bdd}
       \int_X (V_{\theta}-\varphi_j)\, \theta_{\varphi_j}^n   \leq  \int_X   (-\varphi_j) \, \theta_{\varphi_j}^n  &\leq C\sum_{\alpha} \int_{U_{\alpha}} \big( g_{\alpha}  * \chi_j - (\varphi_j+g_{\alpha}) *\chi_j\big) \rho_{\alpha}d\mu
  \nonumber \\&    \leq C\sum_{\alpha} \int_{U_{\alpha}} \big(  g_{\alpha}   * \chi_j -  \varphi_j-g_{\alpha}  \big)  d\mu
   \end{align}
where in the last inequality we have used $V_{\theta}\leq 0$, $0\le\rho_{\alpha} \leq1 $ and the fact that $(\varphi_j+g_{\alpha})*\chi_j\geq\varphi_j+g_{\alpha}$. 
  Using the Chern-Levine-Nirenberg inequality (\ref{cln})  and the fact that  $g_{\alpha}*\chi_j$ is uniformly bounded, we obtain  that $ \int_X (V_{\theta}-\varphi_j)\, \theta_{\varphi_j}^n$ is uniformly bounded in $j$. Then Proposition  \ref{p-energy} implies that there exists a constant $p>1$ such that \begin{align}\label{ineq-p-int} \sup_{j\in\mathbb{N}}\,\int_X |V_{\theta}-\varphi_j|^p\, \theta_{\varphi_j}^n<+\infty\end{align}
  It then follows from \cite[Lemma 11.5]{GZ} applied with 
\( \varphi_j - V_\theta\) (which remains valid because 
\(\{V_\theta = -\infty\}\) is a pluripolar set)  that,after passing to a subsequence, $\varphi_j$ converges in $L^1(X)$ to a $\theta$-psh function $\varphi$ satisfying $\sup_X\varphi=0$ and
       \begin{align}\label{11.5} \lim_{j\to+\infty}\int_X|\varphi_j-\varphi|\,d{\mu}=0. \end{align}
 Let $v_j\coloneqq \text{usc}(\sup_{l\geq j}\varphi_j)$. Note that    $|\varphi_j-\varphi|= \max(\varphi_j,\varphi)-\varphi_j-\varphi \leq 2v_j -\varphi_j-\varphi$,  so we have
\begin{align}\label{convolu_conveger}
 \int_X |\varphi_j-\varphi| \, \theta_{\varphi_j}^n  
   &\le C\sum_{\alpha} \int_{U_{\alpha}} \Bigl(2(v_j+g_{\alpha}) * \chi_j - (\varphi +g_{\alpha}) *\chi_j - (\varphi_j +g_{\alpha}) *\chi_j\Bigr) d\mu \nonumber \\
   &\le C\sum_{\alpha} \int_{U_{\alpha}} \Bigl(2(v_j+g_{\alpha}) * \chi_j - \varphi - g_{\alpha} - \varphi_j - g_{\alpha}\Bigr) d\mu \nonumber   \\
    &\le C\sum_{\alpha}\biggl( \int_{U_{\alpha}} \Bigl(2(v_j+g_{\alpha}) * \chi_j - 2(\varphi+g_{\alpha})\Bigr)d\mu 
          +\int_{U_{\alpha}} (\varphi-\varphi_j) d\mu \biggr). 
    \end{align}
Here the first term  on the right-hand side tends to $0$ by  the monotone convergence      $(v_j+g_{\alpha})  * \chi_j\searrow \varphi+g_{\alpha}$,   while the second term tends to $0$ by   (\ref{11.5}).  Consequently,
\[
\lim_{j\to+\infty}\int_X| \varphi_j-\varphi|\,\theta_{\varphi_j}^n = 0.
\]
By applying Proposition~\ref{1-ener_pro} to (\ref{ineq-p-int}), we  see that  $\varphi \in \mathcal{E}(X,\theta)$. It then follows from   Theorem \ref{1-ener_thm} that   $\varphi_j$ converges to $\varphi$ in capacity. To prove that $\varphi$ satisfies the required equation, it suffices to show that $\theta_{\varphi_j}^n$ converges weakly to $\theta^n_{\varphi}$.  For this purpose, we observe that
\begin{align}\label{argument_rep}
    \int_{\{\varphi_j\leq -t\}} \theta_{\varphi_j}^n &\leq\frac{ 1}{t} \int_{X}(-\varphi_j)\,\theta_{\varphi_j}^n\leq \frac{C'}{t}  \to\; 0 
\qquad \text{as } t \to +\infty,
\end{align}
where the convergence is uniform in $j$. Combined with Lemma \ref{lsc}, this implies that $\theta_{\varphi_j}^n \to \theta^n_{\varphi}$ weakly. Therefore, we conclude that $\theta_{\varphi}^n=c\mu$.

 For the general case, we write $\mu=g\nu$, where $\nu\in\mathcal{M}$ and $g\in L^1(X,\nu)$.  Set $\mu_j=\min(g,j)\, \nu$. By the preceding result, for each $j$, there exists a constant  $c_j>0$ and a function $\varphi_j\in\mathcal{E}(X,\theta)$ with  $\sup_X\varphi_j=0$  such that  $\theta_{\varphi_j}^n=c_j\mu_j$.
 
Since $\mu_j \ge \mu_1$ and $\theta_{\varphi_j}^n \le \ove(\theta)$, the sequence $\{c_j\}$ is bounded from above. By passing to   a subsequence, we may assume that $c_j \to c \ge 0$ and that $\varphi_j$ converges to a $\theta$-psh function $\varphi$ normalized by $\sup_X\varphi = 0$. By Lemma~\ref{lem_2.8} below, we have  $  
\theta_{\varphi}^n \ge c\mu$.
 
Note also that $\theta_{\varphi_j}^n \le C\mu$ for some uniform constant $C>0$. Theorem \ref{als_lem}  implies that $\varphi \in \mathcal{E}(X,\theta)$ and that $\varphi_j$ converges to $\varphi$ in capacity. It then follows from the lower semicontinuity of the non‑pluripolar product that $ \int_X \rho  \,\theta_{\varphi}^n \leq  \liminf_j \int_X c_j \rho\, d\mu_j =\int_X c \rho \, d\mu$ for any non-negative smooth function $\rho$. Consequently,  $\theta_{\varphi}^n = c\mu$,  completing  this proof.
\end{proof}
\begin{rem}
From the preceding result, there exist a function $\varphi \in \mathcal{E}(X,\omega_X)$ and a constant $c>0$ such that $\mu = c\, (\omega_X + dd^c\varphi)^n$. Set $\psi \coloneqq e^{\varphi}$. Then
\[
(\omega_X + dd^c\psi)^n \ge e^{n\varphi}(\omega_X + dd^c\varphi)^n.
\]
We can therefore express $\mu$ as $\mu = g\, (\omega_X + dd^c\psi)^n$, where $\psi$ is a bounded $\omega_X$-psh function and $g \in L^1\bigl(X, (\omega_X + dd^c\psi)^n\bigr)$.
\end{rem}

We now present a lemma that extends \cite[Lemma 2.8]{DDNL21LOG} to the hermitian setting.
\begin{lemma} \label{lem_2.8}  Let $\mu$ be a positive non-pluripolar   Radon measure, and let $\{\varphi_j\}$ be a sequence of $\theta$-psh functions such that $\theta_{\varphi_j}^n \ge f_j\mu$ with $f_j \in L^1(X,\mu)$. 
    Suppose that $f_j$ converges in $L^1(X,\mu)$ to $f \in L^1(X,\mu)$ and that $\varphi_j$ converges in $L^1(X, \omega_X^n)$ to a $\theta$-psh function $\varphi$. 
    Then $\theta_{\varphi}^n \ge f\mu$.
 \end{lemma} 

\begin{proof}  
  Set $v_j^k \coloneqq \max(\varphi_j,\dots,\varphi_{j+k})$ and $v_j \coloneqq \operatorname{usc}\!\bigl(\sup_{l\ge j}\varphi_j\bigr)$.  
    Fix $C>0$ and define  
    \[
    v_j^{k,C} \coloneqq \max\bigl(v_j^k,\; V_{\theta}-C\bigr),\qquad 
    v_j^C \coloneqq \max\bigl(v_j,\; V_{\theta}-C\bigr).
    \]    Recall that   $\Omega\coloneqq \{\rho>-\infty\}$ is a   zariski open set and  that $\{V_{\theta}>-\infty\}\subseteq \Omega$.  Using the maximal principle and the plurifine locality, we have 
\begin{align*}   
  \boldsymbol{1}_{  \Omega} \theta_{v_j^{k,C}}^n \geq \boldsymbol{1}_{\{v_j^k>V_{\theta}-C\}\cap \Omega} \theta_{v_j^{k }}^n\geq\boldsymbol{1}_{\{v_j^k>V_{\theta}-C\}\cap \Omega} \bigl(\inf_{l\geq j} f_l\bigr)\mu.
\end{align*}
 
  We note  that on the open set $\Omega $,  the functions  $v_j^{k,C}$ and $v_j^C$ are locally bounded. By letting $k\to +\infty$, it then follows from  \cite[Theorem 4.26]{GZ}    that 
  $$  \boldsymbol{1}_{  \Omega} \theta_{v_j^{ C}}^n  \geq\boldsymbol{1}_{\{v_j >V_{\theta}-C\}\cap \Omega} \left(\inf_{l\geq j} f_l\right)\mu $$
  Set $\varphi^C\coloneqq\max(\varphi,V_{\theta}-C)$. Because $v_j\searrow\varphi$ and $\varphi^C$ is locally bounded on $\Omega$, letting $j\to+\infty$ and applying \cite[Theorem 4.26]{GZ} again yields  $$\boldsymbol{1}_{  \Omega} \theta_{\varphi^{ C}}^n  \geq\boldsymbol{1}_{\{\varphi >V_{\theta}-C\}\cap \Omega} f\, \mu.$$
  Multiplying both sides by $\boldsymbol{1}_{\{\varphi >V_{\theta}-C\} }$ and using the plurifine locality, we have $$\boldsymbol{1}_{\{\varphi >V_{\theta}-C\} }\theta_{\varphi }^n \geq\boldsymbol{1}_{\{\varphi >V_{\theta}-C\}\cap \Omega} \theta_{\varphi^{ C}}^n  \geq\boldsymbol{1}_{\{\varphi >V_{\theta}-C\}\cap \Omega} f\, \mu.$$
 Finally, letting $C\to+\infty$ and using  that $\mu$ is non-pluripolar gives the desired inequality.
\end{proof}
 Similarly, we solve the complex Monge–Ampère equation with an exponential twist.
\begin{theorem}\label{minimal_nonp_lambda}
Let $\lambda>0$ be a constant and  $\mu$ be a positive non-pluripolar Radon    measure. Then there exists     a unique $\varphi\in\mathcal{E}(X,\theta)$   such that    $\theta_{\varphi}^n= e^{\lambda\varphi}\mu.$   
\end{theorem}
\begin{proof}
The uniqueness of $\varphi$ follows from Corollary~\ref{cor_domi}. 
Without loss of generality, assume $\lambda=1$. We first consider the case  where $\mu \le A(\omega_X + dd^c \psi)^n$ for some bounded $\omega_X$-psh function $\psi$.
Let  $U_{\alpha}$, $\rho_{\alpha}$, $g_{\alpha}$,  $\chi_j$ and 
 $\mu_j$ be as in the proof of Theorem \ref{minimal_nonp_0}.For each $j$, by \cite[Theorem 4.5]{BGL25} there exists a $\theta$-psh function $\varphi_j \simeq V_{\theta}$ such that    $\theta_{\varphi_j}^n= e^{\varphi_j}\mu_j.$  
 
Let $c_j\coloneqq\sup_{X}\varphi_j$ and $v_j\coloneqq\varphi_j-\sup_{X}\varphi_j$. Using Jensen's inequality we have
\begin{equation*} \textstyle \ove(\theta) \ge \int_X \theta_{\varphi_j}^n 
    = e^{c_j} \int_X e^{v_j}\, d\mu_j 
    \ge e^{c_j} \exp\!\Bigl(\frac{\int_X v_j\, d\mu_j}{\int_X d\mu_j}\Bigr). \end{equation*}
Since $\sup_Xv_j=0$, an argument analogous to the one in \eqref{convolution_bdd}   shows that $\int_X (-v_j)d\mu_j$ is uniformly bounded. Thus we obtain that $c_j$ is bounded from above,  
while the lower bound of $c_j$ follows from the inequality $\und(\theta)\leq e^{c_j}\int_X d\mu_j$.
Moreover, $\theta_{\varphi_j}^n \le C\mu_j$ for a uniform constant $C>0$.
    We can now repeat the argument in the proof of Theorem~\ref{minimal_nonp_0} and conclude that, after passing to a subsequence, $\varphi_j$ converges in capacity to a function $\varphi \in \mathcal{E}(X,\theta)$, and $\theta_{\varphi_j}^n \to \theta_{\varphi}^n$ weakly.

Now  we  claim that $e^{\varphi_j}\mu_j\to e^{\varphi}\mu$. weakly.
 Indeed, since   $e^{\varphi_j}$ are uniformly bounded quasi-psh functions, \cite[Lemma 11.5]{GZ} implies  
  $$\lim_{j\to+\infty}\int_X|e^{\varphi_j}-e^{\varphi}|\,d{\mu}=0.$$
Repeating the argument of (\ref{convolu_conveger}) with $\varphi_j$ replaced by $e^{\varphi_j}$ and $\varphi $ replaced by $e^{\varphi}$, we also have 
            $\lim_{j\to+\infty}\int_X|e^{\varphi_j}-e^{\varphi}|\,d{\mu_j}=0.$  
        It thus remains to show that 
\[
\lim_{j\to+\infty} \int_X \rho e^{\varphi} \, (d\mu_j - d\mu) = 0
\]
for every strictly positive  $\rho \in C^{\infty}(X)$. Writing $\rho = e^{\log \rho}$ and absorbing $\log \rho$ into $\varphi$, we may assume $\rho = 1$. By the definition of   convolution,  we have
\begin{align*}
 \lim_{j\to+\infty}   \int_{X} e^{\varphi } d\mu_j & = 
    \lim_{j\to+\infty} \sum_{\alpha} \int_{U_{\alpha}}  \big(  (g_{\alpha}+ e^{\varphi })*\chi_j-g_{\alpha}  *\chi_j    \big) \,\rho_{\alpha}\, d \mu
    \\& =  \sum_{\alpha} \int_{U_{\alpha}}  \big(  g_{\alpha}+ e^{\varphi } -g_{\alpha}     \big) \,\rho_{\alpha}\, d \mu
    = \int_{X} e^{\varphi } d\mu.
\end{align*}
Here we use the fact that $(g_{\alpha}+ e^{\varphi })*\chi_j\searrow (g_{\alpha}+ e^{\varphi })$ and that $g_{\alpha} *\chi_j \to g_{\alpha}$ uniformly. Thus the claim is proved and we obtain $\theta_{\varphi}^n = e^{\varphi}\mu$.

   For the general case, we write $\mu=g(\omega_X+dd^c\psi)^n$, where $\psi$ is a bounded $\omega_X$-psh function   and $g\in L^1\bigl(X, (\omega_X + dd^c\psi)^n\bigr)$.  Set $\mu_j=\min(g,j)\, (\omega_X+dd^c\psi)^n$. From what we have just proved, for each $j$ there exists $\varphi_j \in \mathcal{E}(X,\theta)$ such that   $\theta_{\varphi_j}^n=e^{\varphi_j}\mu_j.$ 
   
  Since $\mu_j \le \mu_k$ for $j \le k$, the domination principle implies that the sequence $\{\varphi_j\}$ is decreasing. Moreover, the estimate $\und(\theta) \le \theta_{\varphi_j}^n \le e^{\sup_X \varphi_j}\mu$ ensures that $\sup_X \varphi_j$ is uniformly bounded. It follows that $\varphi_j$ decreases  to a $\theta$-psh function $\varphi$. By the Lebesgue dominated convergence theorem, we have
\[  
e^{\varphi_j}g_j \to e^{\varphi}g \quad \text{in } L^1\bigl(X, (\omega_X+dd^c\psi)^n\bigr).
\]
 It then follows from  Lemma \ref{lem_2.8}  that \(  \theta_{\varphi}^n\geq e^{\varphi}\mu.\) 
Because $\theta_{\varphi_j}^n \le e^{\sup_X \varphi_1}\mu$, Theorem \ref{als_lem} implies $\varphi \in \mathcal{E}(X,\theta)$ and $\varphi_j \to \varphi$ in capacity.  
By the lower semicontinuity of the non‑pluripolar product, we conclude that 
$\theta_{\varphi}^n = e^{\varphi}\mu$, which completes the proof.
\end{proof}

\section{Monge-Amp\`ere Equations with Prescribed Singularity Type}\label{sec_6}

\subsection{\textbf{Solving complex  \MA\ type equations}}

\hspace{1.5em}
In the sequel, we fix $\phi$ to be a $\theta$-psh model potential.
The  criteria established in Section \ref{general_rel} enable us to construct subsolutions as described below.
\begin{proposition}\label{additive_sol}
 Let $\lambda > 0$ be a constant  and    $u,v \in \mathcal{E}(X,\theta,\phi)$.  Set 
    \[
    \mu \coloneqq e^{-\lambda u}\theta_u^n + e^{-\lambda v}\theta_v^n .
    \]
    Assume that $\theta_v^n \le A(\omega_X+dd^c\psi)^n$ for some bounded $\omega_X$-psh function $\psi$ and some constant $A>0$. 
    Then there exists a unique $\varphi \in \mathcal{E}(X,\theta,\phi)$ such that 
    \(
    \theta_{\varphi}^n = e^{\lambda\varphi}\mu .
    \) 
\end{proposition}
Here we do not assume that $\mu$ has finite total mass. Indeed, we will see that $\varphi \le \min(u, v)$ and $e^{\lambda \varphi} \mu$ is therefore a Radon measure. 
\begin{proof}   Uniqueness follows directly from the domination principle (Corollary \ref{cor_domi}). Without loss of generality, we set $\lambda = 1$.

\textbf{Step 1.}   We begin with the case \( \phi = V_{\theta} \) and \( u, v \simeq V_{\theta} \).  
 Set $u_j \coloneqq \max(u,-j)$ and $v_j \coloneqq \max(v,-j)$, and define 
   $$\mu_j\coloneqq   e^{-u_j}\theta_{u}^n+e^{-v_j}\theta_{v}^n.$$ 
As the function  $u_j$ is   bounded, $\mu_j$ is then a non-pluripolar Radon measure. By Theorem \ref{minimal_nonp_lambda}, there exists a unique $\varphi_j\in\mathcal{E}(X,\theta)$ such that 
     $$\theta_{\varphi_j}^n=e^{\varphi_j}\mu_j.$$

    Let $C>0$ satisfy $|u-v| \le 2C$, and consider  $w\coloneqq\frac{u+v}{2} -C-n\log2  $, which has   minimal singularities.  
Observe that  $$\theta_w^n\geq e^w\mu_j,$$
so by the domination principle (Corollary~\ref{cor_domi}) we have $\varphi_j \ge w$.
  Also, for $j \ge k$, we have $\theta_{\varphi_j}^n = e^{\varphi_j}\mu_j \ge e^{\varphi_j}\mu_k=e^{\varphi_j-\varphi_k}\theta_{\varphi_k}^n$.  Applying the domination principle once more yields $\varphi_j \le \varphi_k$. Hence, the sequence $\{\varphi_j\}$ is decreasing and we let $\varphi \coloneqq \lim_j \varphi_j$ denote its decreasing limit.
  
  Moreover,  note that
   $e^{\varphi_j-u_j}\leq e^{\varphi_1-u}$ and $e^{\varphi_j-v_j} \leq  +e^{\varphi_1-v}.$   
 Because  $\varphi_1,u$ and $v$ have minimal singularities,  
  the Lebesgue dominated convergence theorem implies \[
e^{\varphi_j-u_j} \to e^{\varphi-u} \quad \text{in } L^1(X,\theta_u^n),
\qquad
e^{\varphi_j-v_j} \to e^{\varphi-v} \quad \text{in } L^1(X,\theta_v^n).
\] Applying Lemma~\ref{lem_2.8} yields $\theta_{\varphi}^n \ge e^{\varphi}\mu$,  while the reverse inequality follows from the   lower \mbox{semicontinuity} of non-pluripolar product. 
  Consequently, \( \theta_{\varphi}^n = e^{\varphi} \mu \).

\textbf{Step 2.} Next, we deal with the case   \( u, v \simeq \phi \).  
 Let $u_j\coloneqq\max(u,V_{\theta}-j)$ and $v_j\coloneqq\max(v,V_{\theta}-j)$. Denote 
    $$\mu_j\coloneqq   e^{-u_j}\theta_{u_j}^n+e^{-v_j}\theta_{v_j}^n.$$  
By Step~1 there exists a unique \(\varphi_j \in \mathcal{E}(X,\theta)\) satisfying \(\theta_{\varphi_j}^n = e^{\varphi_j}\mu_j\).  
 We choose \( C > 0 \) such that   $|u_j-v_j|\leq |u-v|\leq 2C.$   Consider the function $w_j\coloneqq\frac{u_j+v_j}{2} -C-n\log2  $, which has minimal singularities.  Then we have  $$\theta_w^n\geq e^w\mu_j.$$ By the domination principle (Corollary \ref{cor_domi}), we obtain $P_{\theta}(u_j,v_j)\geq\varphi_j\geq w_j$.
    After passing to a subsequence, we assume $   \varphi_j\to \varphi$ in $L^1(X)$. It follows  that  $P_{\theta}(u,v)\geq \varphi\geq \frac{u+v}{2} -C-n\log2.$  

  Next, we  claim that $$\lim_{j\to+\infty}\int_X | \varphi_j - \varphi|e^{\varphi} \theta_{\varphi_j}^n=0.$$ Once this is established, since $P_{\theta}(\inf_{l\geq j} \varphi_{ j})\simeq\phi$,   an argument similar to that in the proof of Theorem \ref{1-ener_thm} will show that $\varphi_j$ converges to $\varphi$ in capacity.  
  
 To prove the claim, note that   the function     $ |\varphi_j-\varphi|e^{(\varphi_j+\varphi)/4} $  is   uniformly bounded. Using    plurifine locality and the bounded mass property, we obtain
    \begin{align*}
\int_X |\varphi_j - \varphi| e^{\varphi} \, \theta_{\varphi_j}^n
&\le \int_{\{\min(u,v) > V_{\theta} - j\}} |\varphi_j - \varphi| e^{(\varphi + \varphi_j)/2} \bigl( \theta_u^n + \theta_v^n \bigr) \\
&\quad + \int_{\{\min(u,v) \le V_{\theta} - j\}} |\varphi_j - \varphi| e^{(\varphi + \varphi_j)/4 - j/2} \bigl( \theta_{u_j}^n + \theta_{v_j}^n \bigr) \\
&\le \int_X |\varphi_j - \varphi| e^{(\varphi + \varphi_j)/2} \bigl( \theta_u^n + \theta_v^n \bigr) + O(e^{-j/2}),
\end{align*}
where in the first inequality we have used the inequality $(\varphi + \varphi_j)/2 \le \min(u_j, v_j)$ and the fact that  $(\varphi + \varphi_j)/4 \le -j/2$ on the set $\{\min(u,v) \leq V_{\theta} - j\}$.  By \cite[Lemma 11.5]{GZ}, we have $\max(\varphi_j,-C)\to\max(\varphi,-C)$ in $L^1(X,\mu)$ for every   $C>0$.  After extracting a subsequence, we then have  $\varphi_j\to\varphi$ a.e. with respect to the non-pluripolar measure  $\theta_u^n+\theta_v^n$. Therefore,   the right-hand side converges to $0$ by the Lebesgue  dominated convergence theorem, proving the claim.

We now verify  that $\varphi\in\mathcal{E}(X,\theta,\phi)$ is the solution to $\theta_{\varphi}^n= e^{\varphi}\mu $.  For $\varepsilon>0$, consider the quasi-continuous approximation   \[
\chi_{t}^{\varepsilon} \coloneqq 
\frac{\max\!\bigl(P_{\theta}(\inf_l \varphi_l)+t,\;0\bigr)}
{\max\!\bigl(P_{\theta}(\inf_l \varphi_l)+t,\;0\bigr)+\varepsilon},
\]
  of the characteristic function $\boldsymbol{1}_{ \{P_{\theta}(\inf_l \varphi_l)>-t\}}$.  Set $\varphi_j^t\coloneqq \max(\varphi_j,-t)$ and  $\varphi^t\coloneqq \max(\varphi,-t)$.
  
  Note that $\{P_{\theta}(\inf_l \varphi_l)>-t\}\subseteq\{\min(u,v)>-t\}  $. By the plurifine locality,  we have $$ \chi_{t}^\varepsilon\theta_{\varphi_j}^n=\chi_{t}^\varepsilon e^{\varphi_j}\mu$$  whenever $j\geq t$.   For any non-negative test  function $\rho\in C^{\infty}(X)$,  we deduce from \cite[Theorem 4.26]{GZ} that
 \begin{align}\label{final_weakcon}\int_X \rho \chi_{t}^{\varepsilon} \theta_{\varphi}^n
&= \int_X \rho \chi_{t}^{\varepsilon} \theta_{\varphi^t}^n
 = \lim_{j\to+\infty} \int_X \rho \chi_{t}^{\varepsilon} \theta_{\varphi_j^t}^n
 = \lim_{j\to+\infty} \int_X \rho \chi_{t}^{\varepsilon} \theta_{\varphi_j}^n \nonumber\\
&= \lim_{j\to+\infty} \int_X \rho \chi_{t}^{\varepsilon} e^{\varphi_j} \mu
 = \int_X \rho \chi_{t}^{\varepsilon} e^{\varphi} \mu,\end{align} 
where the last equality follows from the monotone convergence theorem. Indeed, for $j \ge t$, on the set $\{ P_{\theta}(\inf_l \varphi_l) > -t \}$, we have 
\[
\varphi_j \ge P_{\theta}\bigl(\inf_{l\ge j}\varphi_l\bigr) \nearrow \varphi,
\qquad
\varphi_j \le \operatorname{usc}\bigl(\sup_{l\ge j}\varphi_l\bigr) \searrow \varphi,
\quad\text{and}\quad
\operatorname{usc}\bigl(\sup_{l\ge j}\varphi_l\bigr) \le \min(u,v ).
\]

Because $\varphi \le \min(u,v)$,  $e^{\varphi}\mu$ is a non‑pluripolar Radon measure. Therefore it puts no mass on the pluripolar set $\{P_{\theta}(\inf_l \varphi_l)=-\infty\}$. By letting $\varepsilon \to 0$ then $t \to +\infty$ in \eqref{final_weakcon}, we obtain $\int_X \rho \, \theta_{\varphi}^n = \int_X \rho \, e^{\varphi} \mu$ for all non-negative $\rho \in C^\infty(X)$, which implies $\theta_{\varphi}^n = e^{\varphi}\mu$.

\textbf{Step 3.} 
   We now deal with the   case where $u\in \mathcal{E}(X,\theta,\phi)$  and $v\simeq \phi$.    Set $u_j \coloneqq \max(u,\phi-j)$ and define 
   $$\mu_j\coloneqq   e^{-u_j}\theta_{u_j}^n+e^{-v}\theta_{v}^n.$$ 
  By Step 2, there exists a unique $\theta$-psh function  $\varphi_j\simeq \phi  $  such that 
     $$\theta_{\varphi_j}^n=e^{\varphi_j}\mu_j.$$
Since $\theta_{\varphi_j}^n\geq e^{\varphi_j-u_j}\theta_{u_j}^n $ and $\theta_{\varphi_j}^n\geq e^{\varphi_j-v}\theta_{v}^n$, the domination principle (Corollary \ref{cor_domi}) yields that  $$\varphi_j\leq  P_{\theta}(u_j,v).$$

Fix   $a>0$ and let $b>1$ be the unique  number satisfying $(1-1/b)^n=e^{-a}$.  Consider the function $\psi_j \coloneqq P_{\theta}(b\varphi_j-(b-1)u_j)$ and  denote the contact set $C_j\coloneqq\{\psi_j=b\varphi_j-(b-1)u_j\}$. By Proposition~\ref{contact_0} and the minimum  principle, we have   $$  (1-1/b)^n  \boldsymbol{1}_{C_j}\theta^n_{u_j}+(1/b)^n\theta_{\psi_j}^n\leq \boldsymbol{1}_{C_j}\theta^n_{\varphi_j}=\boldsymbol{1}_{C_j} e^{\varphi_j-u_j}\theta_{ u_j}^n+\boldsymbol{1}_{C_j}e^{\varphi_j-v}\theta_{v}^n  $$ 
Set $D_{j}\coloneqq \{\psi_j\leq u-ba\}$. Because $D_j\cap C_j\subset\{\varphi_j\leq u_j-a\}$ and $v\geq u_j-C$ for a uniform constant $C>0$, we obtain that   
        \begin{align*}
             \boldsymbol{1}_{D_j}\theta^n_{\psi_j} \leq b^n   e^{\varphi_j-u_j+C}\boldsymbol{1}_{C_j}\theta_{v}^n\leq  b^n  e^{(\psi_j-u_j)/b+C} \theta_{v}^n.
        \end{align*}
Now put $w_j^k \coloneqq P_{\theta}(\psi_j, \ldots, \psi_{j+k})$ and $D_{j}^k\coloneqq \{w_j^k\leq u -ba\}$. By the minimum  principle,  
 \begin{align*}
      \boldsymbol{1}_{D_j^k}\theta^n_{w_j^k} \leq  e^C b^n  e^{(w_j^k-u_{j+k})/b }\theta_{v}^n.
  \end{align*}
Since the function $xe^{-x/b}$ is bounded for $x\geq 0$,  we  observe that 
   
\begin{align*}
   \sup_{k\in\mathbb{N}} \int_X |w_{j}^k- u_{j+k}| {1}_{D_j^k}\theta_{w_{j}^k}^n&\leq 
   e^C b^n  \sup_{k\in\mathbb{N}}\int_X |w_{j}^k- u_{j+k}|  e^{(w_j^k-u_{j+k})/b }\theta_{v}^n <+\infty.
\end{align*}

By Lemma \ref{diam_lem}, we know that $\widehat{w}_j^k \coloneqq P_{\theta}(u,w_j^k)\in\mathcal{E}(X,\theta,\phi)$. 
   Let $\widehat{D}_j^k \coloneqq \{\widehat{w}_j^k \le u-ba\}$.  It then follows from the minimum  principle   that
   $$  \sup_{k\in\mathbb{N}} \int_X|\w w_j^k-u|\, \boldsymbol{1}_{\w D_j^k}\theta_{\w w_j^k}^n     \leq \sup_{k\in\mathbb{N}} \int_X|  w_j^k-u_{j+k}|\, \boldsymbol{1}_{ D_j^k}\theta_{ w_j^k}^n  <+\infty$$
   
Therefore,    Proposition \ref{1-ener_pro}  implies that $ \w w_j\coloneqq \lim_k \w w_j^k$ is not identically $-\infty$ and belongs to $\mathcal{E}(X,\theta,\phi)$. Observe that $$\w w_j^k \leq P_{\theta}(\psi_j\ldots,\psi_{j+k})\leq P_{\theta}(\varphi_j\ldots,\varphi_{j+k})\leq \varphi_{j+k}\leq u_{j+k}.$$ Up to extracting a subsequence, we  assume $\varphi_j\to \varphi$  in $L^1(X)$.      Since $u_{j+k} \searrow u \in \mathcal{E}(X,\theta,\phi) $   as \( k \to +\infty \), we infer that both $P_{\theta}(\inf_{l\geq j}\varphi_l) $ and $\varphi  $   belong to $\mathcal{E}(X,\theta,\phi)$.

  Following the same argument as in Step 2, 
   we can show  that $\lim_j\int_X | \varphi_j - \varphi|e^{\varphi} \theta_{\varphi_j}^n=0$, which implies that    $\varphi_j$ converges to $\varphi$ in capacity.   
Similarly, it can also be verified that $\varphi \in \mathcal{E}(X, \theta, \phi)$ indeed solves the equation $\theta_{\varphi}^n = e^{\varphi}\mu$.

\textbf{Step 4.} 
  Finally, we treat the case where $u,v \in \mathcal{E}(X,\theta,\phi)$ and      $\theta_v^n\leq A(\omega_X+dd^c\psi)^n, $  for some bounded $\omega_X$-psh function $\psi$ and some constant $A>0$. As in Step 3, we let $u_j\coloneqq\max(u,\phi-j)$ and  define  $\mu_j\coloneqq   e^{-u_j}\theta_{u_j}^n+e^{-v}\theta_{v}^n$.

  By Step 3, there exists a unique $\theta$-psh function  $\varphi_j\in\mathcal{E}(X,\theta,\phi)  $ such that 
      $\theta_{\varphi_j}^n=e^{\varphi_j}\mu_j.$ 
 Using the domination principle (Corollary \ref{cor_domi}) again,  we have   $\varphi_j\leq  P_{\theta}(u_j,v)$.  Since the function $xe^{-x}$ is bounded for $x\geq 0$,  we obtain
    $$ \sup_{j\in\mathbb{N}} \int_X|\varphi_j-\min(u_j,v)|\, \theta_{\varphi_j}^n \leq \sup_{j\in\mathbb{N}} \int_X \left(|\varphi_j- u_j|  e^{\varphi_j-u_j}\theta_{u_j}^n+ |\varphi_j- v|  e^{\varphi_j-v}\theta_{v}^n \right)< +\infty $$
  Define  $\varphi_j'\coloneqq P_{\theta}(u,v,\varphi_j)$.   Lemma \ref{diam_lem} gives that  $\varphi_j'\in \mathcal{E}(X,\theta,\phi)$. Since $P_{\theta}(u_j,v)\geq P_{\theta}(u,v)$,  the minimum  principle yields 
   \begin{align}\label{similar_enve}  \sup_{j\in\mathbb{N}} \int_X|\varphi_j'-P_{\theta}(u,v)|\, \theta_{\varphi_j'}^n     \leq\sup_{j\in\mathbb{N}} \int_X|\varphi_j-\min(u_j,v)|\, \theta_{\varphi_j}^n  <+\infty\end{align}
We thus infer from Proposition \ref{1-ener_pro}  that $\sup_{X}\varphi_j$ is bounded from below.
Moreover, given $a>0$, we take $b>1$ with $(1-1/b)^n=e^{-a}$. Denote $\psi_j \coloneqq P_{\theta}(b\varphi_j-(b-1)u_j)\in \mathcal{E}(X,\theta,\phi)$ and  set   $D_{j}\coloneqq \{\psi_j\leq u-ba\}$. By  Proposition  \ref{contact_0} and the minimum  principle, we have   
        \begin{align}\label{uni_byv}
             \boldsymbol{1}_{D_j}\theta^n_{\psi_j} \leq b^n   e^{ \varphi_j-v} \theta_{v}^n\leq b^n    \theta_{v}^n.
        \end{align}
It follows that      \begin{align*}
    \sup_{j\in\mathbb{N}}\int_X |\psi_{j} - u_{j}| {1}_{D_j}\theta_{\psi_{j}}^n&\leq b^{n+1} \sup_{j\in\mathbb{N}}
       \int_X |\varphi_j- u_{j}|\,  \theta_{v}^n.
\end{align*}
 where the finiteness of the right-hand side follows from the Chern-Levine-Nirenberg inequality, because  $\sup_{X}\varphi_j$ and $\sup_{X}u_j$ are uniformly bounded. Repeating the argument leading to \eqref{similar_enve} we also see that $\sup_X\psi_j$ is bounded below.
 
Hence, by passing to a subsequence, we may assume  $\psi_j \to \psi$ and $\varphi_j \to \varphi$ in $L^1(X)$. Now we apply Proposition \ref{general_asl} to inequality \eqref{uni_byv}, which gives $P_{\theta}\bigl(\inf_{l\ge j}\psi_l\bigr) \in \mathcal{E}(X,\theta,\phi)$. Thus we have $P_{\theta}\bigl(\inf_{l\ge j}\varphi_l\bigr) \in \mathcal{E}(X,\theta,\phi)$. The rest of the argument then follows as in Step~2, yielding $\theta_{\varphi}^n = e^{\varphi}\mu$, which concludes
the proof.
\end{proof}
 
Using the result above, we adapt the supersolution technique employed in  \cite[Theorem 5.4]{ALS25} to solve the Monge–Ampère equation with prescribed singularities.
\begin{theorem}\label{MA_prescribed}
   Fix $\lambda>0$,  and let $\mu$ be a positive non-pluripolar Radon measure. Then there is a unique $\varphi\in\mathcal{E}(X,\theta,\phi)$ such that $$\theta_{\varphi}^n=e^{\lambda\varphi}\mu.$$
\end{theorem}
\begin{proof}
  We may assume 
$ \lambda=1$. The uniqueness follows directly from the domination  principle (Corollary \ref{cor_domi}). 
We first consider the case $\mu\leq A(\omega_X+dd^c\psi)^n$ where $ \psi $ is a bounded $\omega_X$-psh function and $A>0$ is a constant.  By Theorem \ref{minimal_nonp_lambda}, there exists  $\psi\in\mathcal{E}(X,\theta)$ solving the equation  $\theta_{\psi}^n=e^{\psi}\mu$. Proposition~\ref{contact_2} then  shows that $\psi'\coloneqq P_{\theta}[\phi](\psi)$ is a supersolution, i.e.,  $\theta_{\psi'}^n\leq e^{\psi'}\mu$.    
Now we consider the following  family of supersolutions 
  $$ \mathcal{T}\coloneqq\{w  \in\mathcal{E}(X,\theta,\phi):   e^{-\psi'}\theta_{\psi'}^n\leq e^{-w}\theta_{w}^n\leq \mu\}$$
Take a sequence $u_j\in\mathcal{E}(X,\theta,\phi)$ satisfying  
   $$ \int_X e^{-u_j}\theta_{u_j}^n\nearrow \sup_{w\in\mathcal{T}} \int_X e^{-w}\theta_{w}^n.$$
From the estimate $\und(\theta,\phi)\le e^{\sup_X u_j}\mu$ we deduce that  $\sup_X u_j$ is bounded  from below in $j$.  The domination principle also implies  that $u_j\leq \psi'$. After passing to a subsequence, we may therefore assume that $u_j\to u$ in $L^1(X)$. Note that  $\theta_{u_j}^n\leq e^{\sup_{X}\psi'}\mu$.  Then by Theorem \ref{als_lem}, we know that 
$u \in \mathcal{E}(X,\theta,\phi)$, $u_j\to u$ in capacity and that $\theta_{u_j}^n \to \theta_{u}^n$ weakly. Meanwhile, after extracting a further subsequence, there exists a positive Radon measure $\nu\leq\mu$ such that $e^{-u_j}\theta_{u_j}^n\to\nu$ weakly.
Since $e^{-u_j}\theta_{u_j}^n \leq  A(\omega_X+dd^c\psi)^n$,    \cite[Lemma 2.5]{darvas2020relative} implies that    $e^{u_j}(e^{-u_j}\theta_{u_j}^n)$ converges to $e^u\nu$ weakly. Hence 
    $$\theta_u^n=e^u\nu, \quad \text{and}\,\, \int_Xe^{-u}\theta^n_u=\sup_{w\in\mathcal{T}} \int_X e^{-w}\theta_{w}^n. $$
    
  We claim that $\nu=\mu$. If not, arguing  as before, we can construct   a supersolution $v\in\mathcal{E}(X,\theta,\phi)$    such that $e^{-v}\theta_v^n\leq\mu-\nu$.  By Proposition~\ref{additive_sol}, there exists a function $\varphi\in\mathcal{E}(X,\theta,\phi)$ such that    $$e^{-\varphi}\theta_{\varphi}^n=e^{-u}\theta_u^n+e^{-v}\theta_v^n.$$
  But then, we would have $\int_Xe^{-\varphi}\theta_{\varphi}^n>e^{-u}\theta_{u}^n$, which contradicts the maximality of $\int_X e^{-u}\theta_u^n$.

  For the general case, write $\mu=g(\omega_X+dd^c\psi)^n$, where $ \psi $ is a bounded $\omega_X$-psh function and    $g\in L^1\bigl(X, (\omega_X + dd^c\psi)^n\bigr)$. Set $g_j\coloneqq\min(g,j)$. The preceding results yield $\varphi_j\in\mathcal{E}(X,\theta,\phi)$ satisfying
     $$\theta_{\varphi_j}^n=e^{\varphi_j}g_j(\omega_X+dd^c\psi)^n$$
By the domination principle, the sequence $\varphi_j$ is decreasing. Moreover, the inequality $\und(\theta,\phi)\leq e^{\sup_X\varphi_j}\,\mu$ implies that $\sup_X\varphi_j$ is bounded from below.  We therefore set $\varphi \coloneqq \lim_j \varphi_j$ as the decreasing limit.  Repeating the same argument as in the proof of Theorem~\ref{minimal_nonp_lambda}, we finally obtain $\theta_{\varphi}^n = e^{\varphi}\mu$, which completes the proof.
\end{proof}

Weak solutions to the Monge-Ampère equations with $\lambda = 0$ can be obtained via a classical perturbation argument.
\begin{theorem}\label{MA_pre_0}
 Let    $\mu$ be a positive non-pluripolar Radon  measure. Then there exist        $\varphi\in\mathcal{E}(X,\theta,\phi)$ and a unique constant $c>0$ such that    $\theta_{\varphi}^n= c\mu.$ 
\end{theorem}
\begin{proof}
  The uniqueness of $c$ follows from the domination principle (Corollary \ref{cor_domi}).
  By the theorem above, for each $j \ge 1$ there exists $\varphi_j \in \mathcal{E}(X,\theta,\phi)$ satisfying
       $\theta_{\varphi_j}^n= e^{\varphi_j/j}\mu.$ 
       
      Set $c_j\coloneqq e^{\sup_X\varphi_j/j}$ and $v_j\coloneqq \varphi_j-\sup_X \varphi_j$. 
Observe that
    \[
    \theta_{\varphi_j}^n = e^{\varphi_j/j - \varphi_1}\theta_{\varphi_1}^n
    \ge e^{\left(\varphi_j - \varphi_1 - (j-1)\sup_X\varphi_1\right)/j}\theta_{\varphi_1}^n .
    \]
    Hence the domination principle gives $\varphi_j \le \varphi_1 + (j-1)\sup_X\varphi_1$, and therefore $c_j \le c_1$.  After passing to a subsequence, we may assume that $v_j \to v$ in $L^1$ and $c_j \to c \ge 0$.   Because $\theta_{v_j}^n \le c_1\mu$, Theorem~\ref{als_lem} implies that $v \in \mathcal{E}(X,\theta,\phi)$ and that $v_j$ converges  to $ v$ in capacity. Moreover, since $v_j/j \to 0$, it follows from \cite[Lemma 11.5]{GZ} that $c_j e^{v_j/j} \to c$ in $L^1(X,\mu)$. Consequently, by Lemma~\ref{lem_2.8} and the lower semicontinuity of the non-pluripolar product, we obtain $\theta_v^n = c\,\mu$. This completes the proof.
\end{proof}

\begin{theorem}\label{thm_sum}
 
Let $u,v\in\mathcal{E}(X,\theta,\phi)$. Fix $\lambda>0$ and set $\mu \coloneqq e^{-\lambda u}\theta_{u}^n+e^{-\lambda v}\theta_{v}^n$. Then there exists a unique $\varphi\in\mathcal{E}(X,\theta,\phi)$ such that $$\theta_{\varphi}^n=e^{\lambda\varphi}\mu.$$
\end{theorem}
\begin{proof}
Without loss of generality we take $\lambda=1$. Uniqueness follows     directly from the domination  principle (Corollary \ref{cor_domi}).
Define $u_j\coloneqq\max(u,-j)$ and $v_j\coloneqq\max(v,-j)$. We  set 
   $$\mu_j\coloneqq   e^{-u_j}\theta_{u}^n+e^{-v_j}\theta_{v}^n.$$ 
Since each $\mu_j$ is a non‑pluripolar measure, Theorem~\ref{MA_prescribed} implies that for every $j$, there exists   $\varphi_j \in \mathcal{E}(X,\theta,\phi)$ such that
$\theta_{\varphi_j}^n = e^{\varphi_j}\mu_j$.
Using the domination principle, we see that the sequence   $\{\varphi_j\}$ is decreasing.  Moreover, because the function $xe^{-x}$ is bounded for $x\ge0$, we have  
     \[
\sup_{j\in\mathbb{N}} \int_X \bigl|\varphi_j-P_{\theta}(u,v)\bigr| \boldsymbol{1}_{\{\varphi_j\le P_{\theta}(u,v)\}} \theta_{\varphi_j}^n 
\le \sup_{j\in\mathbb{N}} \int_X \Bigl( |\varphi_j-u_j| e^{\varphi_j-u_j}\theta_u^n 
                + |\varphi_j-v_j| e^{\varphi_j-v}\theta_v^n \Bigr) < +\infty .
\]
  By  Proposition \ref{1-ener_pro}, we obtain that $\sup_X \varphi_j$ is bounded from below and  that 
   $ \varphi_j\searrow\varphi\in \mathcal{E}(X,\theta,\phi)$.

We now claim that $\varphi\leq \min(u,v)$. The proof is quite similar to that in Proposition \ref{domina_1}. Fix $\delta>0$. Let  $u'_j\coloneqq\max(u,\varphi_j-\delta)$ and $u' \coloneqq\max(u,\varphi -\delta)$. Then we have  $\boldsymbol{1}_{\{-j<u<u'_j\}}\theta^n_{u}\leq e^{-\delta}\theta_{u'_j}^n$.  
For each $b>1$, we set $u_{j,b}\coloneqq P_{\theta}(bu-(b-1)u'_j)$ and $u_{ b}\coloneqq P_{\theta}(bu-(b-1)u' )$.  Choosing $b$ large enough so that $(1-1/b)^n\ge e^{-\delta}$ gives
   $$\boldsymbol{1}_{\{-j<u<u'_{j}\}}\theta_{u_{j,b}}^n=0.$$
If the set $\{u<u'-\varepsilon\}$ is non‑empty for some $\varepsilon>0$, then $\sup_X u_b\to-\infty$ as $b\to+\infty$.  Observe that $\boldsymbol{1}_{\{ u<u'_{j}\}}\theta_{u_{j,b}}^n= \boldsymbol{1}_{\{ u<u_{j,b}\}}\theta_{u_{j,b}}^n$ and that $u_{j,b}\leq u$. Therefore by Proposition~\ref{contact_0} and the minimum principle, we   have 
\begin{align*}
  0< \und(\theta,\phi) \leq \int_{X}\theta_{u_{j,b}}^n \leq \int_{ \{u=u_{j,b}\}}\theta_{u}^n +\int_{\{u\leq-j\}}\theta_{u_{j,b}}^n \\ \leq
  \int_{ \{  u\leq u_b\}}\theta_{u}^n +b^n\int_{\{u\leq-j\}}\theta_{u}^n
\end{align*}
Letting $j\to+\infty$ and then $b\to+\infty$ yields that $\und(\theta,\phi)=0$,     a contradiction. Because $\delta>0$ was arbitrary, it then follows that $u \ge \varphi$. The same argument also yields $v \ge \varphi$. 

Finally, we verify that $\varphi$ satisfies $\theta_{\varphi}^n = e^{\varphi}\mu$.   Let    $$\chi_{t}^\varepsilon\coloneqq \frac{\max(\varphi+t,0)}{\max( \varphi+t,0)+\varepsilon}.$$ 
   Set $\varphi_j^t\coloneqq \max(\varphi_j,-t)$ and  $\varphi^t\coloneqq \max(\varphi,-t)$.
 By   plurifine locality,  $  \chi_{t}^\varepsilon\theta_{\varphi_j}^n=\chi_{t}^\varepsilon e^{\varphi_j}\mu$ for  $j\geq t$.   Now let $\rho$ be any non‑negative continuous function. Applying \cite[Theorem 4.26]{GZ}, we obtain
 \begin{align*}\int_X \rho \chi_{t}^{\varepsilon} \theta_{\varphi}^n
&= \int_X \rho \chi_{t}^{\varepsilon} \theta_{\varphi^t}^n
 = \lim_{j\to+\infty} \int_X \rho \chi_{t}^{\varepsilon} \theta_{\varphi_j^t}^n
 = \lim_{j\to+\infty} \int_X \rho \chi_{t}^{\varepsilon} \theta_{\varphi_j}^n \\
&= \lim_{j\to+\infty} \int_X \rho \chi_{t}^{\varepsilon} e^{\varphi_j} \mu
 = \int_X \rho \chi_{t}^{\varepsilon} e^{\varphi} \mu,\end{align*} 
 where the last equality follows from the monotone convergence theorem, because on the set \( \{\varphi > -t\} \),  the functions $e^{-u}$ and $e^{-v}$ are bounded.  Note that $ \varphi\leq\min(u,v)$ and therefore $e^{\varphi}\mu$ is a non-pluripolar Radon measure.  
 Hence, by letting $\varepsilon\to 0$ and then $t\to +\infty$,  we deduce  that $\theta_{\varphi}^n=e^{\varphi}\mu$, which completes the proof. 
\end{proof}

\subsection{\textbf{Characterization of the range of the \MA\ operators}}
 \hspace{1.5em} We first consider the case where \(\mu\) has densities in \(L^p\). Let \(0\leq f \in L^{p}(X,\omega_X^n)\) with \(p > 1\) and \(\int_X f\,\omega_X^n > 0\).
Based on the relative \(L^{\infty}\) a priori estimates (Theorem~\ref{rel_infity}) and Theorems~\ref{MA_prescribed} and~\ref{MA_pre_0}, we adapt an argument in \cite[p. 56]{darvas2020relative} to derive the following result, which is included here for the
reader’s convenience.
\begin{corollary}\label{cor_6.1}
\noindent\textup{(i)} There exists a function  \(\varphi \in \psh(X,\theta)\) with $\varphi\simeq\phi$   and a unique constant \(c > 0\) such that      and
 $ \theta_{\varphi}^n = c f \omega_X^n $. 
 
\noindent\textup{(ii)} For any \(\lambda > 0\), there exists a unique \(\varphi \in \psh(X,\theta)\) such that  $\varphi\simeq\phi$   and
$\theta_{\varphi}^n = e^{\lambda\varphi} f \omega_X^n$. 
\end{corollary}
\begin{proof}
  By Theorems~\ref{MA_prescribed} and \ref{MA_pre_0}, for any \(\lambda\geq 0\), there exist \(\varphi\in\mathcal{E}(X,\theta,\phi)\) and \(c>0\) (in particular, \(c=1\) when \(\lambda>0\)) such that \(\theta_{\varphi}^n = ce^{\lambda\varphi} f \omega_X^n\).
  
Set \(\mu\coloneqq c f \omega_X^n\). By Skoda's   integrability theorem  (see, e.g., \cite[Theorem~8.11]{GZ})  and Hölder's inequality, there exist \(\varepsilon>0\) and    $1<q<p$ such that   \(e^{-\varepsilon u}f \in L^q(X,\omega_X^n)\) for any $\theta$-psh function  \(u \). Applying Theorems~\ref{MA_prescribed} and~\ref{MA_pre_0} again, 
we obtain functions $\varphi_j\in\mathcal{E}(X,\theta,\phi)$ satisfying 
\[
\theta^n_{\varphi_j} = \mathbf{1}_{\{\varphi>\phi-j\}} e^{(\lambda+\varepsilon)\varphi_j-\varepsilon \varphi}\mu.
\]

Observing that \(\theta^n_{\varphi_j} = \mathbf{1}_{\{\varphi>\phi-j\}} e^{(\lambda+\varepsilon)(\varphi_j- \varphi)}\theta_{\varphi}^n \leq e^{(\lambda+\varepsilon)(\varphi_j- \max(\varphi,\phi-j))}\theta_{\max(\varphi,\phi-j)}^n\), the  domination principle (Corollary~\ref{cor_domi})  then yields \(\varphi_j \geq \max(\varphi, \phi-j)\), and thus \(\varphi_j \simeq \phi\). 

Furthermore, applying the domination principle again shows that  the sequence \(\{\varphi_j\}\) is decreasing in \(j\). Let \(\tilde{\varphi}\coloneqq \lim_j \varphi_j\geq \varphi\) denote its limit. Since \(\theta_{\varphi_j}^n \leq e^{(\lambda+\varepsilon)\sup_X \varphi_1-\varepsilon \varphi}\mu\), the relative \(L^{\infty}\)  a prior estimates imply that \(\varphi_j \geq \phi-T\) for a uniform constant \(T>0\). 
By Theorem~\ref{als_lem} and the Lebesgue dominated convergence theorem, we  also have \(\theta_{\tilde{\varphi}}^n = e^{(\lambda+\varepsilon)\tilde{\varphi}-\varepsilon \varphi}\mu\), which implies  \(\tilde{\varphi} = \varphi\) by the domination principle.  Consequently, \(\varphi \geq \phi - T\) and therefore \(\varphi \simeq \phi\), which completes the proof.
\end{proof}

Using Corollary \ref{stability_lem}, we obtain the following stability result.

\begin{corollary}\label{cor_6.2}
Fix \(\delta,\lambda > 0\) and $p>1$. Let \(\phi_j, \phi \in \mathcal{S}_{\delta}(X,\theta)\) be model potentials such that \(d_{\theta}(\phi_j,\phi) \to 0\).   
Suppose $0 \leq f_j, f \in L^p(X, \omega_X^n)$ have uniformly bounded $L^p$-norms, satisfy $\int_X f_j \omega_X^n, \int_X f \omega_X^n > 0$, and $f_j\to f$ in $L^1$.  Let $\varphi_j, \varphi \in \psh(X, \theta)$ be such that $\varphi_j \simeq \phi_j$, $\varphi \simeq \phi$, and
\[
\theta_{\varphi_j}^n = e^{\lambda\varphi_j} f_j\, \omega_X^n,
\qquad
\theta_{\varphi}^n = e^{\lambda\varphi} f \,\omega_X^n.
\]
Then $\varphi_j$ converges   to $\varphi$ in capacity.
\end{corollary}

\begin{proof} Let $c_j\coloneqq\sup_{X}\varphi_j$ and $v_j\coloneqq\varphi_j-\sup_{X}\varphi_j$. Let $q>1$ be such that $1/p+1/q=1$. Using Jensen's inequality  and H\"older's inequality, we have
\begin{align*}   \textstyle \ove(\theta) &\ge \int_X \theta_{\varphi_j}^n 
    = e^{\lambda c_j} \int_X e^{\lambda v_j}\, f_j\omega_X^n 
     \\ &\ge e^{\lambda c_j} \exp\!\Bigl(\frac{\int_X \lambda v_j\, f_j\omega_X^n}{\int_X f_j\omega_X^n}\Bigr)   
    \geq e^{\lambda c_j} \exp\!\Bigl(\frac{-\|v_j\|_{L^q} \cdot \|f_j\|_{L^p} }{\|f_j\|_{L^1}}\Bigr) . \end{align*}
 Since $\sup_Xv_j=0$, the Chern-Levine-Nirenberg inequality   implies that $\|v_j\|_{L^q}$ is uniformly bounded. Thus,  $c_j$ is bounded from above.   
Conversely, the inequality $\und(\theta)\leq e^{\lambda c_j}\int_X f_j\,\omega_X^n$ ensures that $c_j$ is  bounded from below.  

 Hence, by passing to a subsequence, we may assume that $\varphi_j$ converges to a $\theta$-psh function $\ww\varphi$ both in $L^1$ and almost everywhere, and that $f_j$ converges to $f$ almost everywhere. Furthermore, the monotone convergence theorem implies that the function $g = \sup_{j \ge 1} f_j$ is integrable. 
 
 Note also that $\theta_{\varphi_j}^n\leq e^{\lambda\sup_jc_j}g\omega_X^n $, Corollary~\ref{stability_lem} yields that   $\ww \varphi\in\mathcal{E}(X,\theta,\phi)$. Moreover,  $\varphi_j$ converges in capacity to $\ww\varphi$ and   $\theta_{\varphi_j}^n$ converges  to $\theta_{\ww\varphi}^n$ weakly.  From the Lebesgue dominated convergence theorem, we also have $e^{\lambda\varphi_j} f_j\to e^{\lambda\ww\varphi } f $ in $L^1$. It follows that  $\theta_{\ww\varphi}^n=e^{\lambda\ww\varphi } f\omega_X^n$. The domination principle (Corollary \ref{cor_domi}) then gives $\varphi=\ww\varphi$. Consequently,  $\varphi_j$ converges in capacity to $ \varphi$, which finishes the proof. 
\end{proof}

We also study the range of the complex \MA\ operator on the   space  $\mathcal{E}^p(X,\theta,\phi)$.  
%For $p>0$,   define the relative finite $p$-energy space  by    $$ \mathcal{E}^p(X,\theta,\phi)\coloneqq\left\{u\in\mathcal{E}(X,\theta): \, \int_X|u-\phi|^p\,\theta_{u}^n<+\infty \right\}$$

\begin{theorem}\label{infinity_energy}
 Let \(\mu\) be a positive non-pluripolar measure satisfying  $$ 
\mu(E) \le A [\operatorname{Cap}_{\omega_X}(E)]^a $$ 
for some constants \(a, A > 0\) and for every Borel set \(E \subset X\).  Let  $p\geq1$. If $a >\frac{np}{n+p}$, 
 then   there exist       $\varphi\in \mathcal{E}^{p}(X,\theta,\phi)$ and a unique constant $c>0$ such that    $\theta_{\varphi}^n= c\mu.$    
 
 In particular, if $a\geq n$,   the   equation admits a solution  $\varphi\in \bigcap_{p\geq1}\mathcal{E}^{p}(X,\theta,\phi)$.
\end{theorem}
\begin{proof} The Chern-Levine-Nirenberg inequality implies that there exists a uniform constant $C>0$ such that for every $\theta$-psh function $u$ normalized by $\sup_X u=0$ and every $t>0$  (see, e.g.,  \cite[Proposition 8.35]{Dinew2019})
   $$\mathrm{Cap}_{\omega_X}( u<-t )\leq C/t .$$
 Hence $\mu\bigl(u < -t\bigr) \leq C/t^a$ for some uniform constant $C>0$. Fix any $a'\in(0,a)$. Note that   $$\int_X (-u)^{a'}d\mu=\int_0^{\infty}t^{a'-1}\mu(u<-t)\,dt<+\infty $$
    and we can therefore define
   \begin{align}\label{A_mu}   A_{a'}(\mu)\coloneqq \sup\left\{  \Bigl(\int_X (-u)^{a'}\,d\mu\Bigr)^{\frac{1}{m}}:  \varphi\,\, \text{is $\theta$-psh  with} \,\sup_X\varphi=0 \right\} .  \end{align}
Let $\mu_j$ be the regularized measures constructed  in the proof of Theorem \ref{minimal_nonp_0}  along with $U_{\alpha},\rho_{\alpha},g_{\alpha}$ and $\chi_j$.   By Corollary~\ref{cor_6.1}, for each $j$, there exists a  $\theta$-psh function $\varphi_j $  normalized by $\sup_X \varphi_j=0$  and a   constant $c_j>0$ such that $ \varphi_j \simeq \phi$ and $\theta_{\varphi_j}^n= c_j\mu_j.$   

Note that   $\ww\varphi_j\coloneqq  - (  - \varphi_j)^{\min(a',1)}$ are   $B\omega_X$-psh    for some $B>0$.    We may then   choose    \(g_{\alpha}\)  such that \(B\omega_X \le dd^c g_{\alpha}\).   Hence, for sufficiently large $j$, we have  
        \begin{align}\label{lower-ener} 
\int_{X} |\varphi_j - \phi|^{\min(a',1)} \, \theta^n_{\varphi_j} 
&\le C \sum_{\alpha} \int_{U_{\alpha}} \big( g_{\alpha} * \chi_j - (\ww\varphi_j + g_{\alpha}) * \chi_j \big) \rho_{\alpha} \, d\mu \nonumber\\
&\le C \sum_{\alpha} \int_{U_{\alpha}} \Big( (g_{\alpha} * \chi_j - g_{\alpha}) + (-\varphi_j)^{\min(a',1)} \Big) \, d\mu \leq C'.
\end{align}
We now  fix sufficiently small \(\alpha, \varepsilon > 0\) and let
\begin{align*}
\ww p  \coloneqq 
\Bigl(1 + \frac{\bigl(1-\frac{\varepsilon}{\min(a',1)}\bigr)
\bigl(\min(a',1)-\alpha\bigr)}{n+\varepsilon}\Bigr) a',
\quad
\ww m  \coloneqq \frac{1}{\min(a',1)} + \frac{1 - \frac{\varepsilon}{\min(a',1)}}{n+\varepsilon}. 
\end{align*}
 
 Using (\ref{A_mu}) and (\ref{lower-ener}), and following an argument similar to that of Proposition~\ref{p-energy}, 
we obtain that for every \(p' < \ww p \), there exists a uniform constant \(\ww C (\alpha , \varepsilon,p', A_{a'}(\mu) , \phi )> 0\) such that
\begin{align}\label{iteration}
\int_X |\varphi_j - \phi|^{p'} \, d\mu 
\;\leq\; 
\ww C\Bigl[ \Bigl(\int_X |\varphi_j -\phi|^{\min(a',1)} \, \theta_{\varphi_j}^n\Bigr)^{\ww m} + 1 \Bigr].
\end{align}
Using \cite[Proposition 3.5.1]{Kol98}, together with the argument from the proof of \cite[Corollary 3.1.4]{Kol98}, we obtain a uniform constant \(\ww A>0\) such that  
\[
\mu_j(E) \le \ww A\bigl[\operatorname{Cap}_{\omega_X}  (E)\bigr]^{a}
\quad  \text{for all  } E \subset X,\] 
because the capacity \(\operatorname{Cap}_{\omega_X}\) is locally equivalent to the Bedford-Taylor capacity (see, e.g., \cite[Proposition 8.30]{Dinew2019}).  Hence, replacing   
$d\mu$ by  $d\mu_j$ in (\ref{iteration}) leads to   $$\sup_{j\in\mathbb{N}}\int_X |\varphi_j -\phi|^{p'} \, \theta_{\varphi_j}^n <+\infty$$
 We now  consider the sequence $\{a_k\}$ defined recursively by     $a_0:=\min(a',1)$, $a_1:= p'$, and $$a_{k+1}\coloneqq \biggl(1 + \frac{\bigl(1-\frac{\varepsilon}{a_{k}}\bigr)
\bigl(a_{k}-\alpha\bigr)}{n+\varepsilon}\biggr) a', \quad \text{for} \, k\geq1.$$
  As we can take $a'$  sufficiently close to $a$ and $\alpha,\varepsilon$  sufficiently small, an iteration argument shows that for any $\gamma < \frac{na}{n-a}$, 
  $$\sup_{j\in\mathbb{N}}\int_X |\varphi_j -\phi|^{\gamma} \, \theta_{\varphi_j}^n <+\infty.$$

By hypothesis we have $\frac{na}{n-a} > p \geq 1$. We  can therefore  repeat the argument of the first part of Theorem~\ref{minimal_nonp_0} (in particular, for (\ref{argument_rep}), we estimate $\int_{\{\varphi_j \leq -t\}} \theta_{\varphi_j}^n$ by $\frac{1}{t^{\min(a',1)}} \int_{X} (-\varphi_j)^{\min(a',1)} \, \theta_{\varphi_j}^n$).   After passing to a subsequence, we    conclude that $\varphi_j$ converges  in capacity to a $\varphi\in\mathcal{E}(X,\theta,\phi)$   satisfying  $\theta_{\varphi }^n=c\mu$ for some $c>0$. 
 Moreover, by the lower semicontinuity of non-pluripolar product, we obtain  that $\int_X |\varphi - \phi|^{p} \, \theta_{\varphi }^n<+\infty $, i.e., $\varphi\in\mathcal{E}^{p}(X,\theta,\phi)$.
\end{proof}

\bibliographystyle{alphaJP}
\bibliography{main}

\newcommand{\etalchar}[1]{$^{#1}$}
\begin{thebibliography}{DDNL21b}

\bibitem[AGL23]{AGL23}
D.~Angella, V.~Guedj, and C.~H. Lu.
\newblock Plurisigned {Hermitian} metrics.
\newblock {\em Trans. Amer. Math. Soc.}, 376(7):4631--4659, 2023.

\bibitem[ALS24]{ALS24}
O.~Alehyane, C.~H. Lu, and M.~Salouf.
\newblock Degenerate complex {Monge}-{Amp{\`e}re} equations on some compact {Hermitian} manifolds.
\newblock {\em J. Geom. Anal.}, 34(10):320, 2024.

\bibitem[ALS25]{ALS25}
O.~Alehyane, C.~H. Lu, and M.~Salouf.
\newblock Monge-amp{\`e}re equations with prescribed singularities on compact {Hermitian} manifolds.
\newblock Preprint, {arXiv}:2511.02339, 2025.

\bibitem[BBE{\etalchar{+}}19]{BBEGZ_flow}
R.~J. Berman, S.~Boucksom, P.~Eyssidieux, V.~Guedj, and A.~Zeriahi.
\newblock K{\"a}hler-{Einstein} metrics and the {K{\"a}hler}-{Ricci} flow on log {Fano} varieties.
\newblock {\em J. Reine Angew. Math.}, 751:27--89, 2019.

\bibitem[BBGZ13]{BBGZ_varma}
R.~J. Berman, S.~Boucksom, V.~Guedj, and A.~Zeriahi.
\newblock A variational approach to complex {Monge}-{Amp{\`e}re} equations.
\newblock {\em Publ. Math., Inst. Hautes {\'E}tud. Sci.}, 117:179--245, 2013.

\bibitem[BD12]{BD12}
R.~Berman and J.-P. Demailly.
\newblock Regularity of plurisubharmonic upper envelopes in big cohomology classes.
\newblock In {\em Perspectives in analysis, geometry, and topology.}, pages 39--66. Basel: Birkh{\"a}user, 2012.

\bibitem[BDL17]{BDL17}
R.~Berman, T.~Darvas, and C.~H. Lu.
\newblock Convexity of the extended {K}-energy and the large time behavior of the weak {Calabi} flow.
\newblock {\em Geom. Topol.}, 21(5):2945--2988, 2017.

\bibitem[BEGZ10]{BEGZ10}
S.~Boucksom, P.~Eyssidieux, V.~Guedj, and A.~Zeriahi.
\newblock Monge-{A}mp\`ere equations in big cohomology classes.
\newblock {\em Acta Math.}, 205(2):199--262, 2010.

\bibitem[BGL25]{BGL25}
S.~Boucksom, V.~Guedj, and C.~H. Lu.
\newblock Volumes of bott–chern classes.
\newblock {\em Peking Mathematical Journal}, pages 1--43, 2025.

\bibitem[B{\l}o11]{blo_L2}
Z.~B{\l}ocki.
\newblock On the uniform estimate in the {Calabi}-{Yau} theorem. {II}.
\newblock {\em Sci. China, Math.}, 54(7):1375--1377, 2011.

\bibitem[BT76]{bt76}
E.~Bedford and B.~A. Taylor.
\newblock The {D}irichlet problem for a complex {M}onge-{A}mp\`ere equation.
\newblock {\em Invent. Math.}, 37(1):1--44, 1976.

\bibitem[BT82]{BT82}
E.~Bedford and B.~A. Taylor.
\newblock A new capacity for plurisubharmonic functions.
\newblock {\em Acta Math.}, 149(1-2):1--40, 1982.

\bibitem[BT87]{BT87}
E.~Bedford and B.~A. Taylor.
\newblock Fine topology, \v silov boundary, and {$(dd^c)^n$}.
\newblock {\em J. Funct. Anal.}, 72(2):225--251, 1987.

\bibitem[Ceg98]{Ceg98}
U.~Cegrell.
\newblock Pluricomplex energy.
\newblock {\em Acta Math.}, 180(2):187--217, 1998.

\bibitem[Che87]{Che87}
P.~Cherrier.
\newblock {\'E}quations de {Monge}-{Amp{\`e}re} sur les vari{\'e}t{\'e}s {Hermitiennes} compactes.
\newblock {\em Bull. Sci. Math., II. S{\'e}r.}, 111:343--385, 1987.

\bibitem[CK06]{CK06}
U.~Cegrell and S.~Ko{\l}odziej.
\newblock The equation of complex {Monge}-{Amp{\`e}re} type and stability of solutions.
\newblock {\em Math. Ann.}, 334(4):713--729, 2006.

\bibitem[Dan24]{Dang_herm}
Q.-T. Dang.
\newblock Hermitian null loci.
\newblock Preprint, {arXiv}:2404.01126, 2024.

\bibitem[Dar15]{Dar_finite}
T.~Darvas.
\newblock The {Mabuchi} geometry of finite energy classes.
\newblock {\em Adv. Math.}, 285:182--219, 2015.

\bibitem[Dar19]{Darvas_geom}
T.~Darvas.
\newblock Geometric pluripotential theory on {K{\"a}hler} manifolds.
\newblock In {\em Advances in complex geometry}, Contemp. Math., pages 1--104. Amer. Math. Soc., Providence, RI, 2019.

\bibitem[DDNL18a]{DDNL18mono}
T.~Darvas, E.~Di~Nezza, and C.~H. Lu.
\newblock Monotonicity of nonpluripolar products and complex {M}onge-{A}mp\`ere equations with prescribed singularity.
\newblock {\em Anal. PDE}, 11(8):2049--2087, 2018.

\bibitem[DDNL18b]{DDNLFULL}
T.~Darvas, E.~Di~Nezza, and C.~H. Lu.
\newblock On the singularity type of full mass currents in big cohomology classes.
\newblock {\em Compos. Math.}, 154(2):380--409, 2018.

\bibitem[DDNL21a]{DDNL21LOG}
T.~Darvas, E.~Di~Nezza, and C.~H. Lu.
\newblock Log-concavity of volume and complex {M}onge-{A}mp\`ere equations with prescribed singularity.
\newblock {\em Math. Ann.}, 379(1-2):95--132, 2021.

\bibitem[DDNL21b]{DDNL21}
T.~Darvas, E.~Di~Nezza, and C.~H. Lu.
\newblock The metric geometry of singularity types.
\newblock {\em J. Reine Angew. Math.}, 771:137--170, 2021.

\bibitem[DDNL25]{darvas2020relative}
T.~Darvas, E.~Di~Nezza, and C.~H. Lu.
\newblock Relative pluripotential theory on compact {K{\"a}hler} manifolds.
\newblock {\em Pure Appl. Math. Q.}, 21(3):1037--1118, 2025.

\bibitem[DH12]{DH_con}
S.~Dinew and P.~H. Hiep.
\newblock Convergence in capacity on compact {K{\"a}hler} manifolds.
\newblock {\em Ann. Sc. Norm. Super. Pisa, Cl. Sci. (5)}, 11(4):903--919, 2012.

\bibitem[Din19]{Dinew2019}
S.~Dinew.
\newblock {\em Lectures on Pluripotential Theory on Compact Hermitian Manifolds}, pages 1--56.
\newblock Springer International Publishing, Cham, 2019.

\bibitem[DK12]{DK12}
S.~Dinew and S.~Ko{\l}odziej.
\newblock Pluripotential estimates on compact {Hermitian} manifolds.
\newblock In {\em Advances in geometric analysis, Adv. Lect. Math. (ALM)}, pages 69--86. Somerville, MA: International Press; Beijing: Higher Education Press, 2012.

\bibitem[DNT24]{DT23}
E.~Di~Nezza and S.~Trapani.
\newblock The regularity of envelopes.
\newblock {\em Ann. Sci. {\'E}c. Norm. Sup{\'e}r. (4)}, 57(5):1347--1370, 2024.

\bibitem[DP04]{DP04}
J.-P. Demailly and M.~Paun.
\newblock Numerical characterization of the {K{\"a}hler} cone of a compact {K{\"a}hler} manifold.
\newblock {\em Ann. Math. (2)}, 159(3):1247--1274, 2004.

\bibitem[DV25a]{DV_qua1}
H.-S. Do and D.-V. Vu.
\newblock Quantitative stability for complex {Monge}-{Amp{\`e}re} equations. {I}.
\newblock {\em Anal. PDE}, 18(5):1271--1308, 2025.

\bibitem[DV25b]{DV_qua2}
H.-S. Do and D.-V. Vu.
\newblock Quantitative stability for the complex {Monge}-{Amp{\`e}re} equations. {II}.
\newblock {\em Calc. Var. Partial Differ. Equ.}, 64(8):269, 2025.

\bibitem[DX24]{DX21}
T.~Darvas and M.~Xia.
\newblock The volume of pseudoeffective line bundles and partial equilibrium.
\newblock {\em Geom. Topol.}, 28(4):1957--1993, 2024.

\bibitem[EGZ09]{EGZ09}
P.~Eyssidieux, V.~Guedj, and A.~Zeriahi.
\newblock Singular {K{\"a}hler}-{Einstein} metrics.
\newblock {\em J. Amer. Math. Soc.}, 22(3):607--639, 2009.

\bibitem[GL10]{GL10}
B.~Guan and Q.~Li.
\newblock Complex {Monge}-{Amp{\`e}re} equations and totally real submanifolds.
\newblock {\em Adv. Math.}, 225(3):1185--1223, 2010.

\bibitem[GL22]{GL2}
V.~Guedj and C.~H. Lu.
\newblock Quasi-plurisubharmonic envelopes 2: {Bounds} on {Monge}-{Amp{\`e}re} volumes.
\newblock {\em Algebr. Geom.}, 9(6):688--713, 2022.

\bibitem[GL23]{GL3}
V.~Guedj and C.~H. Lu.
\newblock Quasi-plurisubharmonic envelopes 3: solving {Monge}-{Amp{\`e}re} equations on {Hermitian} manifolds.
\newblock {\em J. Reine Angew. Math.}, 800:259--298, 2023.

\bibitem[GL25]{GL1}
V.~Guedj and C.~H. Lu.
\newblock Quasi-plurisubharmonic envelopes 1: {Uniform} estimates on {K{\"a}hler} manifolds.
\newblock {\em J. Eur. Math. Soc. (JEMS)}, 27(3):1185--1208, 2025.

\bibitem[GLZ19]{GLZ17}
V.~Guedj, C.~H. Lu, and A.~Zeriahi.
\newblock Plurisubharmonic envelopes and supersolutions.
\newblock {\em J. Differ. Geom.}, 113(2):273--313, 2019.

\bibitem[GP24]{GP22}
B.~Guo and D.~H. Phong.
\newblock On {{\(L^\infty\)}} estimates for fully non-linear partial differential equations.
\newblock {\em Ann. Math. (2)}, 200(1):365--398, 2024.

\bibitem[GT23]{GTru_quasimo}
V.~Guedj and A.~Trusiani.
\newblock Quasi-monotone convergence of plurisubharmonic functions.
\newblock {\em Bull. Sci. Math.}, 188:103341, 2023.

\bibitem[GZ05]{GZ07}
V.~Guedj and A.~Zeriahi.
\newblock Intrinsic capacities on compact {K{\"a}hler} manifolds.
\newblock {\em J. Geom. Anal.}, 15(4):607--639, 2005.

\bibitem[GZ07]{GZ05}
V.~Guedj and A.~Zeriahi.
\newblock The weighted {Monge}-{Amp{\`e}re} energy of quasiplurisubharmonic functions.
\newblock {\em J. Funct. Anal.}, 250(2):442--482, 2007.

\bibitem[GZ17]{GZ}
V.~Guedj and A.~Zeriahi.
\newblock {\em {D}egenerate complex {M}onge-{A}mp\`ere equations}, volume~26 of {\em EMS Tracts in Mathematics}.
\newblock European Mathematical Society (EMS), Z\"urich, 2017.

\bibitem[Han96]{HANANI96}
A.~Hanani.
\newblock Equations du type de monge–ampère sur les variétés hermitiennes compactes.
\newblock {\em Journal of Functional Analysis}, 137(1):49--75, 1996.

\bibitem[KN15]{KN15}
S.~Ko{\l}odziej and N.~C. Nguyen.
\newblock Weak solutions to the complex {Monge}-{Amp{\`e}re} equation on {Hermitian} manifolds.
\newblock In {\em Analysis, complex geometry, and mathematical physics}, pages 141--158. Providence, RI: American Mathematical Society (AMS), 2015.

\bibitem[KN19]{KN19}
S.~Ko{\l}odziej and N.~C. Nguyen.
\newblock Stability and regularity of solutions of the {Monge}-{Amp{\`e}re} equation on {Hermitian} manifolds.
\newblock {\em Adv. Math.}, 346:264--304, 2019.

\bibitem[KN21]{KN_22domin}
S.~Ko{\l}odziej and N.~C. Nguyen.
\newblock Continuous solutions to {Monge}-{Amp{\`e}re} equations on {Hermitian} manifolds for measures dominated by capacity.
\newblock {\em Calc. Var. Partial Differ. Equ.}, 60(3):93, 2021.

\bibitem[KN22]{KN22_weak}
S.~Ko{\l}odziej and N.~C. Nguyen.
\newblock Weak convergence of {M}onge-{A}mp{\`e}re measures on compact {H}ermitian manifolds.
\newblock {\em The Conference on Complex Geometric Analysis}, pages 113--123, 2022.

\bibitem[Ko{\l}98]{Kol98}
S.~Ko{\l}odziej.
\newblock The complex {Monge}-{Amp{\`e}re} equation.
\newblock {\em Acta Math.}, 180(1):69--117, 1998.

\bibitem[Ko{\l}03]{KOL03}
S.~Ko{\l}odziej.
\newblock The {Monge}-{Amp{\`e}re} equation on compact {K{\"a}hler} manifolds.
\newblock {\em Indiana Univ. Math. J.}, 52(3):667--686, 2003.

\bibitem[LN22]{LN19}
C.~H. Lu and V.-D. Nguy{\^e}n.
\newblock Complex {Hessian} equations with prescribed singularity on compact {K{\"a}hler} manifolds.
\newblock {\em Ann. Sc. Norm. Super. Pisa, Cl. Sci. (5)}, 23(1):425--462, 2022.

\bibitem[LPT21]{LPT}
C.~H. Lu, T.-T. Phung, and T.-D. T{\^o}.
\newblock Stability and {H{\"o}lder} regularity of solutions to complex {Monge}-{Amp{\`e}re} equations on compact {Hermitian} manifolds.
\newblock {\em Ann. Inst. Fourier}, 71(5):2019--2045, 2021.

\bibitem[Ngu16]{Ngu16}
N.~C. Nguyen.
\newblock The complex {Monge}-{Amp{\`e}re} type equation on compact {Hermitian} manifolds and applications.
\newblock {\em Adv. Math.}, 286:240--285, 2016.

\bibitem[Rai69]{Radon}
J.~Rainwater.
\newblock A note on the preceding paper.
\newblock {\em Duke Math. J.}, 36:799--800, 1969.

\bibitem[RWN14]{RWN14}
J.~Ross and D.~Witt~Nystr\"om.
\newblock Analytic test configurations and geodesic rays.
\newblock {\em J. Symplectic Geom.}, 12(1):125--169, 2014.

\bibitem[TW10]{TW10}
V.~Tosatti and B.~Weinkove.
\newblock The complex {Monge}-{Amp{\`e}re} equation on compact {Hermitian} manifolds.
\newblock {\em J. Amer. Math. Soc.}, 23(4):1187--1195, 2010.

\bibitem[Vu19]{Vu_loc}
D.-V. Vu.
\newblock Locally pluripolar sets are pluripolar.
\newblock {\em Int. J. Math.}, 30(13):1950029, 2019.

\bibitem[WN19]{WNmono}
D.~Witt~Nystr\"om.
\newblock Monotonicity of non-pluripolar {M}onge-{A}mp\`ere masses.
\newblock {\em Indiana Univ. Math. J.}, 68(2):579--591, 2019.

\bibitem[Xin09]{Xing}
Y.~Xing.
\newblock Continuity of the complex {Monge}-{Amp{\`e}re} operator on compact {K{\"a}hler} manifolds.
\newblock {\em Math. Z.}, 263(2):331--344, 2009.

\bibitem[Yau78]{YAU78}
S.-T. Yau.
\newblock On the {Ricci} curvature of a compact {K{\"a}hler} manifold and the complex {Monge}-{Amp{\`e}re} equation. {I}.
\newblock {\em Commun. Pure Appl. Math.}, 31:339--411, 1978.

\end{thebibliography}
\end{sloppypar}
\end{document}